\documentclass[pagebackref,colorlinks,citecolor=blue,linkcolor=blue,urlcolor=blue,filecolor=blue]{article}
\pdfoutput=1
\usepackage[margin=1in]{geometry}
\usepackage{tikz-cd}
\usetikzlibrary{calc}
\usepackage{float}
\usepackage{amsmath}
\usepackage{amsfonts}
\usepackage{amsthm}
\usepackage{enumitem}
\usepackage{dsfont}
\usepackage{amssymb}
\usepackage{pifont}
\usepackage{mathtools}
\usepackage{comment}
\usepackage{graphicx, caption} 
\usepackage{float}
\usetikzlibrary{matrix}

\usepackage{pinlabel} 

\usepackage{hyperref}

\newtheorem{thm}{Theorem}[section]
\newtheorem{lem}[thm]{Lemma}
\newtheorem{prop}[thm]{Proposition}

\newtheorem{exam}[thm]{Example}
\theoremstyle{definition}

\theoremstyle{definition}

\newtheorem{defn}[thm]{Definition}

\newtheorem{remark}[thm]{Remark}

\newtheorem*{claim*}{Claim}
\newtheorem*{quest*}{Question}
\newtheorem*{remark*}{Remark}
\newtheorem*{fact*}{Fact}

\newcommand{\Z}{\ensuremath{\mathbb{Z}}}
\newcommand{\Q}{\ensuremath{\mathbb{Q}}}
\newcommand{\R}{\ensuremath{\mathbb{R}}}

\title{Representation Stability for Disks in a Strip}
\author{Nicholas Wawrykow}
\date{}

\begin{document}
\maketitle
\begin{abstract}
We consider the ordered configuration space of $n$ open unit-diameter disks in the infinite strip of width $w$.
In the spirit of Arnol'd and Cohen, we provide a finite presentation for the rational homology groups of this ordered configuration space as a twisted algebra.
We use this presentation to prove that the ordered configuration space of open unit-diameter disks in the infinite strip of width $w$ exhibits a notion of first-order representation stability similar to Church--Ellenberg--Farb and Miller--Wilson's first-order representation stability for the ordered configuration space of points in a manifold.
In addition, we prove that for large $w$ this disk configuration space exhibits notions of second- (and higher) order representation stability.
\end{abstract}

\section{Introduction}
Given a manifold $X$, the \emph{ordered configuration space of $n$ points in $X$}, denoted $F_{n}(X)$, is the space of ways of embedding $n$ labeled points in $X$, where the topology is the subspace topology of $X^{n}$:
\[
F_{n}(X):=\big\{(x_{1}, \dots, x_{n})\in X^{n} | x_{i}\neq x_{j}\big\}.
\]
The \emph{unordered configuration space of $n$ points in $X$}, denoted $C_{n}(X)$, is the quotient of $F_{n}(X)$ by the symmetric group action on the labels of the points:
\[
C_{n}(X):=F_{n}(X)/S_{n}.
\]
Stability patterns among the homology groups of configuration spaces of points have been studied for almost half a century. 
If $X$ is a connected non-compact finite type manifold of dimension $d\ge2$, McDuff \cite{mcduff1975configuration} and Segal \cite{segal1979topology} proved that the homology groups of the unordered configuration space of points in $X$ stabilize: if $n\ge 2k+2$, there is an isomorphism $H_{k}\big(C_{n}(X)\big)\xrightarrow{\sim}H_{k}\big(C_{n+1}(X)\big)$. 
This is \emph{homological stability} for unordered configuration spaces of points.
Ordered configuration spaces of points do not exhibit homological stability as the rank of homology is polynomial in the number of points in the configuration space.
However, Church--Ellenberg--Farb proved that if $X$ is orientable, the homology groups of the ordered configuration space of points in $X$ stabilize in a representation-theoretic sense. 
In particular, they proved that for fixed $k$, the sequence $H_{k}\big(F_{\bullet}(X)\big)$ has an algebraic structure called an FI\#-module structure, and this structure is finitely generated in degree at most $2k$ \cite{church2015fi}. 
This \emph{first-order representation stability} for ordered configuration spaces of points, which was extended to non-orientable $X$ by Miller--Wilson \cite{miller2019higher}, implies that if $n>2k$, the decomposition of $H_{k}\big(F_{n}(X);\Q\big)$ into irreducible symmetric group representations is entirely determined by the decompositions of the $H_{k}\big(F_{m}(X);\Q\big)$ into irreducible symmetric group representations for all $m\le 2k$. 
The stabilization maps $H_{k}\big(C_{n}(X)\big)\xrightarrow{\sim}H_{k}\big(C_{n+1}(X)\big)$ and $H_{k}\big(F_{n}(X)\big)\to H_{k}\big(F_{n+1}(X)\big)$ arise from maps between configuration spaces $C_{n}(X)\to C_{n+1}(X)$ and $F_{n}(X)\to F_{n+1}(X)$ that ``add a point at infinity"; for a more precise definition, see, for example, \cite[Section 1.1]{miller2019higher}. 
In the case of homological stability, the stabilization map takes a class in $H_{k}\big(C_{n}(X)\big)$ and tensors it with the fundamental class of $H_{0}\big(C_{1}(\R^{d})\big)$ to get a class in $H_{k}\big(C_{n+1}(X)\big)$; see Figure \ref{configofpointsinclusionhom}.
The representation stability map is defined similarly, up to the relabeling of points.

\begin{figure}[h]
\centering
\captionsetup{width=.8\linewidth}
\includegraphics[width = 10cm]{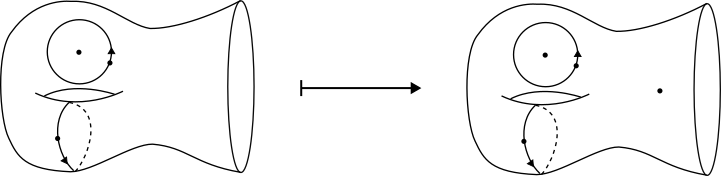}
\caption{The action of the stabilization map on a class in $H_{2}\big(C_{3}(T^{\circ})\big)$, where $T^{\circ}$ denotes the once-punctured torus. 
The point at the far right in the class in $H_{2}\big(C_{4}(T^{\circ})\big)$ arises from the fundamental class of $H_{0}\big(C_{1}(\R^{2})\big)$. 
It does not interact with the topology of $T^{\circ}$ or the other points, which correspond to an embedded non-bounding torus in configuration space.
}
\label{configofpointsinclusionhom}
\end{figure}

Miller--Wilson also proved that there is a stability pattern among the unstable classes of first-order representation stability. 
Namely, they proved that for fixed $k$ the sequence $\mathcal{W}^{X}_{i}(\bullet):=H^{\text{FI\#}}_{0}\Big(H_{\frac{|\bullet|+i}{2}}\big(F_{\bullet}(X);\Q\big)\Big)$ of FI\#-homology generators has an algebraic structure, namely that of a finitely generated FIM$^{+}$-module \cite{miller2019higher}. 
In this \emph{second-order representation stability}, the stabilization map arises from the map on the homology of configuration space that tensors a class in $H_{k}\big(F_{n}(X)\big)$ with the fundamental class of $H_{1}\big(F_{2}(\R^{2})\big)$, which can be thought of as an orbiting pair of points at infinity, producing a class in $H_{k+1}\big(F_{n+2}(X)\big)$; see Figure \ref{configofpointsinclusiontwohom} and \cite[Section 1.1]{miller2019higher}. 
Their proof, like many other (representation) stability arguments, relies on a Quillen-type argument on a carefully chosen spectral sequence.
See \cite{miller2019higher, ho2020higher, wawrykow2022secondary} for more on higher-order representation stability for ordered configuration spaces of points.

\begin{figure}[h]
\centering
\captionsetup{width=.8\linewidth}
\includegraphics[width = 10cm]{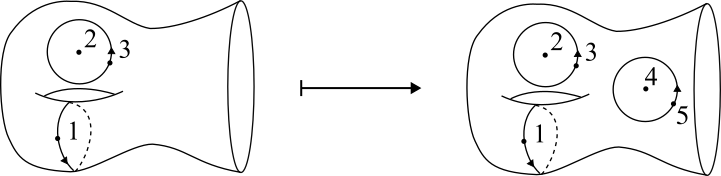}
\caption{The second-order stabilization map that sends a class in $H_{2}\big(F_{3}(T^{\circ})\big)$ to a class in $H_{3}\big(F_{5}(T^{\circ})\big)$.
The orbiting pair of points labeled $4$ and $5$ arise from the fundamental class of $H_{1}\big(F_{2}(\R^{2})\big)$.
}
\label{configofpointsinclusiontwohom}
\end{figure}

Recently, there has been increased interest in geometric generalizations of configuration spaces of points.
Perhaps the most studied of these generalizations are disk configuration spaces, where disks replace points, and the geometry of the underlying manifold $X$ comes to the forefront.
The structure of these disk configuration spaces, which had been studied in the contexts of motion planning in \cite{farber2008invitation}, physical chemistry in \cite{diaconis2009markov}, and physics in \cite{carlsson2012computational}, can be approached via Morse theory as done by Barynshnikov--Bubenik--Kahle in \cite{BBK}.
The most studied of these disk configuration spaces is \emph{$\text{conf}(n,w)$}, \emph{the ordered configuration space of $n$ open unit-diameter disks in the infinite strip of width $w$}:
\[
\text{conf}(n,w):=\Big\{(x_{1}, y_{1},\dots, x_{n}, y_{n})\in \R^{2n}|(x_{i}-x_{j})^{2}+(y_{i}-y_{j})^{2}\ge 1\text{ and }\frac{1}{2}\le y_{i}\le w-\frac{1}{2}\Big\},
\]
here $(x_{i}, y_{i})$ parametrizes the center of the $i^{\text{th}}$ disk.
Alpert--Kahle--MacPherson proved that the rank of $H_{k}\big(\text{conf}(n,w)\big)$ grows  exponentially in $n$ \cite{alpert2021configuration}; as a result, the ordered configuration space of open unit-diameter disks in the infinite strip of width $w$ is not first-order representation stable in the sense of Church--Ellenberg--Farb, Miller--Wilson.
However, in the case $w=2$, Alpert proved that the ordered configuration space of $n$ open unit-diameter disks in the infinite strip of width $2$ exhibits a reasonable notion of first-order representation stability, as for all $k$ the homology groups $H_{k}\big(\text{conf}(\bullet, 2)\big)$ have the structure of a finitely generated FI$_{k+1}$-module, generated in degree at most $3k$, where FI$_{k+1}$ is a generalization of FI\# \cite{alpert2020generalized}.
One can use Alpert's results to decompose $H_{k}\big(\text{conf}(n, 2);\Q\big)$ into a direct sum of irreducible $S_{n}$-representations as was done in \cite{wawrykow2022On}.
Alpert's proof relies on determining a basis for $H_{k}\big(\text{conf}(n, 2)\big)$, and using this basis to show that there are well-defined ways to insert the fundamental class of $H_{0}\big(\text{conf}(1,w)\big)$ into a class in $H_{k}\big(\text{conf}(n, 2)\big)$ to get a class in $H_{k}\big(\text{conf}(n+1, 2)\big)$.

Does Alpert's proof technique extend to larger widths $w$?
No.
Alpert--Manin found bases for all $H_{k}\big(\text{conf}(n, w)\big)$ analogous to Alpert's bases for $H_{k}\big(\text{conf}(n, 2)\big)$, and they proved that these homology groups are free abelian \cite[Theorem B]{alpert2021configuration1}.
They also proved that for $w\ge 3$, these bases are not well behaved under the symmetric group action \cite[Proposition 8.6]{alpert2021configuration1}, and that this precludes the possibility of there being enough well-defined ways of inserting the fundamental class of $H_{0}\big(\text{conf}(1,w)\big)$ into a homology class in $H_{k}\big(\text{conf}(n,w)\big)$ to get all the classes in $H_{0}\big(\text{conf}(n+1,w)\big)$, that is, it shows there are not enough ways to give $H_{k}\big(\text{conf}(\bullet,w)\big)$ the structure of a finitely generated FI$_{d}$ module.
In fact, their argument suggests that such problems are unavoidable when one considers integral homology.
In addition to computing bases for homology, Alpert--Manin proved that $H_{*}\big(\text{conf}(\bullet, w);\Z\big)$ has an algebraic structure, that of a finitely generated twisted (non-commutative) algebra \cite[Theorem A]{alpert2021configuration1}.

In this paper, we consider the rational homology of the ordered configuration space of open unit-diameter disks in the infinite strip of width $w$, and we find a new set of generators for $H_{*}\big(\text{conf}(\bullet, w);\Q\big)$ as a twisted algebra.
Through a series of inductive arguments we use these generators to construct a basis for $H_{k}\big(\text{conf}(n, w);\Q\big)$, distinct from Alpert--Manin's.
We exhibit families of relations among products of our generators by studying certain cellular complexes, and use our basis to prove that these relations and our generators give a finite presentation of $H_{*}\big(\text{conf}(\bullet, w);\Q\big)$ as a twisted algebra.

\begin{thm}\label{presentation lem intro}
The homology groups $H_{*}\big(\text{conf}(\bullet, w);\Q\big)$ have the structure of a finitely presented twisted algebra.
\end{thm}

The relations in the presentation given in the proof of Theorem \ref{presentation lem intro} show that our generators are well-behaved under the symmetric group action, and we use this to show that there is a family of well-defined disk insertion maps between rational homology classes.
Additionally, letting $w\to \infty$, our presentation recovers results of Arnol'd \cite{arnold1969cohomology} and Cohen \cite[Chapter 3]{cohen2007homology} on the structure of $H_{*}\big(F_{\bullet}(\R^{2});\Q\big)$ as a twisted algebra; for more, see, for example, \cite{sinha2006homology, knudsen2018configuration}.
We use this presentation to prove that the ordered configuration space of open unit-diameter disks in the infinite strip of width $w$ satisfies a notion of first-order representation stability.

\begin{thm}\label{first order stab intro}
For all $k\ge0$ and $w\ge 2$, the sequence $H_{k}\big(\text{conf}(\bullet, w);\Q\big)$ has the structure of a finitely generated FI$_{b+1}$-module over $\Q$, where $b=\big\lfloor\frac{k}{w-1}\big\rfloor$, finitely generated in degree at most $2k$ for $w\ge 3$, and degree at most $3k$ for $w=2$.
\end{thm}

In the spirit of \cite{church2014fi, church2015fi, ramos2017generalized}, one can use this FI$_{b+1}$-module structure and results of Ramos on finitely generated FI$_{b+1}$-modules to get upper bounds for the multiplicities of irreducible symmetric group representations in $H_{k}\big(\text{conf}(n, w);\Q\big)$ for large $n$.
Namely, for $w\ge3$, if for every $m\le 2k$ we have a direct sum decomposition of $H_{k}\big(\text{conf}(m, w);\Q\big)$ into irreducible $S_{m}$-representations, then for $n>2k$ we have upper bounds for the multiplicities of the irreducible $S_{n}$-representations that occur in a direct sum decomposition of $H_{k}\big(\text{conf}(n, w);\Q\big)$ into irreducible $S_{n}$-representations.
This in turn gives upper bounds on the Betti numbers of $H_{k}\big(\text{conf}(n, w);\Q\big)$ for $n> 2k$, if for every $m\le 2k$ we know the Betti numbers of $H_{k}\big(\text{conf}(m, w);\Q\big)$.

The FI$_{b+1}$-module structure of $H_{k}\big(\text{conf}(\bullet, w);\Q\big)$ arises from the twisted algebra structure of homology as follows:
Any class in $H_{k}\big(\text{conf}(n, w);\Q\big)$ can be written as a sum of products of generators called \emph{wheels} and \emph{averaged-filters}; see Section \ref{rat hom sec} for their definitions.
Given a decomposition of a class in $H_{k}\big(\text{conf}(n, w);\Q\big)$ as a product of these generators, at most $b=\big\lfloor\frac{k}{w-1}\big\rfloor$ of the generators occupy the entire width of the strip as the most ``efficient'' generators use $w$ disks and have homological degree $w-1$.
To send a class in $H_{k}\big(\text{conf}(n, w);\Q\big)$ to a class in $H_{k}\big(\text{conf}(n+1, w);\Q\big)$, one can insert a wheel on $1$ disk, that is, the fundamental class of $H_{0}\big(\text{conf}(1,w);\Q\big)$, immediately to the left of any generator that occupies the entire width of the strip, or at the far right.
The relations given in our proof of Theorem \ref{presentation lem intro} ensure that these wheel insertion operations are well defined, see Figure \ref{firstorderstabmapdisks}.
The finite generation in degree at most $2k$ for $w\ge 3$ (or $3k$ for $w=2$) is the fact that for $n>2k$, after an action of $S_{n}$, every class in $H_{k}\big(\text{conf}(n, w);\Q\big)$ can be written as sum of classes in $H_{k}\big(\text{conf}(n, w);\Q\big)$, each of which can be thought of as a class in $H_{k}\big(\text{conf}(n-1, w);\Q\big)$ with a new disk inserted in one of the $b+1$ possible places.

\begin{figure}[h]
\centering
\captionsetup{width=.8\linewidth}
\includegraphics[width = 12cm]{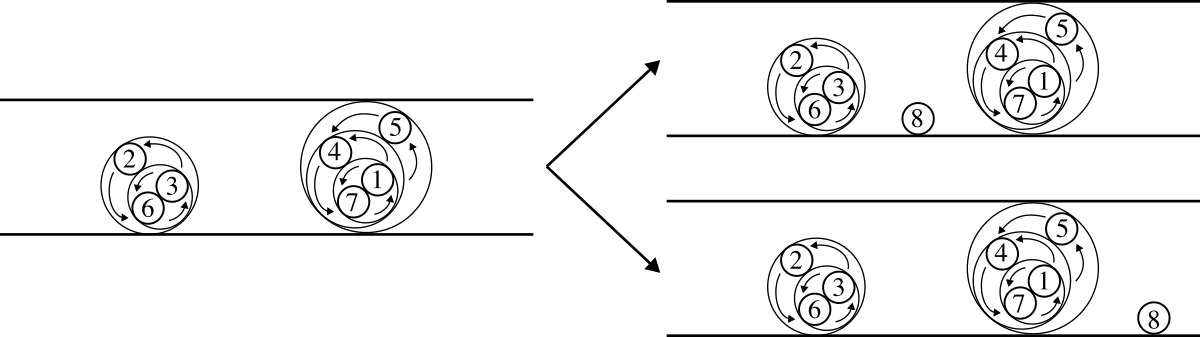}
\caption{On the left: A class in $H_{5}\big(\text{conf}(7, 4);\Q\big)$ that is the product of $2$ wheels, i.e., homology classes arising from disks orbiting each other, corresponding to embedded tori in the configuration space.
On the right: Two classes in $H_{5}\big(\text{conf}(8, 4);\Q\big)$ arising the FI$_{2}$-module structure on $H_{5}\big(\text{conf}(\bullet, 4);\Q\big)$.
The wheel on the disks $1$, $4$, $5$, and $7$ occupies the entire width of the strip; therefore, we can place a new disk labeled $8$ representing the fundamental class of $H_{0}\big(\text{conf}(1,4);\Q\big)$ immediately to the left of this wheel (the top image) or at the far right (the bottom image). 
Placing the disk labeled $8$ to the left of the wheel consisting of the disks $2$, $3$, and $6$ is homologous to placing it immediately to the left of the larger wheel, that is, in between the two wheels as in the top image, as there is space above the wheel on $3$ disks to slide the disk labeled $8$ across it.
}
\label{firstorderstabmapdisks}
\end{figure}

Miller--Wilson's proof of second-order representation stability for ordered configuration spaces of points relies on the Noetherianity of a twisted skew-commutative algebra.
Since the twisted algebras corresponding to third and higher order representation stability are not known to be Noetherian, it is not possible to extend Miller--Wilson's methods to the higher order cases at the moment.
We run into a similar problem when trying to use such a method to prove higher order representation stability for the ordered configuration space of open-unit diameter disks in the infinite strip of width $w$.
As such, we must get our hands dirty studying the homology groups of $\text{conf}(n,w)$.
Fortunately, a large part of this work is already done in the proof of Theorem \ref{presentation lem intro}.
An inductive argument determining the structure of the quotient of the twisted algebra $H_{*}\big(\text{conf}(\bullet, w);\Q\big)$ by the twisted commutative algebras corresponding to higher order representation stability allows us to prove that the ordered configuration space of open unit-diameter disks in the infinite strip of width $w$ satisfies notions of $1^{\text{st}}$- through $\big\lfloor\frac{w+1}{3}\big\rfloor^{\text{th}}$-order representation stability.

\begin{thm}\label{higher order stability intro}
Let
\[
\mathcal{W}^{m}_{i}(\bullet):=H^{FIW(m-1)_{\infty}}_{0}\Bigg(\cdots \bigg(H^{FI_{\infty}}_{0}\Big(H_{\frac{(m-1)|\bullet|+i}{m}}\big(\text{conf}(\bullet, w);\Q\big)\Big)\bigg)\cdots\Bigg),
\]
where FIW$(j)_{*}$ is an appropriate category to encode $j^{\text{th}}$-order representation stability. Then, for $1\le m\le\big\lfloor\frac{w+1}{3}\big\rfloor$, we have that $\mathcal{W}^{m}_{i}(\bullet)$ is a finitely generated FIW$(m)_{b+1}$-module, where $b=\big\lfloor\frac{mi}{w-m}\big\rfloor$. Moreover, this FIW$(m)_{b+1}$-module is generated in degree $\le(m+1)i$ for $w\ge 2m+1$, and degree $\le(m+1)i+m$ for $w=2m$.
\end{thm}

The FIW$(m)_{b+1}$-module structure of $\mathcal{W}^{m}_{i}(\bullet)$ arises from the twisted algebra structure of homology as follows:
Any class in $\mathcal{W}^{m}_{i}(\bullet)$ can be written as a sum of products of averaged-filters and wheels on at least $m$ disks.
To send a class in $\mathcal{W}^{m}_{i}(n)$ to a class in $\mathcal{W}^{m}_{i}(n+m)$, one can insert a wheel on $m$ disks, that is, a set of $m$ disks orbiting each other that corresponds to an embedded $(m-1)$-torus in configuration space, immediately to the left of any averaged-filter or any wheel on at least $w+1-m$ disks, or at the far right.
In a class of $\mathcal{W}^{m}_{i}(n)$ that can be written as a single product of generators, at most $b=\big\lfloor\frac{mi}{w-m}\big\rfloor$ of these generators can be averaged-filters or wheels on at least $w+1-m$ disks.
The relations given in our proof of Lemma \ref{presentation lem intro} ensure that this wheel insertion operation is well defined, see Figure \ref{firstorderstabmapdisks2}.

\begin{figure}[h]
\centering
\captionsetup{width=.8\linewidth}
\includegraphics[width = 12cm]{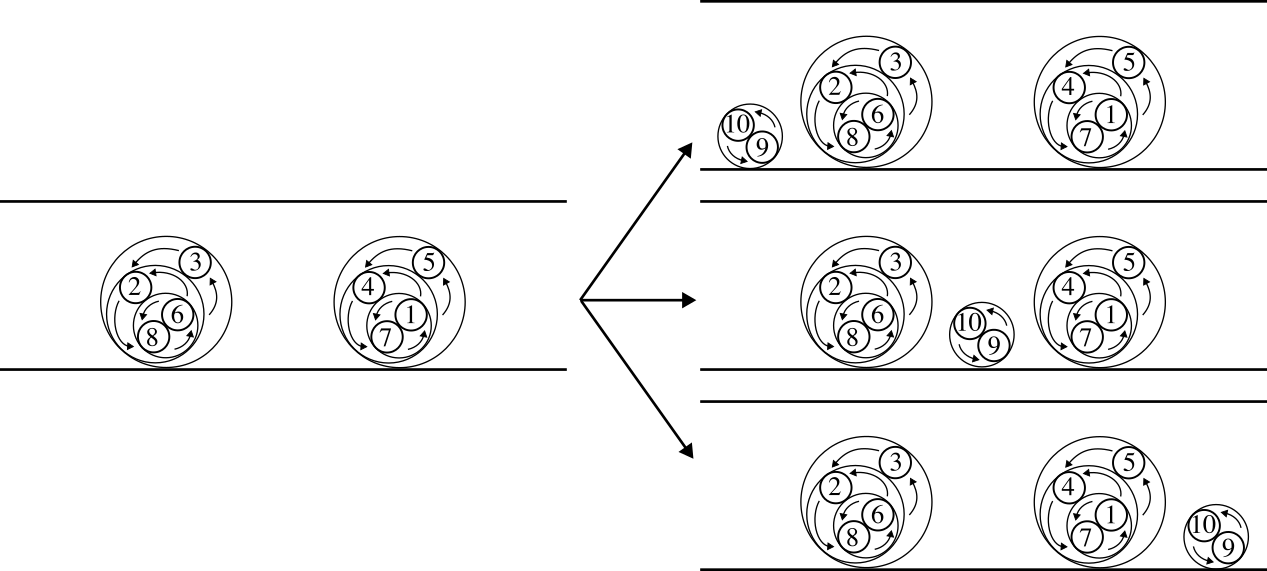}
\caption{On the left: A class in $H_{6}\big(\text{conf}(8, 5);\Q\big)$ that has non-trivial quotient in $\mathcal{W}^{2}_{3}(8)$.
On the right: Three classes in $H_{7}\big(\text{conf}(10, 5);\Q\big)$ arising from the FIW$(2)_{3}$-module structure on $\mathcal{W}^{2}_{3}(\bullet)$.
Both the wheel on the disks $2$, $3$, $6$, and $8$ and the wheel on the disks $1$, $4$, $5$, and $7$ occupy enough of the strip to prevent one from moving a wheel on 2 disks past either.
Therefore, we can place a new wheel on the disks labeled $9$ and $10$ immediately to the left of the wheel on the disks $2$, $3$, $6$, and $8$ (the top image), in between the two wheels (the middle image), or to at the far right (the bottom image).
}
\label{firstorderstabmapdisks2}
\end{figure}

Theorem \ref{higher order stability intro} is the first proof of third- or higher order representation stability for the homology of an ordered configuration space other than the ordered configuration space of points in the plane. See Figure \ref{stability diagram for disks}.

\subsection*{Outline}\label{outline}
In section \ref{cell complexes} we recall the cellular complexes $\text{cell}(n,w)$ and $P(n, w)$ and generalize them by defining the weighted cellular complexes $\text{cell}(A, \mathcal{W}, w)$ and $P(A, \mathcal{W}, w)$.
Then, we relate these cellular complexes to weighted disk configuration spaces.
Next, in section \ref{rat hom sec}, we define chain-maps that we will use to generate a family of classes in rational homology. 
In Lemma \ref{relationlem}, we prove the existence of a family of relations in the rational homology of $\text{conf}(n,w)$. 
We recall Alpert--Manin's basis for $H_{k}\big(\text{conf}(n,w);\Q\big)$, and through a series of propositions relying on Lemma \ref{relationlem}, we determine a new basis for rational homology in Theorem \ref{AMWthmB''}.
In section \ref{rep stab sec}, we recall the definition of a twisted (commutative) algebra, and we use the basis of Theorem \ref{AMWthmB''} to provide a finite presentation for $H_{*}\big(\text{conf}(\bullet,w);\Q\big)$ as a twisted algebra in Theorem \ref{finitepresentation}.
Finally, we use the presentation given in the proof of Theorem \ref{finitepresentation} to prove Theorems \ref{first order stability} and \ref{higherorderstability} showing that the ordered configuration space of unit-diameter disks in the infinite strip of width $w$ exhibits notions of $1^{\text{st}}$- through $\big\lfloor\frac{w+1}{3}\big\rfloor^{\text{th}}$-order representation stability.

\begin{figure}[H]
\centering
\captionsetup{width=.8\linewidth}
\includegraphics[width = \textwidth]{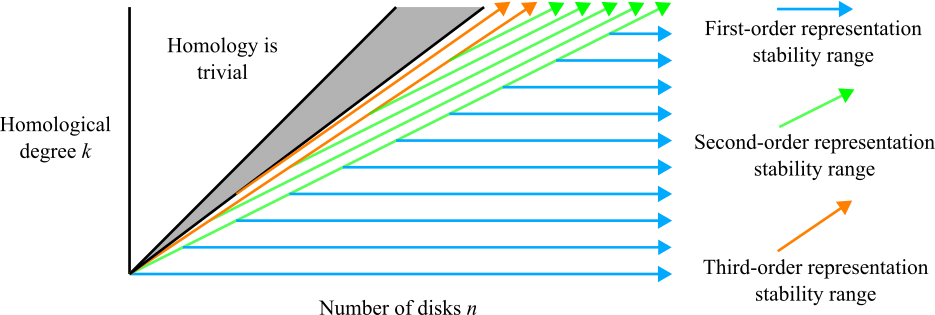}
\caption{Representation stability for $H_{k}\big(\text{conf}(n, 8);\Q\big)$.}
\label{stability diagram for disks}
\end{figure}

\subsection*{Acknowledgements}
The author would like to thank Hannah Alpert, Nir Gadish, Fedor Manin, Andrew Snowden, and Jennifer Wilson for helpful conversations and comments. 
The author would also like to thank the anonymous referee for pointing out an error in the previous version of Theorem \ref{higher order stability intro} and their helpful comments on the organization of this paper.

\section{$\text{conf}(n,w)$ and $\text{cell}(n,w)$}\label{cell complexes}
In this section we recall the definitions of the cellular complexes $\text{cell}(n)$ and $P(n)$ and some of their generalizations. 
As we will see, these complexes greatly simplify the study of the topology of $\text{conf}(n,w)$ due to their combinatorial characteristics.
We use Alpert--Manin's conventions for giving these complexes cellular chain structures and we recall the definitions of certain chain maps Alpert--Manin used in their paper.

\begin{defn}
The \emph{ordered configuration space of $n$ open unit-diameter disks in the infinite strip of width $w$}, denoted \emph{$\text{conf}(n,w)$}, is the space of ways of putting $n$ non-overlapping labeled open unit-diameter disks in the infinite strip of width $w$ such that the topology is the subspace topology of $\R^{2n}$:
\[
\text{conf}(n,w):=\Big\{(x_{1}, y_{1},\dots, x_{n}, y_{n})\in \R^{2n}|(x_{i}-x_{j})^{2}+(y_{i}-y_{j})^{2}\ge 1\text{ and }\frac{1}{2}\le y_{i}\le w-\frac{1}{2}\Big\},
\]
Here $(x_{i}, y_{i})$ are the coordinates of the center of the $i^{\text{th}}$ disk. See Figure \ref{point in conf(10,3)} for an example.
\end{defn}

\begin{figure}[h]
\centering
\captionsetup{width=.8\linewidth}
\includegraphics[width = 12cm]{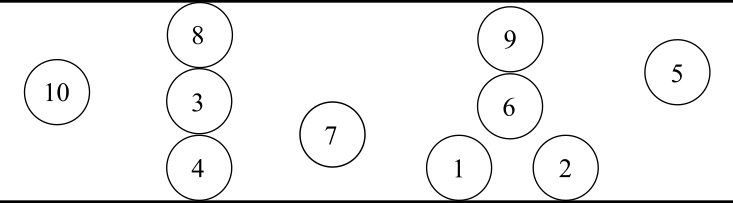}
\caption{A visualization of a point in $\text{conf}(10,3)$.}
\label{point in conf(10,3)}
\end{figure}

Directly studying the topology of these configuration spaces is a daunting task; fortunately, Alpert--Kahle--MacPherson proved that $\text{conf}(n,w)$ is homotopy equivalent to a polyhedral cellular complex introduced by Blagojevi\`{c}--Ziegler in \cite{blagojevic2014convex} called $\text{cell}(n,w)$ \cite[Theorem 3.1]{alpert2021configuration}.

\begin{defn}
Given two disjoint ordered sets $A=\{a_{1},\dots, a_{n}\}$ and $B=\{b_{1}, \dots, b_{m}\}$, a \emph{shuffle} of the ordered set $\{a_{1}, \dots, a_{n}, b_{1}, \dots, b_{m}\}$ is a permutation of the ordered set $\{a_{1}, \dots, a_{n}, b_{1}, \dots, b_{m}\}$ preserving the order of the elements of $A$ and the order of the elements of $B$.
\end{defn}

\begin{defn}
The cellular complex \emph{$\text{cell}(n)$} has cells represented by \emph{symbols} consisting of an ordering of the numbers $1,\dots, n$ separated by vertical bars into \emph{blocks} such that no block is empty.
A cell $f\in \text{cell}(n)$ is a top dimensional \emph{face} of a cell $g\in \text{cell}(n)$ if $g$ can be obtained by deleting a bar in $f$ and shuffling the resulting block.
\end{defn}

By restricting how big a block can be, one gets the cellular complexes $\text{cell}(n,w)$. 

\begin{defn}
The cell complex \emph{$\text{cell}(n, w)$} consists of the cells in $\text{cell}(n)$ represented by symbols whose blocks have at most $w$ elements. 
See Figure \ref{cell(3,2)} for an example.
\end{defn}

\begin{figure}[h]
\centering
\captionsetup{width=.8\linewidth}
\includegraphics[width = 8cm]{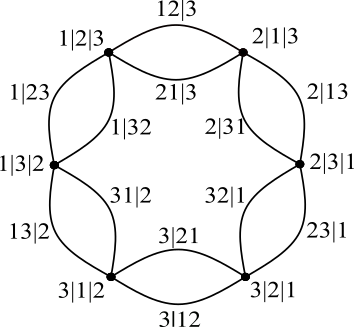}
\caption{The complex $\text{cell}(3,2)$.}
\label{cell(3,2)}
\end{figure}

\begin{exam}
The cell $2\,4|3\, 5\,1$ is in $\text{cell}(5)$ and $\text{cell}(5,3)$, but not in $\text{cell}(5,2)$.
\end{exam}

It will be useful to consider disk configuration spaces where the disks have different diameters. 
In the case of the infinite strip of width $w$ we can study such configuration spaces through the use of weighted sets.

\begin{defn}
A \emph{weighted set} $(A, \mathcal{W})$ is a pair of sets with a fixed bijection between elements $a\in A$ and elements $w_{a}\in\mathcal{W}$.
We say that $w_{a}$ is the \emph{weight} of $a$.
\end{defn}

We will only consider weighted sets where all weights are positive integers. 
Next, we define two topological objects that are homotopy equivalent.

\begin{defn}
Given a weighted set $(A,\mathcal{W})$, the weighted configuration space \emph{$\text{conf}(A,\mathcal{W}, w)$} is the space of ways of putting $|A|$ disks with distinct labels in $A$ such that the diameter of the disk labeled $a$ is $w_{a}\in \mathcal{W}$ in the infinite strip of width $w$ such that no two disks overlap. 
See Figure \ref{weightedconfig} for an example.
\end{defn}

\begin{figure}[h]
\centering
\captionsetup{width=.8\linewidth}
\includegraphics[width = 14cm]{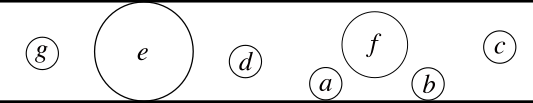}
\caption{A point in $\text{conf}\big(\{a, b, c, d, e, f, g\}, \{1, 1, 1, 1, 3, 2, 1\}, 3\big)$}
\label{weightedconfig}
\end{figure}

We generalize the cellular complexes $\text{cell}(n,w)$ through the use of weighted sets.

\begin{defn}
Given a weighted set $(A,\mathcal{W})$, the cellular complex \emph{$\text{cell}(A, \mathcal{W}, w)$} has cells represented by symbols consisting of an ordering of the elements of $A$ separated by vertical bars into blocks such that no block is empty, and such that sum of the weights of the elements in a block is at most $w$. 
See Figure \ref{weightedcell3} for an example.
\end{defn}

\begin{figure}[h]
\centering
\captionsetup{width=.8\linewidth}
\includegraphics[width = 6cm]{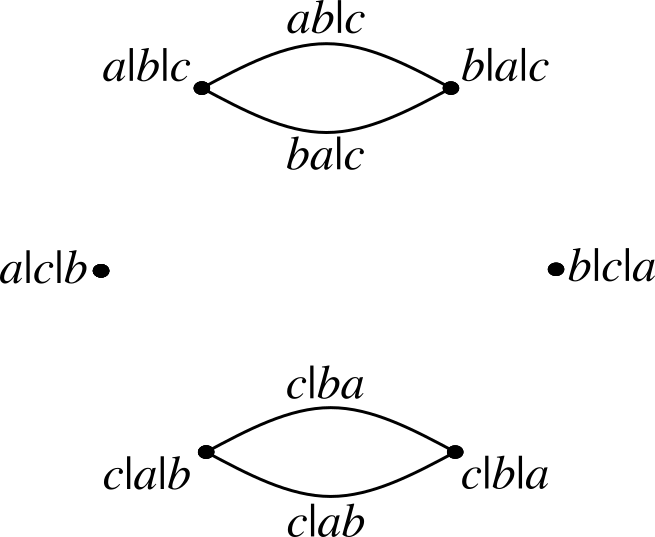}
\caption{The weighted cell complex $\text{cell}\big(\{a, b, c\}, \{1, 1, 3\}, 3\big)$. 
The two components with non-trivial first homology correspond to the subspaces of $\text{conf}\big(\{a, b, c\}, \{1, 1, 3\}, 3\big)$ where the disks $a$ and $b$ are on the same side of disk $c$.
The other two components correspond to the subspaces of $\text{conf}\big(\{a, b, c\}, \{1, 1, 3\}, 3\big)$ where the disks $a$ and $b$ are on different sides of disk $c$.
Since disk $c$ has weight (diameter) $3$, it fills the entire width of the strip and the other two disks cannot pass by it.}
\label{weightedcell3}
\end{figure}

Alpert--Kahle--MacPherson's proof that $\text{conf}(n,w)$ and $\text{cell}(n,w)$ are homotopy equivalent \cite[Theorem 3.1]{alpert2021configuration} generalizes to $\text{conf}(A, \mathcal{W}, w)$ and $\text{cell}(A, \mathcal{W}, w)$.

\begin{prop}\label{weightedconfiscell}
There is a homotopy equivalence $\text{conf}(A,\mathcal{W},w)\simeq \text{cell}(A,\mathcal{W},w)$. 
Moreover, the homotopy equivalences for $w$ and $w+1$ commute up to homotopy with the inclusions $\text{cell}(A,\mathcal{W},w)\hookrightarrow \text{cell}(A,\mathcal{W}, w+1)$ and $\text{conf}(A,\mathcal{W}, w)\hookrightarrow\text{conf}(A,\mathcal{W},w+1)$.
\end{prop}


A cell $a_{1, 1}\,\cdots\, a_{1, n_{1}}|\cdots|a_{m, 1}\,\cdots\, a_{m, n_{m}}$ in $\text{cell}(A,\mathcal{W},w)$ is homotopy equivalent to the following subspace of $\text{conf}(A,\mathcal{W},w)$:
A block $a_{i, 1}\,\cdots\, a_{i, n_{i}}$ corresponds to the subspace of the weighted configuration space $\text{conf}\big(\{a_{i, 1},\dots, a_{i, n_{i}}\}, \{w_{a_{i, 1}},\dots, w_{a_{i, n_{i}}}\},w\big)$ where the disk labeled $a_{i, n_{i}}$ (of diameter $w_{i, n_{i}}$) is never above all of the disks $a_{i, 1},\dots a_{i, n_{i}-1}$.
The disk labeled $a_{i, n_{i}-1}$ (of diameter $w_{i, n_{i}-1}$) can be anywhere in the strip as long as its bottom does not go below the bottom of $a_{i, n_{i}}$.
The other disks are similar, as the disk labeled $a_{i, j}$ (of diameter $w_{i,j}$) must be no lower than the bottom of the disk labeled $a_{i, j+1}$.
Taking the product of these subspaces in this order gives a subspace of $\text{conf}(A,\mathcal{W},w)$, as the disks in the $i^{\text{th}}$ block are always to the left of any of the disks in the $(i+1)^{\text{th}}$ block. 
See Figure \ref{celltoconf}.

\begin{figure}[h]
\centering
\captionsetup{width=.8\linewidth}
\includegraphics[width = 14cm]{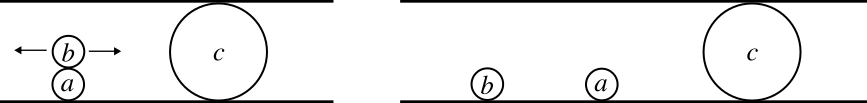}
\caption{On the left: The subspace of $\text{conf}\big(\{a, b, c\}, \{1,1,3\}, 3\big)$ corresponding to the cell $b\,a|c$ in $\text{cell}\big(\{a, b, c\}, \{1,1,3\}, 3\big)$.
On the right: The subspace of $\text{conf}\big(\{a, b, c\}, \{1,1,3\}, 3\big)$ corresponding to the cell $b|a|c$ in $\text{cell}\big(\{a, b, c\}, \{1,1,3\}, 3\big)$. 
This subspace forms part of the boundary of the subspace on the left. 
See the top right of Figure \ref{weightedcell3}.}
\label{celltoconf}
\end{figure}

If $S_{A}$ is the symmetric group on the set $A$ and $S_{A, \mathcal{W}}\subseteq S_{A}$ is the subgroup of $S_{A}$ that consists of permutations of elements with the same weight, then there is an action of $S_{A, \mathcal{W}}$ on $\text{conf}(A,\mathcal{W},w)$ and $\text{cell}(A,\mathcal{W},w)$, and the homotopy equivalence of Proposition \ref{weightedconfiscell} is equivariant with respect to this action.

From now on we use $\text{conf}(n,w)$ and $\text{cell}(n,w)$ interchangeably, and we will do the same for their weighted variants.

Our results depend on determining relations between classes in the rational homology of $\text{cell}(A, \mathcal{W}, w)$. 
We use Alpert--Manin's conventions to specify orientations on cells and signs for the boundary operator to determine a cellular chain structure on these cellular complexes. 
First, we recall the definition of the concatenation product of cells. 

\begin{defn}
Given symbols $f_{A}$ and $f_{B}$ on disjoint sets $A$ and $B$ the \emph{concatenation product} of $f_{A}$ and $f_{B}$ is $f_{A}|f_{B}$. 
This induces a map
\[
|:\text{cell}(A)\times \text{cell}(B)\to \text{cell}(A\sqcup B).
\]
\end{defn}

Given Proposition \ref{weightedconfiscell} the concatenation product of $f_{A}$ with $f_{B}$ corresponds to putting a configuration of disks with labels in $A$ to the left of a configuration of disks with labels in $B$.  

\begin{exam}
The cell $2\,4|3\,5\,1$ is a concatenation product of top dimensional cells of $\text{cell}\big(\{2,4\}\big)$ and $\text{cell}\big(\{1, 3, 5\}\big)$. 
See Figure \ref{concatenationproduct}.
\end{exam}

\begin{figure}[h]
\centering
\captionsetup{width=.8\linewidth}
\includegraphics[width = 14cm]{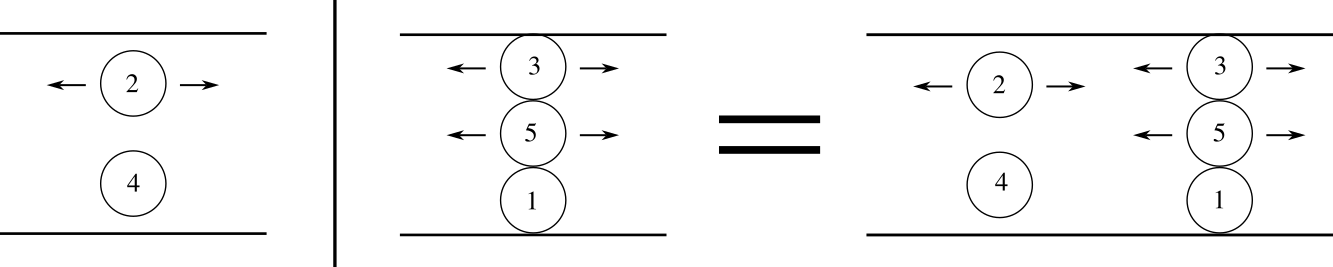}
\caption{The corresponding image in $\text{conf}(5,3)$ of the concatenation product of $2\,4$ with $3\,5\,1$.
 Note this is the product of a $1$-cell with a $2$-cell, and the resulting object is the $3$-cell $2\,4|3\,5\,1$.}
\label{concatenationproduct}
\end{figure}

Any cell in $\text{cell}(A)$ is either a top-dimensional cell or a concatenation product of cells of smaller dimension. 
The sign of a cell $f=e_{1}|e_{2}$ in the boundary of a top dimensional cell $g$ of $\text{cell}(A)$ is defined to be
\[
(-1)^{\text{length}(e_{1})}\cdot \text{sign}(\text{permutation }g\to e_{1}e_{2}),
\]
where $\text{length}(e_{1})$ is the number of elements in the block $e_{1}$, and this allows us to define a boundary map $\partial$ on the top dimensional cells. 
We extend $\partial$ to the product of cells $g_{1}|g_{2}$ via a Leibniz rule:
\[
\partial(g_{1}|g_{2}):=\partial g_{1}|g_{2}+(-1)^{\dim(g_{1})}g_{1}|\partial g_{2}.
\]

Moreover, the concatenation product defines an injective chain complex homomorphism on cellular chains,
\[
|:C_{*}\big(\text{cell}(A)\big)\otimes C_{*}\big(\text{cell}(B)\big)\to C_{*}\big(\text{cell}(A\sqcup B)\big),
\]
using the standard tensor product on chain complexes, where the differential is defined by
\[
\partial(a\otimes b):=\partial a\otimes b+(-1)^{\text{deg}(a)}a\otimes \partial b.
\]

While we could simply adopt these conventions for our weighted cellular complexes without any modification as Alpert--Manin do, we choose instead choose to define weighted versions of $\text{length}$, $\text{dim}$, and $\text{deg}$, which we denote $\text{wlength}$, $\text{wdim}$, and $\text{wdeg}$, respectively. 
This will allow us to more easily consider certain symmetric group actions on chains in these cellular complexes.
We do this on a block $e$ by setting $\text{wlength}(e)$ to be the sum of the weights of the elements in $e$, setting $\text{wdim}(e)$ to be $1$ less than the sum of the weights of elements in $e$, and the weighted degree $\text{wdeg}(e)$ to be $1$ less than the sum of the weights of elements in $e$.
Additionally, we define the weighted sign of a permutation of weighted elements, which we denote $\text{wsgn}$, by setting the weighted sign of a transposition of weighted elements $a$ and $b$ to be $(-1)^{w_{a}w_{b}}$ and extending this to other permutations.
These definitions give $\text{cell}(A,\mathcal{W})$ a chain complex structure that we will use throughout the paper.

Alpert--Manin proved that $\partial$ is a boundary operator \cite[Proposition 2.3]{alpert2021configuration1}, and 
their proof extends to the weighted complexes.

In their computation of a basis for $H_{k}\big(\text{cell}(n,w)\big)$, Alpert--Manin used simpler cellular complexes called (weighted) permutohedra. 
We will use weighted permutohedra and our weighted differential for several important calculations.

\begin{defn}
The \emph{permutohedron} $P(n)$ is the cellular complex where cells are represented by \emph{symbols} consisting of the set $\{1,\dots, n\}$, broken into nonempty \emph{blocks} where the ordering of the elements of a block does not matter.
A cell $f$ is a top dimensional face of a cell $g$ if $g$ can be obtained by deleting a bar in $f$.
\end{defn}

\begin{exam}
The permutohedron $P(3)$ has a single $2$-cell: $1\,2\,3$; six $1$-cells: $1|2\,3$, $2\,3|1$, $2|1\,3$, $1\,3|2$, $3|1\,2$, $1\,2|3$, each of which is a face of the $2$-cell; and six $0$-cells: $1|2|3$, $1|3|2$, $2|1|3$, $2|3|1$, $3|1|2$, and $3|2|1$. 
See Figure \ref{permutohedra}.
\end{exam}

\begin{defn}
Given a weighted set $(A, \mathcal{W})$, the \emph{weighted permutohedron} $P(A, \mathcal{W}, w)$ is the cellular complex such that cells are represented by symbols, i.e., the elements of $A$, broken into blocks such that the order of the elements in a block does not matter and the total weight of the elements in any block is at most $w$. 
A cell $f$ is a top dimensional face of a cell $g$ if $g$ can be obtained by deleting a bar in $f$. 
If there are no restrictions on the weight of a cell, we write $P(A)$. 
See Figure \ref{permutohedra} for an example.
\end{defn}

\begin{figure}[H]
\centering
\captionsetup{width=.8\linewidth}
\includegraphics[width = 12cm]{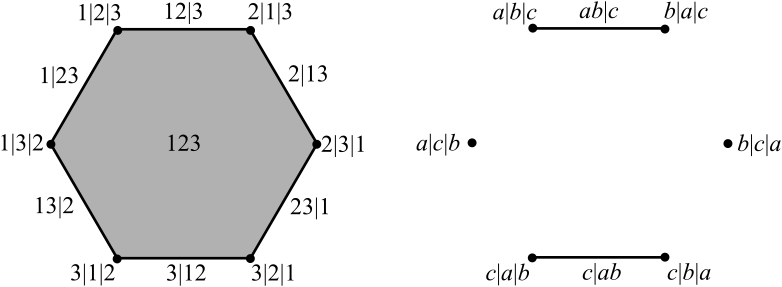}
\caption{On the left the permutodron $P(3)$. On the right, the weighted permutohedron $P\big(\{a, b, c\}, \{1, 1, 3\}, 3\big)$. Compare with Figure \ref{weightedcell3}.}
\label{permutohedra}
\end{figure}

From now on, we will order $A=\{a_{1},\dots, a_{n}\}$ such that $a_{1}<\cdots<a_{n}$ and represent the cells of $P(A)$ by symbols such that elements in each block are increasing order. 
This ordering induces an inclusion 
\[
i_{id}:P(A,\mathcal{W}, w)\hookrightarrow\text{cell}(A, \mathcal{W}, w),
\]
that sends a cell in $P(A,\mathcal{W}, w)$ to the cell in $\text{cell}(A, \mathcal{W}, w)$ with the same label.
By pulling back the boundary operator $\partial$ along $i_{id}$ we can study cellular chains in $P(A,\mathcal{W}, w)$, see Figure \ref{permutohedrainclusion}.

\begin{figure}[h]
\centering
\captionsetup{width=.8\linewidth}
\includegraphics[width = 12cm]{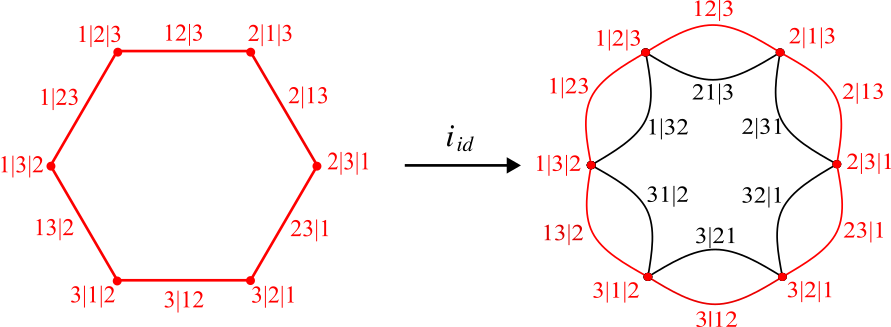}
\caption{The image (red outer loop) of the permutohedron $P\big(\{1,2,3\}, \{1,1,1\}, 2\big)$ in $\text{cell}\big(\{1,2,3\}, \{1,1,1\}, 2\big)$, under the map $i_{id}$.}
\label{permutohedrainclusion}
\end{figure}

Much like $\text{cell}(A, \mathcal{W}, w)$, one can interpret $P(A, \mathcal{W}, w)$ as a configuration space.

\begin{defn}
Given a weighted set $(A, \mathcal{W})$ the \emph{weighted no-$(w+1)$-equal space} $\text{no}_{w+1}(A, \mathcal{W})$ is the space of configurations of $|A|$ weighted points in $\R$ with weights in $\mathcal{W}$ such that no set of coincident points has total weight greater than $w$.
\end{defn}

Alpert--Manin proved that $P(A, \mathcal{W}, w)$ is homotopy equivalent to $\text{no}_{w+1}(A, \mathcal{W})$ \cite[Theorem 2.4]{alpert2021configuration1}, and used this homotopy equivalence to decompose the homology of the ordered configuration space of $n$ unit-diameter disks in the infinite strip of width $w$.

Next, we study the rational homology of $\text{conf}(n,w)$ by examining families of chain maps.

\section{The Rational Homology of $\text{conf}(n,w)$}\label{rat hom sec}
In this section we recall the definitions of Alpert--Manin's two types of homology generators: wheels and filters. 
We define a new type of rational homology class that we will call averaged-filters, and demonstrate a relation in rational homology between concatenation products of wheels and these averaged-filters. 
We use this relation and several others to give a basis for rational homology distinct from Alpert--Manin's basis.

\subsection{Spin-maps}
In order to define the homology classes called wheels, filters, and averaged-filters that we will use throughout the paper we need to define what we will call $\text{spin}$-maps between weighted cellular complexes. 
These are not the $\text{spin}$-maps of Section 4 of Alpert--Manin \cite[Section 4]{alpert2021configuration1}, which are maps from weighted permutohedra to weighted configuration spaces, rather they are maps from weighted configuration spaces to weighted configuration spaces that Alpert--Manin reference in passing.
First, we state a generalized version of a technical lemma of Alpert--Manin that will allow us to construct our $\text{spin}$-maps \cite[Lemma 4.6]{alpert2021configuration1}.
The proof of this lemma follows from their proof with the obvious generalizations.

\begin{lem}\label{AM technical lemma}
Let $(A, \mathcal{W})$ be a finite weighted set and $\big(\{a\}, \{w_{a}\}\big)$, $\big(\{b\}, \{w_{b}\}\big)$, and $\big(\{c\}, \{w_{c}\}\big)$ be weighted pairs not in $(A, \mathcal{W})$ such that $w_{a}=w_{b}+w_{c}$. 
Then, there is a map
\[
\iota:[0,1]\times P\big(A\cup \{a\}, \mathcal{W}\cup \{w_{a}\}\big)\to P\big(A\cup\{b,c\},\mathcal{W} \cup \{w_{b},w_{c}\}\big)
\]
with the following properties:
\begin{enumerate}
\item It is a homeomorphism and a cellular map. 
That is, it sends every $k$-face to a disk which is a union of $k$-faces.
\item It is equivariant with respect to the $\Z/ 2\Z$-actions given by $t\to 1-t$ on the domain $[0,1]$ and $b\leftrightarrow c$ on the codomain.
\item For each cell $f$ of $P\big(A\cup \{a\}, \mathcal{W}\cup \{w_{a}\}\big)$, the image of $(0,1)\times f$ under $\iota$ is the interior of the cell of $P\big(A\cup\{b,c\},\mathcal{W} \cup \{w_{b},w_{c}\}\big)$ with the same blocks except that $a$ is replaced in its block by $b\,c$.
\item It takes $\{0\}\times P\big(A\cup \{a\}, \mathcal{W}\cup \{w_{a}\}\big)$ to the union of those cells in which $b$ and $c$ are contained in separate blocks with $b$ preceding $c$. 
Similarly, it takes $\{1\}\times P\big(A\cup \{a\}, \mathcal{W}\cup \{w_{a}\}\big)$ to the union of those cells in which $c$ precedes $b$.
\end{enumerate}
See Figure \ref{iotamap}.
\end{lem}

\begin{figure}[h]
\centering
\captionsetup{width=.8\linewidth}
\includegraphics[width = 10cm]{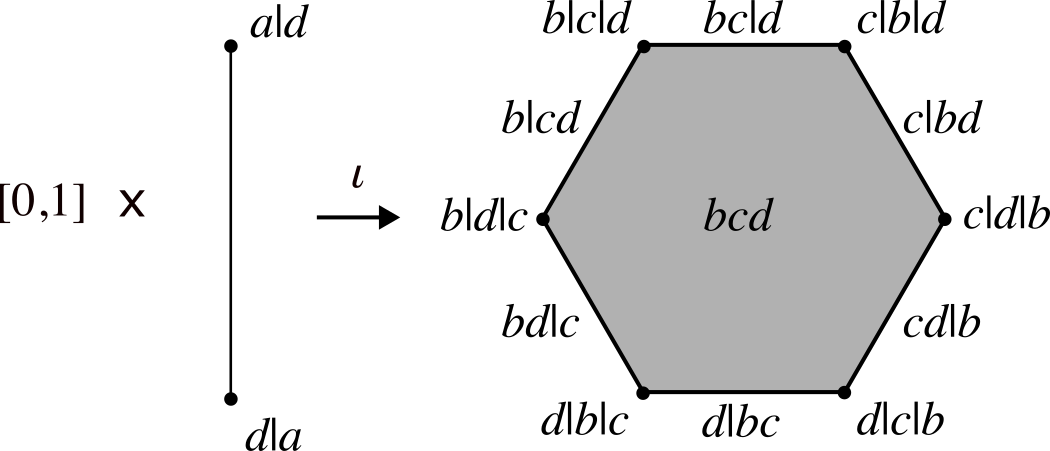}
\caption{The map $\iota:[0,1]\times P\big(\{a, d\}, \{2, 1\}\big)\to P\big(\{b,c, d\}, \{1,1, 1\}\big)$ that replaces $\big(\{a\}, \{2\}\big)$ with $\big(\{b,c\}, \{1,1\}\big)$. 
}
\label{iotamap}
\end{figure}
Recall that ordering $A\cup\{b,c\}$ and labeling the cells of $P\big(A\cup\{b,c\},\mathcal{W}\cup\{w_{b}, w_{c}\}\big)$ so that in each block of a cell the labels are in increasing order induces the map
\[
i_{id}:P\big(A\cup\{b,c\},\mathcal{W}\cup\{w_{b}, w_{c}\}\big)\hookrightarrow\text{cell}\big(A\cup\{b,c\},\mathcal{W}\cup\{w_{b}, w_{c}\}\big)
\]
that sends a cell in $P\big(A\cup\{b,c\},\mathcal{W}\cup\{w_{b}, w_{c}\}\big)$ to the cell in $\text{cell}\big(A\cup\{b,c\},\mathcal{W}\cup\{w_{b}, w_{c}\}\big)$ with the same label.
It follows that $i_{id}$ induces a chain map.

We also have the map
\[
i_{\tau}:P\big(A\cup\{b,c\},\mathcal{W}\cup\{w_{b}, w_{c}\}\big)\hookrightarrow\text{cell}\big(A\cup\{b,c\},\mathcal{W}\cup\{w_{b}, w_{c}\}\big)
\]
that sends a cell in $P\big(A\cup\{b,c\},\mathcal{W}\cup\{w_{b}, w_{c}\}\big)$ where $b$ and $c$ are in different blocks to the cell with the same label in $\text{cell}\big(A\cup\{b,c\},\mathcal{W}\cup\{w_{b}, w_{c}\}\big)$, and that sends a cell in $P\big(A\cup\{b,c\},\mathcal{W}\cup\{w_{b}, w_{c}\}\big)$ where $b$ and $c$ are in the same block to the similarly labeled cell in $\text{cell}\big(A\cup\{b,c\},\mathcal{W}\cup\{w_{b}, w_{c}\}\big)$ where every label but $b$ and $c$ is in the same place, but the order of $b$ and $c$ is switched.
Like $i_{id}$, the map $i_{\tau}$ induces a chain map.

In Lemma \ref{AM technical lemma} we made the choice to replace $a$ with $b\,c$; alternatively, we could have replace $a$ with $c\,b$.
These two choices coincide on the image of $\{0, 1\}\times P\Big((A, \mathcal{W})\cup \big(\{a\}, \{w_{a}\}\big)\Big)$ under $\iota$. 
Identifying the two subspaces gives a map we call $\text{spin}'$:
\[
\text{spin}':S^{1}\times P\big(A\cup \{a\}, \mathcal{W}\cup \{w_{a}\}\big)\to \text{cell}\big(A\cup\{b,c\},\mathcal{W}\cup\{w_{b}, w_{c}\}\big).
\]
Since top cells of $\text{cell}\big(A\cup \{a\}\big)$ correspond to orderings of $A\cup \{a\}$, this yields a map
\[
\text{spin}_{a:b,c}:\text{cell}\big(A\cup \{a\}, \mathcal{W}\cup \{w_{a}\}\big)\to \text{cell}\big(A\cup\{b,c\},\mathcal{W}\cup\{w_{b}, w_{c}\}\big),
\]
which, by restricting to blocks of weight at most $w$, gives a map
\[
\text{spin}_{a:b,c}:\text{cell}\big(A\cup \{a\}, \mathcal{W}\cup \{w_{a}\}, w\big)\to \text{cell}\big(A\cup\{b,c\},\mathcal{W}\cup\{w_{b}, w_{c}\}, w\big).
\]
The $\Z/2\Z$-equivariance of $\iota$ means that $\text{spin}_{a:b,c}$ is well-defined.

Note that the $\text{spin}$-maps induce chain maps, which we also call $\text{spin}$:
\[
\text{spin}_{a:b,c}:C_{*}\Big(\text{cell}\big(A\cup \{a\}, \mathcal{W}\cup \{w_{a}\}, w\big)\Big)\to C_{*+1}\Big(\text{cell}\big(A\cup\{b,c\},\mathcal{W}\cup\{w_{b}, w_{c}\}, w\big)\Big).
\]
These maps send an $n$-cell in $\text{cell}\big(A\cup \{a\}, \mathcal{W}\cup \{w_{a}\}, w\big)$ to the signed sum of the two $(n+1)$-cells in $\text{cell}\big(A\cup\{b,c\},\mathcal{W}\cup\{w_{b}, w_{c}\}, w\big)$ where $b\,c$ and $c\,b$ take the place of $a$ and every other label remains the same.
Here the cell where $b\,c$ replaces $a$ has the same sign as the cell with $a$ in it, and the cell with $c\,b$ in it has a sign of $(-1)^{w_{b}w_{c}-1}$.
This ensures that $\text{spin}_{a:b,c}$ is a chain map.
See Figure \ref{spinabcmap} for an example of a $\text{spin}$-map.

\begin{figure}[h]
\centering
\captionsetup{width=.8\linewidth}
\includegraphics[width = 10cm]{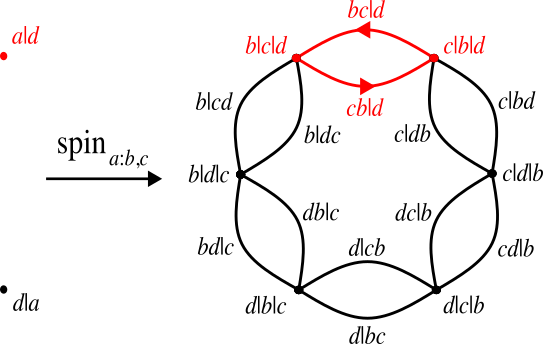}
\caption{The map $\text{spin}_{a:b,c}$ from $\text{cell}\big(\{a, d\}, \{2, 1\}, 2\big)$ to $\text{cell}\big(\{b,c, d\}, \{1,1, 1\}, 2\big)$. 
The image of the $0$-cycle $a|d$ is in red.
}
\label{spinabcmap}
\end{figure}

\begin{prop}\label{spinorder}
Let $(A,\mathcal{W})$ be a weighted set and $\big(\{a\}, \{w_{a}\}\big)$, $\big(\{b\}, \{w_{b}\}\big)$, $\big(\{c\}, \{w_{c}\}\big)$, $\big(\{a'\}, \{w_{a'}\}\big)$, $\big(\{b'\}, \{w_{b'}\}\big)$, and $\big(\{c'\}, \{w_{c'}\}\big)$ be weighted symbols not in $(A,\mathcal{W})$ such that $w_{a}=w_{b}+w_{c}$ and $w_{a'}=w_{b'}+w_{c'}$. Then, $\text{spin}_{a:b,c}$ and $\text{spin}_{a':b',c'}$ commute when their composition is viewed as a map $C_{*}\Big(\text{cell}\big(A\cup\{a, a'\}, \mathcal{W}\cup \{w_{a},w_{a'}\}, w\big)\Big)\to C_{*+2}\Big(\text{cell}\big(A\cup\{b,c, b', c'\}, \mathcal{W}\cup \{w_{b},w_{c}, w_{b'}, w_{c'}\}, w\big)\Big)$.
\end{prop}

\begin{proof}
This follows immediately from the definition of the $\text{spin}$-maps. See Figure \ref{spinmaponconfpic}.
\end{proof}


We will write $\text{spin}_{a:b,c}$ for the chain maps induced by $\text{spin}_{a:b,c}$. 
Next, we use the $\text{spin}$-maps to define the homology classes of wheels and filters that Alpert-Manin uses to describe a basis for the homology of $\text{cell}(n,w)$.

\subsection{Wheels and Filters}

Alpert--Manin noted that ordering the elements of $S_{n}$ and considering the correspondence between the top dimensional cells of $\text{cell}(n)$ and the elements of $S_{n}$ gives one a filtration of $\text{cell}(n)$, see \cite[Section 4.1]{alpert2021configuration1}.
They used this filtration to decompose the homology of $\text{cell}(n,w)$ as a direct sum of homology groups of certain weighted permutohedra. 
We recall some of their results, which will be of use to us later, and define generalized notions of wheels and filters, which will allow us to better describe the relations in our presentation for $H_{*}\big(\text{conf}(\bullet, w);\Q\big)$ as a twisted algebra.

\begin{defn}
Write the elements of $S_{n}$ as permutations of $1\cdots n$. The \emph{lexicographic ordering} on $S_{n}$ is a total ordering on $S_{n}$ such that $\sigma<\rho$ if the first value where the permutations of representing $\sigma$ and $\rho$ disagree when read left to right is smaller in $\sigma$ than in $\rho$.
\end{defn}

\begin{exam}
In $S_{5}$, $2\,3\,1\,4\,5<2\,3\,4\,1\,5$, as $2\,3\,1\,4\,5$ and $2\,3\,4\,1\,5$ first differ in the third entry where $1<4$.
\end{exam}

\begin{defn}
Given an element $\sigma\in S_{n}$, an \emph{axle} of $\sigma$ is the largest entry of $\sigma$ to that point, and a \emph{wheel} is the axle and all the following smaller entries before the next axle.
\end{defn}

\begin{defn}
Given $\sigma\in S_{n}$, we write \emph{$n-\#\sigma$} for \emph{number of wheels} in $\sigma$, and \emph{$\big(n-\#\sigma,\mathcal{W}(\sigma)\big)$} for the ordered weighted set whose elements are labeled by the wheels of $\sigma$ such that the weight of a wheel is the number of entries in the wheel and such that the set of wheels is ordered by the largest element in each wheel.
For a weighted set $(A, \mathcal{W})$ and $\sigma\in S_{A}$, we write \emph{$\big(A-\#\sigma,\mathcal{W}(\sigma)\big)$} for the ordered weighted set whose elements are labeled by the wheels of $\sigma$ such that the weight of a wheel is the total weight of the elements in the wheel and such that the set of wheels is ordered by the largest element in each wheel.
\end{defn}

\begin{exam}
The permutation $4\,5\,6\,7\,2\,8\,1\,3\in S_{8}$ has wheels $4$, $5$, $6$, $7\,2$, and $8\,1\,3$ with axles $4$, $5$, $6$, $7$ and $8$, respectively. 
The weighted set corresponding to $4\,5\,6\,7\,2\,8\,1\,3$ is $\big(\{4, 5, 6, 7\,2, 8\,1\,3\}, \{1,1,1,2,3\}\big)$.
\end{exam}

We will consider the weighted cellular complex $\text{cell}\big(n-\#\sigma, \mathcal{W}(\sigma),w\big)$,
and note that when all the cells have weight $1$ we have $\text{cell}\big(\{1,\dots, n\}, (1,\dots, 1), w\big)\simeq \text{cell}(n,w)$. 

\begin{defn}
Given $\sigma\in S_{A}$ the map \emph{$\text{spin}_{\sigma}$} is the map from $\text{cell}(A-\#\sigma, \mathcal{W}(\sigma), w)$ to $\text{cell}(A,\mathcal{W},w)$ obtained by applying a sequence of maps of the form $\text{spin}_{a:b,c}$, where these $\text{spin}$-maps unravel each wheel of $\sigma$ one disk at a time by splitting the wheel $a$ into wheels $b$ and $c$ such that $c$ is the last entry of the wheel $a$ and $b$ is the remainder of $a$. 
The map \emph{$\text{spin}_{\sigma}$} is obtained by repeating this process until all labels are recovered.
\end{defn}

By Proposition \ref{spinorder} we can expand the wheels of $\sigma$ in any order we want, explaining the lack of reference to which wheel we expand first when calculating $\text{spin}_{\sigma}$. 
We will also write $\text{spin}_{\sigma}$ for the corresponding chain map.

\begin{exam}
Let $\sigma=2\,1\,4\,3\in S_{4}$ and $w=3$. Then
\[
\text{spin}_{2\,1\,4\,3}:\text{cell}\big(\{2\,1, 4\,3\},\{2,2\}, 3\big)\to \text{cell}(4,3)
\]
makes the following diagram commute:

\begin{center}
\begin{tikzcd}
{\text{cell}\big(\{2\,1,4\,3\},\{2,2\},3\big)} \arrow[rr, "{\text{spin}_{4\,3;4,3}}"] \arrow[dd, "{\text{spin}_{2\,1;2,1}}"'] \arrow[rrdd, "{\text{spin}_{2\,1\,4\,3}}"] &  & {\text{cell}\big(\{2\,1,3,4\},\{2,1,1\},3\big)} \arrow[dd, "{\text{spin}_{2\,1;2,1}}"] \\
                                                                                                                                                                         &  &                                                                                        \\
{\text{cell}\big(\{2,1,4\,3\},\{1,1,2\},3\big)} \arrow[rr, "{\text{spin}_{4\,3;4,3}}"']                                                                                  &  & {\text{cell}(4,3)}                                                                    
\end{tikzcd}
\end{center}
Moreover, this diagram commutes at the level of chain maps. 
See Figure \ref{spinmaponconfpic}.
\end{exam}

\begin{figure}[h]
\centering
\captionsetup{width=.8\linewidth}
\includegraphics[width = 14cm]{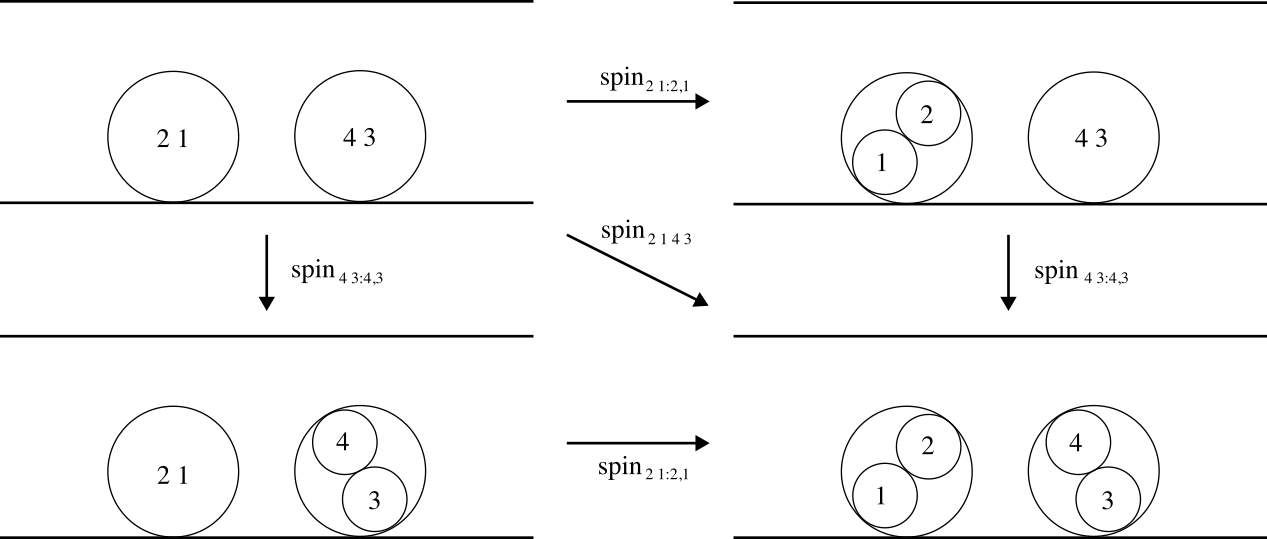}
\caption{The spin-map $\text{spin}_{2\,1\,4\,3}$ and the of the commutativity of the $\text{spin}$-maps $\text{spin}_{2\,1:2,1}$ and $\text{spin}_{4\,3:4,3}$ as maps on disk configuration space. 
The smaller disks inside the larger disks can take any position relative to each other.
This corresponds to two $1$-dimensional tori, or a $2$-torus in configuration space.
}
\label{spinmaponconfpic}
\end{figure}

\begin{defn}
Given $\sigma \in S_{n}$ let $S_{Wh(\sigma)}$ denote the symmetric group on the set of wheels of $\sigma$, and let \emph{$S(\sigma)$} denote the elements in $S_{n}$ that when written out as permutations of $\{1,\dots, n\}$ contain all the wheels of $\sigma$ as proper substrings, i.e., the $S_{Wh(\sigma)}$-orbit of $\sigma$.
\end{defn}

\begin{exam}
Let $\sigma=2\,1\,3\,4$.
Then, $S(\sigma)=\{2\,1\,3\,4, 2\,1\, 4\,3, 3\,2\,1\,4, 3\,4\,2\,1, 4\,2\,1\,3, 4\,3\,2\,1\}$, there are $6$ permutations of $\{1, 2,3,4\}$ in which $1$ immediately follows $2$ as is the case in the wheel $2\,1$.
\end{exam}

All the elements of $S(\sigma)$ other than $\sigma$ have strictly fewer wheels than $\sigma$, and for any $\tau\in S(\sigma)$ there are sequences of $\text{spin}$-maps from $\text{cell}\big(n-\#\tau, \mathcal{W}(\tau), w\big)$ to $\text{cell}(n, w)$ that break up the wheels of $\tau$ to get the wheels of $\sigma$.

\begin{defn}
Given $\sigma\in S_{n}$ and $\tau\in S(\sigma)$, let the map \emph{$\text{spin}_{\tau, \sigma}$} from $\text{cell}\big(n-\#\tau, \mathcal{W}(\tau), w\big)$ to $\text{cell}\big(n-\#\sigma, \mathcal{W}(\sigma), w\big)$ be the product of $\text{spin}$-maps that recovers the wheels of $\sigma$ from the wheels of $\tau$ by unwinding the wheels of $\tau$ right to left via maps of the form $\text{spin}_{a:b,c}$ where $a$ is a wheel of $\tau$ and $c$ is the first wheel of $\sigma$ on the right of $a$.
That is, if $w_{\tau}=w_{\sigma, 1}\cdots w_{\sigma, k}$ is a wheel of $\tau$ composed of wheels $w_{\sigma, 1}, \dots, w_{\sigma, k}$ of $\sigma$ we first apply the map  $\text{spin}_{w_{\sigma, 1}\cdots w_{\sigma, k}:w_{\sigma, 1}\cdots w_{\sigma, k-1}, w_{\sigma, k}}$, and repeat this process for the wheel $w_{\sigma, 1}\cdots w_{\sigma, k-1}$.
\end{defn}

\begin{exam}
Let $\sigma=2\,1\,3\,4$ and $\tau=4\,3\,2\,1$. 
Then $\tau\in S(\sigma)$, and $\text{spin}_{\tau, \sigma}=\text{spin}_{4\,3:4,3}\circ\text{spin}_{4\,3\,2\,1:4\,3,2\,1}$.
\end{exam}

Note that the composition
\[
\text{spin}_{\sigma}\circ i_{id}:P\big(A-\#\sigma,\mathcal{W}(\sigma), w\big) \to \text{cell}(A,\mathcal{W},w)
\]
induces a map on the chain complexes. 
Alpert--Manin proved that this map splits, giving a splitting on homology \cite[Theorem 4.1]{alpert2021configuration1}. 
We will need a generalized version of this result to give a new basis for $H_{k}\big(\text{conf}(n,w);\Q\big)$.

\begin{prop}\label{AMdecomp2}
The homology of $\text{cell}(A,\mathcal{W}, w)$ decomposes as
\[
H_{*}\big(\text{cell}(A,\mathcal{W}, w)\big)=\bigoplus_{\sigma\in S_{A}}H_{*-\#\sigma}\Big(P\big(A-\#\sigma, \mathcal{W}(\sigma), w\big)\Big).
\]
\end{prop}

To obtain a class in $H_{*}\big(\text{cell}(A,\mathcal{W}, w)\big)$ from a class in $H_{*-\#\sigma}\Big(P\big(A-\#\sigma, \mathcal{W}(\sigma), w\big)\Big)$ one applies a spin-map, unraveling the wheels of $A-\#\sigma$ to get the elements (weighted disks) of $A$.

\begin{exam}
Consider the weighted set $(A,\mathcal{W})=\big(\{1, 2, 3\},\{1,1,1\}\big)$.
We have that $\text{cell}(A, \mathcal{W}, 2)=\text{cell}(3,2)\simeq \text{conf}(3,2)$.
Note $S_{A}=S_{3}$ has $6$ elements, namely $1\,2\,3$, $1\,3\,2$, $2\,1\,3$, $2\,3\,1$, $3\,1\,2$, and $3\,2\,1$.
From these elements we get the weighted sets $\big(3-\#1\,2\,3, \mathcal{W}(1\,2\,3)\big)=\big(\{1,2,3\}, \{1,1,1\}\big)$, 
$\big(3-\#1\,3\,2,\mathcal{W}(1\,3\,2)\big)=\big(\{1,3\,2\}, \{1,2\}\big)$, 
$\big(3-\#2\,1\,3,\mathcal{W}(2\,1\,3)\big)=\big(\{2\,1,3\}, \{2,1\}\big)$, 
$\big(3-\#2\,3\,1,\mathcal{W}(2\,3\,1)\big)=\big(\{2,3\,1\}, \{1,2\}\big)$, 
$\big(3-\#3\,1\,2, \mathcal{W}(3\,1\,2)\big)=\big(\{3\,1\,2\}, \{3\}\big)$, and $\big(3-\#3\,2\,1, \mathcal{W}(3\,2\,1)\big)=\big(\{3\,2\,1\}, \{3\}\big)$
By Proposition \ref{AMdecomp2}, we have that
\[
H_{1}\big(\text{cell}(3,2)\big)=\bigoplus_{\sigma\in S_{3}}H_{1-\#\sigma}\Big(P\big(3-\#\sigma, \mathcal{W}(\sigma), w\big)\Big),
\]
i.e.,
\begin{multline*}
H_{1}\big(\text{cell}(3,2)\big)=H_{1}\Big(P\big(\{1,2,3\}, \{1,1,1\},2\big)\Big)\oplus H_{0}\Big(P\big(\{1,3\,2\}, \{1,2\},2\big)\Big)\oplus H_{0}\Big(P\big(\{2\,1,3\}, \{2,1\}, 2\big)\Big)\\
\oplus H_{0}\Big(P\big(\{2,3\,1\}, \{1,2\}, 2\big)\Big)\oplus H_{-1}\Big(P\big(\{3\,1\,2\}, \{3\},2\big)\Big)\oplus H_{-1}\Big(P\big(\{3\,2\,1\}, \{3\},2\big)\Big).
\end{multline*}
The last two summands are trivial as $P\big(\{3\,1\,2\}, \{3\},2\big)$ and $P\big(\{3\,2\,1\}, \{3\},2\big)$ are empty as the weight of the single element is more than the limit. 
The weighted permutohedra $P\big(\{1,3\,2\}, \{1,2\},2\big)$, $P\big(\{2\,1,3\}, \{2,1\}, 2\big)$, and $P\big(\{2,3\,1\}, \{1,2\}, 2\big)$ are two points sets, and the weighted permutahedron $P\big(\{1,2,3\}, \{1,1,1\},2\big)$ is homotopy equivalent to a circle.
It follows that $H_{1}\big(\text{conf}(3,2)\big)\cong H_{1}\Big(\text{cell}\big(\{1,2,3\}, \{1,1,1\},2\big)\Big)\cong \Z^{7}$.
See Figure \ref{gluepermutadratogetcell}.
\begin{figure}[h]
\centering
\captionsetup{width=.8\linewidth}
\includegraphics[width = 14cm]{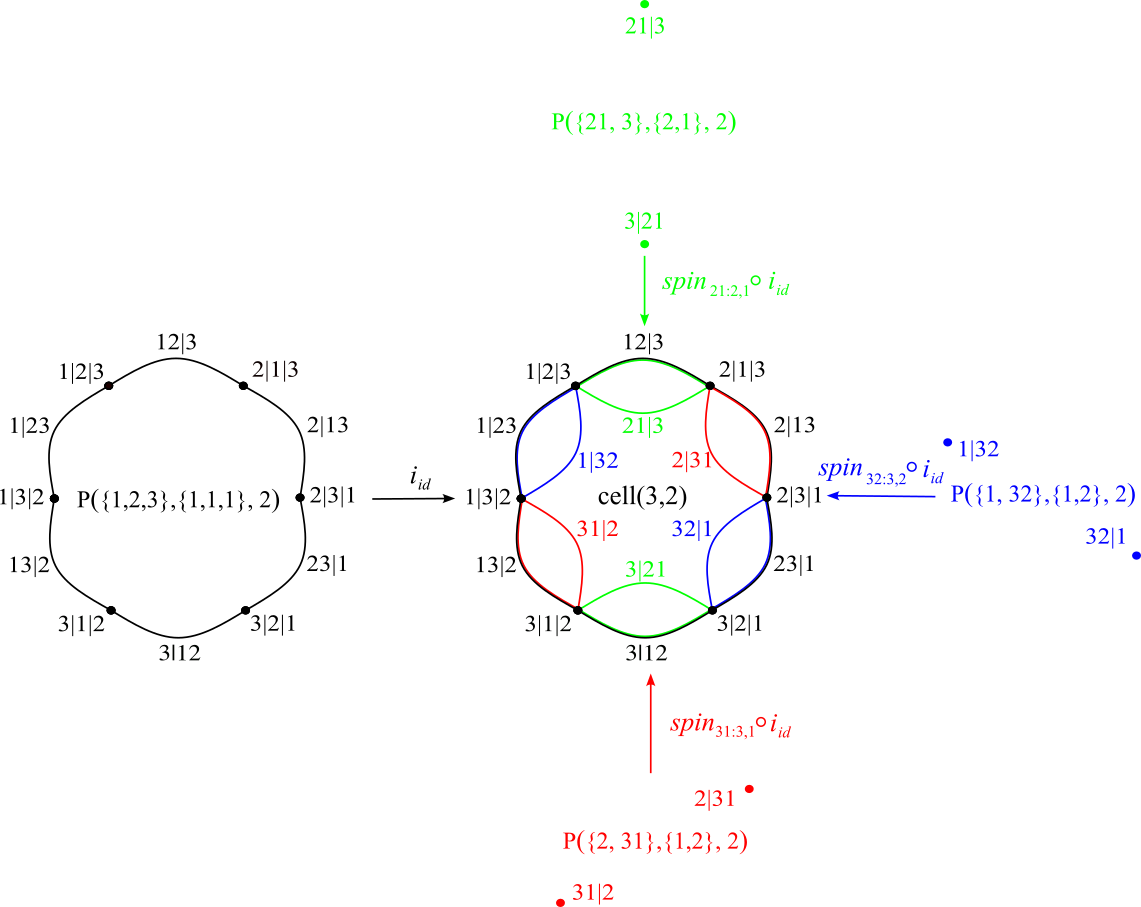}
\caption{The decomposition of Proposition \ref{AMdecomp2} of the homology of $\text{cell}(3,2)$ visualized at the level of cell complexes.
A quick visual check confirms that $H_{1}\big(\text{cell}(3,2)\big)\cong \Z^{7}$, which is the direct sum of $H_{0}\Big(P\big(\{1,3\,2\}, \{1,2\},2\big)\Big)\cong \Z^{2}$, $H_{0}\Big(P\big(\{2\,1,3\}, \{2,1\}, 2\big)\Big)\cong \Z^{2}$, $H_{0}\Big(P\big(\{2,3\,1\}, \{1,2\}, 2\big)\Big)\cong \Z^{2}$ and  $H_{1}\Big(P\big(\{1,2,3\}, \{1,1,1\},2\big)\Big)\cong \Z$.
Additionally, we can how the classes in $H_{1}\big(\text{cell}(3,2)\big)$ can be viewed as arising by applying the spin-maps to the classes corresponding to the weighted permutohedra.
}
\label{gluepermutadratogetcell}
\end{figure}
\end{exam}

This proposition follows from Alpert--Manin's argument by replacing $S_{n}$ with $S_{A}$, ordering the elements of $S_{A}$, noting the correspondence of these elements with top dimensional cells of $\text{cell}(A, \mathcal{W})$, using this to get a filtration of $C_{*}\big(\text{cell}(A, \mathcal{W})\big)$, and noting that Lemma \ref{AM technical lemma} holds when we replace $S_{n}$ with $S_{A}$. 

This result proves that the integral homologies of $\text{cell}(n,w)$ and $\text{cell}(A,\mathcal{W}, w)$ are free $\Z$-modules as the homology groups of the weighted permutohedra are free. 
Unfortunately, these decompositions are not well behaved with respect to the action of the corresponding permutation groups as the wheels of $\sigma\in S_{A}$ are not preserved by the $S_{A}$-action on $\sigma$. 
This poor behavior with respect to the symmetric group action is carried over to Alpert--Manin's basis for homology, causing several problems. 
As a result, we seek a new basis for (rational) homology. 
Before we do that, we recall Alpert--Manin's basis.

We begin by defining two classes of elements in homology: wheels and filters. 
Our definition of wheels, and, as a result, filters is more general than Alpert--Manin's, as they only consider what we will call proper wheels, see \cite[Section 1]{alpert2021configuration1}. 
We extend the definition of filters to include several classes that are trivial in homology.

\begin{defn}
Let $(A, \mathcal{W})=\big(\{W\}, \{n\}\big)$ be a weighted set such that $W$ is a permutation of $\{1,\dots, n\}$, where $n\le w$. 
A \emph{wheel} is the image of the fundamental class of the cycle $W\in C_{0}\big(P(A,\mathcal{W}, w)\big)$ in $H_{n-1}\big(\text{cell}(n,w)\big)$ under the composition of $i_{id}$ and any sequence of spin-maps. 
If the $\text{spin}$-maps are of the form $\text{spin}_{W';W_{1}',W_{2}'}$ where $W'$ is a permutation of some ordered subset of $\{1, \dots, n\}$, the last element of the ordered subset is $W_{2}'$, and $W_{1}'$ is the remaining elements with their order preserved, then we say that the wheel is \emph{proper}.
\end{defn}

Since $n\le w$, it follows that $P(A,\mathcal{W}, w)$ and $\text{cell}(A, \mathcal{W}, w)$ are points and $i_{id}$ is an isomorphism between $C_{*}\big(P(A,\mathcal{W}, w)\big)$ and $C_{*}\big(\text{cell}(A, \mathcal{W}, w)\big)$. 
By Theorem \ref{AMthmB} every wheel is homologous to a sum of a subset of the proper wheels. 
Letting $w\to \infty$, Theorem \ref{AMthmB} recovers the ``tall trees" basis for $H_{n-1}\big(F_{n}(\R^{2})\big)$, see, for example, \cite{knudsen2018configuration, sinha2006homology}.

Wheels have a nice topological interpretation: A wheel on $n$ disks is a non-contractible embedded $(n-1)$-torus in $\text{conf}(n,w)$.
This is best seen in the case of proper wheels.
We write $W(i_{1},\dots, i_{n})$ for the proper wheel arising from $(A,\mathcal{W})=\big(\{i_{1}\cdots i_{n}\}, \{n\}\big)$. 
This corresponds to the disk labeled $i_{2}$ orbiting the disk labeled $i_{1}$, the disk labeled $i_{3}$ independently orbiting that orbiting pair, and so on.
See Figure \ref{wheels}.

Given a wheel $W$, we write $\text{spin}_{W}$ for the product of $\text{spin}$-maps that produce the wheel $W$.

\begin{figure}[h]
\centering
\captionsetup{width=.8\linewidth}
\includegraphics[width = 10cm]{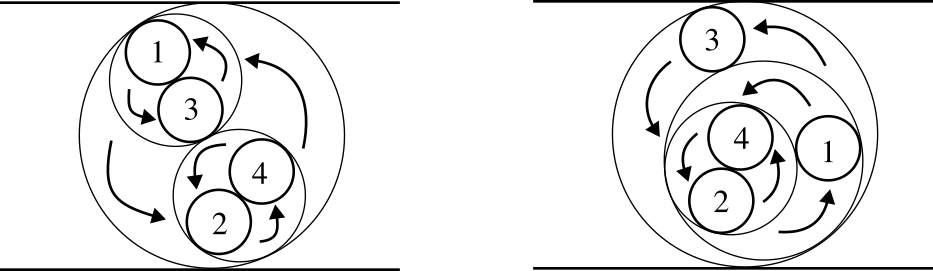}
\caption{On the left: A wheel on the disks $1, 2, 3, 4$. 
On the right: The proper wheel $W(4,2,1,3)$.
}
\label{wheels}
\end{figure}

Next, we use the wheels to define filters, though we first make a remark on the topology of certain weighted permutohedra.
\begin{remark}
Given $w$ and $(A, \mathcal{W})=\big(\{a_{1}, \dots, a_{m}\}, \{w_{1}, \dots, w_{m}\}\big)$, a weighted set on $m\ge 2$ elements such that $w_{1}+\cdots+\widehat{w_{k}}+\cdots+w_{m}\le w$ for all $k$ and $w_{1}+\cdots+w_{m}>w$, then the weighted permutohedron $P\big(\{a_{1}, \dots, a_{m}\}, \{w_{1}, \dots, w_{m}\}, w\big)$ is homotopy equivalent to $S^{m-2}$, and removing a single $(m-2)$-cell from $P\big(\{a_{1}, \dots, a_{m}\}, \{w_{1}, \dots, w_{m}\}, w\big)$ changes the homotopy type to $\R^{m-2}$.
Additionally, every $(m-2)$-cell of $P\big(\{a_{1}, \dots, a_{m}\}, \{w_{1}, \dots, w_{m}\}\big)$ is in $P\big(\{a_{1}, \dots, a_{m}\}, \{w_{1}, \dots, w_{m}\}, w\big)$.
\end{remark}

\begin{defn}
For $m\ge 2$ let $(A, \mathcal{W})=\big(\{W_{1}, \dots, W_{m}\}, \{n_{1}, \dots, n_{m}\}\big)$ be a weighted set such that each $W_{j}$ is a permutation of $A_{j}\subset\{1,\cdots, n\}$ such that $A_{j}\cap A_{i}=\emptyset$ if $j\neq i$, and $\bigcup A_{j}=\{1,\dots, n\}$; $n_{j}=|A_{j}|$, and for any $k$ the sum of all but one of the weights $n_{1}+\cdots+\widehat{n_{k}}+\cdots+n_{m}$ is at most $w$. 
Let $Z$ denote the cycle in $C_{m-2}\big(P(A, \mathcal{W},w)\big)$ that is the boundary of $W_{1}\,\cdots\, W_{m}\in C_{m-1}\big(P(A,\mathcal{W})\big)$.
A \emph{filter} is the induced image of $Z$ in $H_{n-2}\big(\text{conf}(n,w)\big)$ under the composition of $i_{id}$ and any sequence of $\text{spin}$-maps that yield $m$ wheels such that the $j^{\text{th}}$ wheel is on the disks in $A_{j}$.
\end{defn}

If $n\le w$, then $P(A, \mathcal{W}, w)$ is contractible, so $H_{m-2}\big(P(A, \mathcal{W}, w)\big)$ is trivial and any filter arising from $Z$ is homologically trivial.
Though these filters are trivial, we will consider them as we will need to consider analogues of them in Lemma \ref{relationlem}. 

There is a nice interpretation of filters in the $\text{conf}(n,w)$-model: A non-trivial filter on $m$ wheels arises from $m$ wheels moving left and right in the strip past each other such that the centers of mass of all the wheels cannot have the same $x$-coordinate and the $i^{\text{th}}$ wheel always moves above the $j^{\text{th}}$ wheel if $i<j$. 
A filter is an embedded $S^{m-2}\times T^{n-m}$ in configuration space where the movement of the disks in the wheels produces the $T^{n-m}$ and the movement of the wheels past each other produces the $S^{m-2}$. 
It follows that there can be at most $\frac{3w}{2}$ wheels in a filter.

Given an ordered set of wheels $\{W_{1}, \dots, W_{m}\}$ on disjoint labels such that the $i^{\text{th}}$-wheel consists of $n_{i}$ disks, let $F(W_{1}, \dots, W_{m})$ denote the filter arising from $P\big(\{W_{1},\dots, W_{m}\}, \{n_{1}, \dots, n_{m}\}, w\big)$.
The ordering of the wheels $W_{1},\dots, W_{m}$ corresponds to insisting that the wheel $W_{i}$ always moves above the wheel $W_{j}$ if $i<j$. 
If the wheels of a filter are proper we say that the filter is \emph{proper} and we write $F\big(W(i_{1, 1},\dots, i_{1, n_{1}}),\dots, W(i_{m,1},\dots, i_{m, n_{m}})\big)$ for the proper filter on the proper wheels $W(i_{1, 1},\dots, i_{1, n_{1}})$, $\dots$, $W(i_{m,1},\dots, i_{m, n_{m}})$. See Figure \ref{filteron3wheels}.

\begin{figure}[h]
\centering
\captionsetup{width=.8\linewidth}
\includegraphics[width = 6cm]{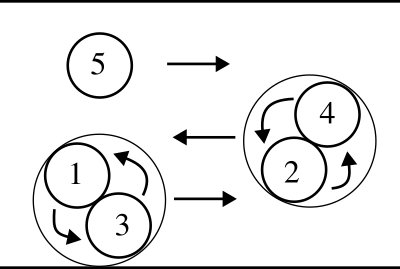}
\caption{The filter $F\big(W(5), W(4,2), W(3,1)\big)$. 
The movement of the disks in the wheels $W(4,2)$ and $W(3,1)$ generate a $2$-torus that can be thought of as the product of two $1$-tori.
The movement of the wheels past each other generates an $S^{1}$.
As a result, $F\big(W(5), W(4,2), W(3,1)\big)$ can be represented by an embedded non-contractible $S^{1}\times T^{2}$ in $\text{conf}(5,4)$.
}
\label{filteron3wheels}
\end{figure}

In the case $m=2$ we have that
\[
F(W_{1}, W_{2}):=W_{1}|W_{2}+(-1)^{\big(n_{1}-1\big)\big(n_{2}-1\big)+1}W_{2}|W_{1},
\]
where the concatenation product is the one arising from the chain complexes.
These filters on $2$ wheels are different than the filters on $2$ wheels defined by Alpert--Manin, as we have a different convention for the sign of the second concatenation product of wheels \cite[Section 1]{alpert2021configuration1}. 
This is a result of our weighted boundary operator, and this convention ensures that if $n_{1}+n_{2}\le w$, the filter $F(W_{1}, W_{2})$ is trivial in homology, see Proposition \ref{commutewheels}.
This is not necessarily true for Alpert--Manin's filters on $2$ wheels.
Later, we will do away with filters on $2$ wheels as we will want a presentation for rational homology as a twisted algebra that does not use such filters as generators, as they worsen the first-order representation stability range and obfuscate the process of taking the product of two wheels.
First, we state Alpert--Manin's basis theorem as we will use it to prove that our basis for rational homology is, in fact, a basis.

\begin{thm}\label{AMthmB}
(Alpert--Manin \cite[Theorem B]{alpert2021configuration}) The homology group $H_{k}\big(\text{conf}(n,w)\big)$ is free abelian and has a basis consisting of concatenations of proper wheels where the largest labeled disk comes first and non-trivial proper filters on at least 2 wheels such that in each wheel the largest labeled disk comes first. 
We say one proper wheel ranks above another if it has more disks, or has the same number of disks and its largest label is greater. A cycle is in the basis if and only if:
\begin{enumerate}
\item The proper wheels inside each filter on more than $3$ wheels are in order of increasing largest label (regardless of the number of disks). 
\item The proper wheels inside a filter on 2 wheels are in order of increasing rank.
\item Adjacent proper wheels not inside a proper filter are in order of decreasing rank.
\item Every proper wheel immediately to the left of a proper filter ranks above the least wheel in the filter.
\end{enumerate}
\end{thm}

This basis is not preserved by the action of the symmetric group as the order of the wheels in a filter and the order of the disks in a wheel matters. 
The latter problem is a minor one due to Proposition \ref{properwheels}. 
The former is much more of a challenge. 
Soon, we will replace filters with an average of filters. 
These averages behave nicely when the wheels or the disks are permuted, and they will allow us to provide a finite presentation for $H_{*}\big(\text{cell}(\bullet,w);\Q\big)$ as a twisted algebra.

\subsection{Wheel and Filter Relations}
We prove several relations on wheels and filters that will be used throughout this paper to find a new basis for the homology of the ordered configuration space of open unit-diameter disks in the infinite strip of width $w$, as well as demonstrate a finite presentation the homology of this space as a twisted algebra.

\begin{prop}\label{properwheels}
For $n\le w$, any wheel $W$ on $n$ disks can be written as a sum of proper wheels on $n$ disks such that the largest label in each proper wheel comes first.
\end{prop}

\begin{proof}
The only elements in the basis for $H_{n-1}\big(\text{conf}(n,w)\big)$ in Theorem \ref{AMthmB} are the proper wheels on $n$ disks where the largest labeled disk comes first.
\end{proof}

Proposition \ref{properwheels} holds at the level of chain complexes. 

\begin{prop}\label{spin relation}
Given a wheel $W$ on $n\le w$ disks, there are proper wheels on the same set of disks of the form $W(i_{1}, \dots, i_{n})$ such that $i_{1}>i_{j}$ for all $j\neq1$ such that 
\[
\text{spin}_{W}=\sum a_{(i_{1},\dots, i_{n})}\text{spin}_{W(i_{1}, \dots, i_{n})}.
\]
\end{prop}

\begin{proof}
One can check that Proposition \ref{properwheels} holds at the level of chain complexes through the use of a Jacobi type identity---Sinha does something equivalent to this in the proof of \cite[Proposition 2.7]{sinha2006homology}. 
It follows that this is true for the $\text{spin}$-maps from $C_{0}\Big(P\big(\{W\}, \{n\}, w\big)\Big)$ to $C_{n-1}\big(\text{cell}(n,w)\big)$. 
At the chain level the $\text{spin}$-maps send a cell that includes the label $a$ to the sum of the $2$ cells where $a$ is replaced by $b\,c$ and $c\,b$, i.e.,  $a\mapsto b\,c+(-1)^{w_{b}w_{c}-1} c\,b$.
It follows that we can ignore the rest of the labels, so the relation holds for all $\text{spin}_{W}$.
\end{proof}

Next, we state a relation that will allow us to (anti)-commute certain pairs of wheels.

\begin{prop}\label{commutewheels}
Let $W_{1}$ and $W_{2}$ be two wheels on $n_{1}$ and $n_{2}$ disks, respectively, such that $n_{1}+n_{2}\le w$. Then,
\[
W_{1}|W_{2}=(-1)^{\big(n_{1}-1\big)\big(n_{2}-1\big)}W_{2}|W_{1}.
\] 
\end{prop}

\begin{proof}
Note that the homology classes $W_{1}|W_{2}$ and $W_{2}|W_{1}$ arise from the cycles $W_{1}|W_{2}$ and $W_{2}|W_{1}$ in $C_{0}\Big(P\big(\{W_{1}, W_{2}\}, \{n_{1}, n_{2}\}, w\big)\Big)$, respectively, and the spin-maps $\text{spin}_{W_{1}}$ and $\text{spin}_{W_{2}}$.
Since $n_{1}+n_{2}\le w$, the cell $W_{1}W_{2}$ is in $P\big(\{W_{1}, W_{2}\}, \{n_{1}, n_{2}\}, w\big)$. 
Taking the boundary of $W_{1}W_{2}$ we see that
\[
\partial(W_{1}W_{2})=(-1)^{n_{1}}W_{1}|W_{2}+(-1)^{n_{2}}(-1)^{n_{1}n_{2}}W_{2}|W_{1}.
\]
In homology it follows that the class arising from $(-1)^{n_{1}}W_{1}|W_{2}+(-1)^{n_{2}}(-1)^{n_{1}n_{2}}W_{2}|W_{1}$ is trivial.
Therefore
\[
(-1)^{n_{1}}W_{1}|W_{2}+(-1)^{n_{2}}(-1)^{n_{1}n_{2}}W_{2}|W_{1}=0,
\]
which can be rewritten as
\[
W_{1}|W_{2}=(-1)^{\big(n_{1}-1\big)\big(n_{2}-1\big)}W_{2}|W_{1}.
\] 
\end{proof}

Proposition \ref{commutewheels} has a nice topological interpretation.
The wheels $W_{1}$ and $W_{2}$ can freely move past each other in the strip as $n_{1}+n_{2}\le w$. 
It follows that the $0$-cells $W_{1}|W_{2}$ and $W_{2}|W_{1}$ are homologous up to sign.
The difference in sign arises from the way we parametrize the embedded torus in $\text{conf}(n_{1}+n_{2}, w)$ yielding these classes.
This $(n_{1}+n_{2}-2)$-torus is the product of an $(n_{1}-1)$-torus yielding $W_{1}$ and an $(n_{2}-1)$-torus yielding $W_{2}$.
Switching the order in which we parametrize these subtori yields a sign change of $(-1)^{\big(n_{1}-1\big)\big(n_{2}-1\big)}$ as we need $\big(n_{1}-1\big)\big(n_{2}-1\big)$ transpositions to do so, and transpositions act by multiplication by $-1$.


We state a relation that proves that any sufficiently nice filter can be written as a sum of proper filters in Alpert--Manin's basis.
We use this family of relations to prove that we can replace filters with rational homology classes that we will call averaged-filters.

\begin{prop}\label{filter on m wheels is a proper filter on m wheels}
For $m\ge 3$, let $W_{1},\dots, W_{m}$ be wheels on $n_{1},\dots, n_{m}$ disks, respectively, such that if $i<j$ the highest labeled disk in $W_{j}$ has a larger label than the highest labeled disk in $W_{i}$, and if $m=2$, then $W_{2}$ outranks $W_{1}$.
Then, the filter
\[
F(W_{1}, \dots, W_{m})
\]
can be written as a sum of proper filters
\[
F(W_{1}, \dots, W_{m})=\sum a_{i_{1},
\dots, i_{n_{1}+\cdots+i_{n_{m}}}} F\big(W(i_{1, 1},\dots, i_{1, n_{1}}), \dots, W(i_{m, 1}, \dots, i_{m, n_{m}})\big)
\]
such that each filter is in the basis of Theorem \ref{AMthmB}, and the labels of disks in the wheel $W_{j}$ are the same as the labels of the disks in the wheel $W(i_{j, 1}, \dots, i_{j, n_{j}})$.
\end{prop}

\begin{proof}
By the definition of a filter, the filter $F(W_{1}, \dots, W_{m})$ is the image of $\text{spin}_{W_{m}}\circ\cdots\circ\text{spin}_{W_{1}}\circ i_{id}(Z)$ in $H_{n-2}\big(\text{conf}(n,w)\big)$, where $Z$ is the boundary of the top dimensional cell of $P\big(\{W_{1}, \dots, W_{m}\}, \{n_{1},\dots, n_{m}\}, w\big)$, i.e., $Z:=\partial(W_{1}\,\cdots\,W_{m}).$

By Proposition \ref{spin relation} we can write $\text{spin}_{W_{j}}$ as a sum of maps of the form $\text{spin}_{W(i_{j, 1}, \dots, i_{j, n_{j}})}$, where the labels of disks are the same as the labels of the disks in $W_{j}$, i.e., 
\[
\text{spin}_{W_{j}}=\sum a_{(i_{j,1},\dots, i_{j, n_{j}})}\text{spin}_{W(i_{j, 1}, \dots, i_{j, n_{j}})}.
\]

Since these are chain maps it follows that 
\begin{multline*}
\text{spin}_{W_{m}}\circ\cdots\circ\text{spin}_{W_{1}}\circ i_{id}(Z)\\
=\big(\sum a_{(i_{m,1},\dots, i_{m, n_{m}})}\text{spin}_{W(i_{m, 1}, \dots, i_{m, n_{m}})}\big)\circ\cdots\circ\big(\sum a_{(i_{1,1},\dots, i_{1, n_{1}})}\text{spin}_{W(i_{1, 1}, \dots, i_{1, n_{1}})}\big)\circ i_{id}(Z)\\
=\sum a_{(i_{m,1},\dots, i_{m, n_{m}})}\cdots a_{(i_{1,1},\dots, i_{1, n_{1}})}\text{spin}_{W(i_{m, 1}, \dots, i_{m, n_{m}})}\circ\cdots\circ\text{spin}_{W(i_{1, 1}, \dots, i_{1, n_{1}})}\circ i_{id}(Z).
\end{multline*}

Note that the cycle $\text{spin}_{W(i_{m, 1}, \dots, i_{m, n_{m}})}\circ\cdots\circ\text{spin}_{W(i_{1, 1}, \dots, i_{1, n_{1}})}\circ i_{id}(Z)$ represents the homology class $F\big(W(i_{1, 1},\dots, i_{1, n_{1}}), \dots, W(i_{m, 1}, \dots, i_{m, n_{m}})\big)$.
It follows that 
\[
F(W_{1}, \dots, W_{m})
\]
can be written as a sum of proper filters on $m$ wheels in the basis of Theorem \ref{AMthmB}, such that the labels of disks in the wheel $W_{j}$ are the same as the labels of the disks in the wheels $W(i_{j, 1}, \dots, i_{j, n_{j}})$:
\[
F(W_{1}, \dots, W_{m})=\sum a_{i_{1},
\dots, i_{n_{1}+\cdots+i_{n_{m}}}}F\big(W(i_{1, 1},\dots, i_{1, n_{1}}), \dots, W(i_{m, 1}, \dots, i_{m, n_{m}})\big).
\]
\end{proof}


\subsection{A New Basis for Rational Homology}

In this subsection we consider the rational homology of $\text{conf}(n,w)$.
Not much is lost in using rational coefficients as $H_{k}\big(\text{cell}(n,w)\big)$ is free abelian; moreover, there is much to be gained.
Representations of the symmetric group over $\Q$ are easier to work with than representations over $\Z$, and we need to be able to take averages over families of chain maps in order to define a new type of class in rational homology that we will call averaged-filters.
We will prove a relation between concatenation products of wheels and averaged-filters, and we will use this relation to find a basis for rational homology distinct from that of Alpert--Manin.
Additionally this relation will play a critical role in determining the how much higher order representation stability $\text{conf}(\bullet, w)$ exhibits.
All of the results in this subsection hold over a commutative ring $R$ where $n!$ is invertible.

\begin{defn}
Fix an ordering of the weighted set $(A, \mathcal{W})$, and label the cells of $P(A, \mathcal{W})$ so that the labels in each block are in increasing order. 
For $\sigma\in S_{(A,\mathcal{W})}$, let
\[
i_{\sigma}:P(A, \mathcal{W})\to \text{cell}(A, \mathcal{W})
\]
be the inclusion arising from the map that sends $a_{1}\,\dots\, a_{|A|}$, the top dimensional cell of $P(A, \mathcal{W})$, to $a_{\sigma(1)}\,\dots\, a_{\sigma(|A|)}$, the top dimensional cell in $\text{cell}(A, \mathcal{W})$. 
The \emph{averaged-inclusion} of the chain maps induced by an average of the $i_{\sigma}$, denoted \emph{$q$}, is the map
\[
q:C_{*}\big(P(A, \mathcal{W});\Q\big)\to C_{*}\big(\text{cell}(A, \mathcal{W});\Q\big)
\]
where
\[
q:=\frac{1}{|A|!}\sum_{\sigma\in S_{A}} \text{wsgn}(\sigma)(i_{\sigma})_{*}.
\]
See Figure \ref{qinclusion}.
\end{defn}

Since $q$ is a sum of chain maps, it is a chain map, and we use $q$ to define a new family of classes in rational homology of $\text{conf}(n,w)$ that are much better behaved with respect to the symmetric group actions than filters.


\begin{figure}[h]
\centering
\captionsetup{width=.8\linewidth}
\includegraphics[width = 12cm]{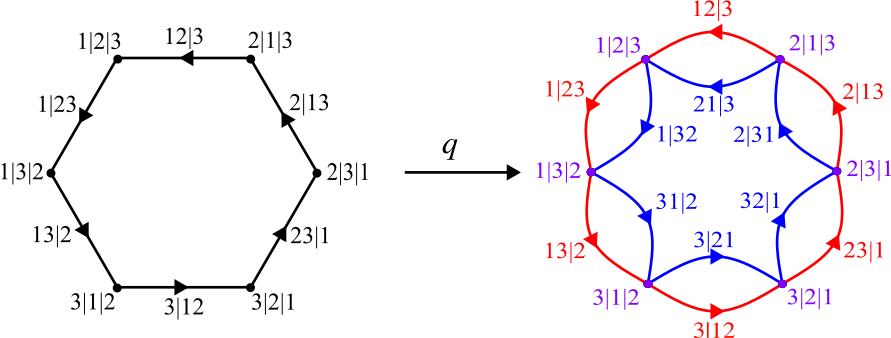}
\caption{The image of the cycle in $C_{1}\Big(P\big(\{1,2,3\}, \{1,1,1\}, 2\big);\Q\Big)$ corresponding to going once around the circle counter-clockwise. 
It is the cycle in $C_{1}\Big(\text{cell}\big(\{1,2,3\}, \{1,1,1\}, 2\big);\Q\Big)$ corresponding to taking the average of going once around the inner blue circle counter-clockwise and once around the outer red circle counter-clockwise, i.e., an average of two embedded $S^{1}$s in $\text{conf}(3,2)$.
The blue and red cells appear with coefficients $\pm \frac{1}{2}$ in the cycle.
}
\label{qinclusion}
\end{figure}

\begin{defn}
For $m\ge 2$, let $(A, \mathcal{W})=\big(\{W_{1}, \dots, W_{m}\}, \{n_{1}, \dots, n_{m}\}\big)$ be a weighted set such that each $W_{j}$ is a permutation of $A_{j}\subset\{1,\cdots, n\}$ such that $A_{j}\cap A_{i}=\emptyset$ if $j\neq i$, and $\bigcup A_{j}=\{1,\dots, n\}$; $n_{j}=|A_{j}|$, and for any $k$, we have $n_{1}+\cdots+\widehat{n_{k}}+\cdots+n_{m}\le w$. 
Let $Z$ denote the cycle in $C_{m-2}\big(P(A, \mathcal{W},w);\Q\big)$ that is the boundary of $W_{1}\,\cdots\, W_{m}\in C_{m-1}\big(P(A,\mathcal{W});\Q\big)$.
An \emph{averaged-filter} is the image of $Z$ in $H_{n-2}\big(\text{conf}(n,w);\Q\big)$ under $q$ and any sequence of $\text{spin}$-maps that yield $m$ wheels such that the $j^{\text{th}}$ wheel is on the disks in $A_{j}$.
If the  $\text{spin}$-maps yield proper wheels, we say that the averaged-filter is \emph{proper}. 
\end{defn}

In a proper averaged-filter the wheels are proper. 
It also follows from the definition that there are at most $\frac{3w}{2}$ disks in an averaged-filter, and more than $w$ disks in a non-trivial averaged-filter.
If $n\le w$, then $P(A, \mathcal{W}, w)$ is contractible and the averaged-filter is trivial in homology.
We will still consider such filters as they will aid in the statement of Lemma \ref{relationlem}. 

If $n\le \frac{3w}{2}$ and $\sigma\in S_{n}$ is such that any collection of all but one of the wheels in $\sigma$ has total weight at most $w$ and $\sigma$ has at least $2$ wheels, then $P\big(n-\#\sigma,\mathcal{W}(\sigma), w\big)$ is homotopy equivalent to $S^{m-2}$. 
As a result, a non-trivial averaged-filter has $m\ge 2$ wheels on $n>w$ total disks such that each wheel is made of at least $n-w$ disks, and the wheels are ordered by the largest label.
The wheels in the averaged-filter can move as follows:
Every wheel can perform its rotations independently, for a total of $n-m$ degrees of freedom, resulting in a $T^{n-m}$; additionally, each wheel can move left and right in the strip, crossing over or under each other in any order, and we take the average of all possible crossings in the same direction. 
Any $m-1$ wheels can have the same $x$-coordinate (since they contain at most $w$ disks total), but all $m$ cannot (this is not true for trivial averaged-filters). 
This movement can be thought of as an average of embedded copies of $S^{m-2}$ in configuration space.
As a result, the whole $(n-2)$-cycle arises from an average of embedded copies of $S^{m-2}\times T^{n-m}$. 
See Figure \ref{qinclusion} for an example of an averaged-filter on $3$ wheels each consisting of $1$ disk.

We write $AF(W_{1},\dots, W_{m})$ for the averaged-filter on the ordered set of wheels $W_{1}, \dots, W_{m}$.
Here the orientation of the corresponding class represented by the embedded $(\sum S^{m-2})\times T^{n-m}$ is determined by the order of the wheels. We write $AF\big(W(i_{1, 1}, \dots, i_{1, n_{1}}),\dots, W(i_{m, 1}, \dots, i_{m, n_{m}})\big)$ for the proper averaged-filter on the proper wheels $W(i_{1, 1}, \dots, i_{1, n_{1}}),\dots, W(i_{m, 1}, \dots, i_{m, n_{m}})$.
It follows from the definition of $q$ that
\[
AF(W_{1},\dots, W_{m})=\frac{1}{|A|!}\sum_{\sigma\in S_{A}} \text{wsgn}(\sigma)F(W_{\sigma(1)},\dots, W_{\sigma(m)}),
\]
hence an averaged-filter is an average of filters.
As a result, the averaged-filter on two wheels $AF(W_{1}, W_{2})$ is equal to the filter $F(W_{1}, W_{2})$.
Alternatively, one could note that the complexes $P\big(\{W_{1}, W_{2}\}, \{n_{1}, n_{2}\}, w\big)$ and $\text{cell}\big(\{W_{1}, W_{2}\}, \{n_{1}, n_{2}\}, w\big)$ have identical cellular structures, so $i_{id}=q$ and $AF(W_{1}, W_{2})=F(W_{1}, W_{2})$.

Unlike filters where the order of the wheels determines both the filter and its orientation, the order of the wheels in an averaged-filter only determine its orientation.
Namely, for the averaged-filter $AF(W_{1},\dots, W_{m})$, the symmetric group $S_{m}$ determines the orientation of the average of the spheres by an action of weighted sign, as well as determining the order in which we consider the tori.
The following proposition shows we can rearrange the wheels in an averaged-filter so the wheels are in increasing order of largest labeled disk or in order of increasing rank at the cost of a sign change.

\begin{prop}\label{sign of averaged-filter}
Let $AF(W_{1}, \dots, W_{m})$ be an averaged-filter, such that $W_{i}$ is a disk on $n_{i}$ wheels. Then, the weighted permutation $\sigma$ of the weighted set $\big(\{W_{1}, \dots, W_{m}\}, \{n_{1}, \dots, n_{m}\}\big)$ acts by $\text{wsgn}(\sigma)$ on $AF(W_{1}, \dots, W_{m})$.
\end{prop}

\begin{proof}
The averaged-filter $AF(W_{1}, \dots, W_{m})$ arises from the cycle that is the boundary of $W_{1}\,\cdots\,W_{m}$, the top dimensional cell of the weighted permutohedron $P(A,\mathcal{W})=P\big(\{W_{1},\dots, W_{m}\}, \{n_{1}, \dots, n_{m}\}\big)$.
The codimenison-1 cells of $P(A,\mathcal{W}, w)$ are of the form $W_{i_{1}}\cdots W_{i_{l}}|W_{i_{l+1}}\cdots W_{i_{m}}$ and every codimension-1 cell appears in the boundary of $W_{1}\,\cdots\,W_{m}$. 
The sign of the cell $W_{i_{1}}\cdots W_{i_{l}}|W_{i_{l+1}}\cdots W_{i_{m}}$ in $P(A,\mathcal{W}, w)$ is 
\[
(-1)^{n_{i_{1}}+\cdots+n_{i_{l}}}\cdot\text{wsgn}(\text{permutation } W_{1}\cdots W_{m}\mapsto W_{i_{1}}\cdots W_{i_{l}}|W_{i_{l+1}}\cdots W_{i_{m}}),
\]
and the sign of $\sigma(W_{i_{1}}\cdots W_{i_{l}}|W_{i_{l+1}}\cdots W_{i_{m}})$ in $(A,\mathcal{W}, w)$ is 
\[
(-1)^{n_{i_{1}}+\cdots+n_{i_{l}}}\cdot\text{wsgn}(\text{permutation } \sigma^{-1}(W_{1}\cdots W_{m})\mapsto W_{i_{1}}\cdots W_{i_{l}}|W_{i_{l+1}}\cdots W_{i_{m}})
\]
\[
=(-1)^{n_{i_{1}}+\cdots+n_{i_{l}}}\cdot\text{wsgn}(\text{permutation } W_{1}\cdots W_{m}\mapsto W_{i_{1}}\cdots W_{i_{l}}|W_{i_{l+1}}\cdots W_{i_{m}})\cdot \text{wsgn}(\sigma).
\]
Thus, they differ by $\text{wsgn}(\sigma)$. 
Therefore $\sigma$ acts on $AF(W_{1}, \dots, W_{m})$ by $\text{wsign}(\sigma)$.

\end{proof}

We will show that we can replace proper filters with proper averaged-filters in a basis for rational-homology. 
First, we show that every averaged-filter can be written as a sum of proper averaged-filters.

\begin{prop}\label{replace filter with sum of proper proper}
The averaged-filter $AF(W_{1}, \dots, W_{m})$ can be written as a sum of proper averaged-filters of the form $AF\big(W(i_{1, 1}, \dots, i_{1, n_{1}}),\dots, W(i_{m, 1}, \dots, i_{m, n_{m}})\big)$, where, for all $k$, the labels of the disks in the wheel $W_{k}$ are $i_{k, 1},\dots, i_{k, n_{k}}$, and $i_{k,1}>i_{k, l}$ for all $l\neq 1$.
\end{prop}

\begin{proof}
The averaged-filter $AF(W_{1}, \dots, W_{m})$ from the cycle $Z\in C_{m-2}\Big(P\big(\{W_{1},\dots, W_{m}\}, \{n_{1}, \dots, n_{m}\}, w\big);\Q\Big)$ that is the boundary of the cell $W_{1}\,\cdots\,W_{m}$ in $P\big(\{W_{1},\dots, W_{m}\}, \{n_{1}, \dots, n_{m}\}\big)$, the map $q$, and a sequence of $\text{spin}$-maps such that the wheel $W_{k}$ has labels $i_{k, 1}, \dots, i_{k, n_{k}}$. 

By Proposition \ref{properwheels} it follows that the sequence of $\text{spin}$-maps that produce $W_{k}$ is equal to a sum of $\text{spin}$-maps that all produce proper wheels on the same set of labels and the biggest label comes first, i.e., if $\text{spin}_{W_{k}}$ is the spin-map that yields the wheel $W_{k}$ then 
\[
\text{spin}_{W_{k}}=\sum a_{i_{k, 1}, \dots, i_{k, n_{k}}} \text{spin}_{W(i_{k, 1}, \dots, i_{k, n_{k}})}
\]
where the maps $\text{spin}_{W(i_{k, 1}, \dots, i_{k, n_{k}})}$ yield proper wheels on the same set of labels as $\text{spin}_{W_{k}}$, and $i_{k,1}>i_{k,l}$ for all $l\neq 1$. 
By Proposition \ref{spinorder} the order of these maps does not matter, and since these are chain maps it follows that 
\begin{multline*}
\text{spin}_{W_{1}}\circ\cdots \circ \text{spin}_{W_{m}}\\
=\sum a_{i_{1, 1}, \dots, i_{1, n_{1}}} \text{spin}_{W(i_{1, 1}, \dots, i_{1, n_{1}})}\circ\cdots\circ \sum a_{i_{m, 1}, \dots, i_{m, n_{m}}} \text{spin}_{W(i_{m, 1}, \dots, i_{m, n_{m}})}\\
=\sum a_{i_{1, 1}, \dots, i_{1, j(1)}}\cdots \sum a_{i_{m, 1}, \dots, i_{m, n_{m}}} \text{spin}_{W(i_{1, 1}, \dots, i_{1, n_{1}})}\circ\cdots\circ  \text{spin}_{W(i_{m, 1}, \dots, i_{m, n_{k}})}.
\end{multline*}
Note that
\[
\text{spin}_{W(i_{1, 1}, \dots, i_{1, n_{k}})}\circ\cdots\circ  \text{spin}_{W(i_{m, 1}, \dots, i_{m, n_{k}})}\circ q(Z)
\]
is the proper averaged-filter on the wheels $W(i_{1, 1}, \dots, i_{1, n_{1}}),\dots, W(i_{m, 1}, \dots, i_{m, n_{m}})$, where $W_{k}$ has labels $i_{k, 1},\dots, i_{k, n_{k}}$. 
Therefore, 
\[
AF(W_{1}, \dots, W_{m})=\sum a_{i_{1, 1}, \dots, i_{1, n_{1}}}\cdots \sum a_{i_{m, 1}, \dots, i_{m, n_{m}}}  AF\big(W(i_{1, 1}, \dots, i_{1, n_{1}}),\dots, W(i_{m, 1}, \dots, i_{m, n_{m}})\big),
\]
where the labels of the disks in the wheel $W_{k}$ are $i_{k, 1},\dots, i_{k, n_{k}}$ for all $k$, and $i_{k,1}>i_{k,l}$ for all $l\neq 1$.
\end{proof}

Next, we prove a technical proposition that describes how one can write a proper averaged-filter as a sum of elements in the basis of Theorem \ref{AMthmB}.
This will allow us to replace filters with averaged-filters.

\begin{prop}\label{newfilterfromold}
For $m\ge 3$, let $AF\big(W(i_{1, 1}, \dots, i_{1, n_{1}}),\dots, W(i_{m, 1}, \dots, i_{m, n_{m}})\big)$ be a non-trivial proper averaged-filter on $n=\sum_{j=1}^{m}n_{j}$ disks, such that for all $k$, we have $i_{k,1}>i_{k,l}$ for $l\neq 1$ and $i_{k,1}<i_{k+1, 1}$.
Then,
\begin{multline*}
AF\big(W(i_{1, 1}, \dots, i_{1, n_{1}}),\dots, W(i_{m, 1}, \dots, i_{m, n_{m}})\big)\\
=F\big(W(i_{1, 1}, \dots, i_{1, n_{1}}),\dots, W(i_{m, 1}, \dots, i_{m, n_{m}})\big)+\sum_{2\le l<m} a_{F(W_{1},\dots, W_{l})}F(W_{1},\dots, W_{l})
\end{multline*}
Moreover, $F\big(W(i_{1, 1}, \dots, i_{1, n_{1}}),\dots, W(i_{m, 1}, \dots, i_{m, n_{m}})\big)$ and the $F(W_{1},\dots, W_{l})$ are filters in the basis for $H_{n-2}\big(\text{cell}(n,w);\Z\big)$ of Theorem \ref{AMthmB}.
\end{prop}

\begin{proof}
Let $\sigma$ be the permutation $i_{1, 1}\, \cdots\, i_{1, n_{1}}\,\cdots\,i_{m, 1}\, \cdots\,i_{m, n_{m}}$, then $m=n-\#\sigma$, and the wheels of $\sigma$ are $i_{1, 1}\,\cdots\,i_{1, n_{1}},\dots, i_{m, 1}\, \cdots\, i_{m, n_{m}}$.
Therefore, $\Big(\big\{W(i_{1, 1}, \dots, i_{1, n_{1}}),\dots, W(i_{m, 1}, \dots, i_{m, n_{m}})\big\}, \{n_{1}, \dots, n_{m}\}\Big)=\big(n-\#\sigma, \mathcal{W}(\sigma)\big)$.
By definition, the averaged-filter $AF\big(W(i_{1, 1}, \dots, i_{1, n_{1}}),\dots, W(i_{m, 1}, \dots, i_{m, n_{m}})\big)$ is the image of the cycle $Z\in C_{m-2}\Big(P\big(n-\#\sigma, \mathcal{W}(\sigma), w\big);\Q\Big)$ that is the boundary of the top dimensional cell $W(i_{1, 1}, \dots, i_{1, n_{1}})\,\cdots\,W(i_{m, 1}, \dots, i_{m, n_{m}})$ in $P\big(n-\#\sigma, \mathcal{W}(\sigma)\big)$ under the maps $q$ and $\text{spin}_{\sigma}$, i.e., 
\[
AF\big(W(i_{1, 1}, \dots, i_{1, n_{1}}),\dots, W(i_{m, 1}, \dots, i_{m, n_{m}})\big)=\text{spin}_{\sigma}\circ q(Z).
\]

Since the averaged-filter is non-trivial it follows that for $\tau\in S(\sigma)$, the weighted permutohedron $P\big(n-\#\tau, \mathcal{W}(\tau), w\big)$ is homotopy equivalent to $S^{n-\#\tau-2}$.
Proposition \ref{AMdecomp2} proves that
\[
H_{*}\Big(\text{cell}\big(n-\#\sigma, \mathcal{W}(\sigma), w\big);\Q\Big)\cong \bigoplus_{\tau\in S(\sigma)}H_{*-\#\tau}\Big(P\big(n-\#\tau, \mathcal{W}(\tau), w\big);\Q\Big).
\]
Setting $*=m-2$, i.e., $2$ less than the number of wheels of $\sigma$, we have
\[
H_{m-2}\Big(\text{cell}\big(n-\#\sigma, \mathcal{W}(\sigma), w\big);\Q\Big)\cong \bigoplus_{\tau\in S(\sigma)}H_{m-\#\tau-2}\Big(P\big(n-\#\tau, \mathcal{W}(\tau), w\big);\Q\Big).
\]
Note that $m-\#\tau-2$ is $2$ less than the number of wheels of $\tau$.
As $q(Z)$ can be viewed as a class in $H_{m-2}\Big(\text{cell}\big(n-\#\sigma, \mathcal{W}(\sigma), w\big);\Q\Big)$, Proposition \ref{AMdecomp2} implies that
\[
q(Z)=\sum_{\tau\in S(\sigma)} a_{\tau}\text{spin}_{\tau, \sigma}\big(i_{id}(Z_{\tau})\big),
\]
where $Z_{\rho}\in C_{n-\#\rho-2}\Big(P\big(n-\#\rho, \mathcal{W}(\rho), w\big);\Q\Big)$ is the boundary of top dimensional cell of $P\big(n-\#\rho, \mathcal{W}(\rho)\big)$ if $n-\rho\ge 3$, and if $n-\rho=2$, then $Z_{\rho}$ is a sum of the two $0$-cells.
Since the $P\big(n-\#\sigma, \mathcal{W}(\sigma)\big)$ are spheres, it follows that the images of these classes under $i_{id}$ and the spin-maps generate $H_{m-2}\Big(\text{cell}\big(n-\#\sigma, \mathcal{W}(\sigma), w\big);\Q\Big)$.

We rewrite this as
\[
q(Z)=a_{\sigma}i_{id}(Z_{\sigma})+ \sum_{\tau\in S(\sigma)|\tau\neq\sigma} a_{\tau}\text{spin}_{\tau, \sigma}\big(i_{id}(Z_{\tau})\big),
\]
so 
\[
\text{spin}_{\sigma}\circ q(Z)=a_{\sigma}\text{spin}_{\sigma}\circ i_{id}(Z_{\sigma})+\sum_{\tau\in S(\sigma)|\tau\neq\sigma} a_{\tau}\text{spin}_{\sigma}\circ\text{spin}_{\tau, \sigma}\big(i_{id}(Z_{\tau})\big).
\]

If $\tau\in S(\sigma)$ has $1$ wheel, then $Z_{\tau}=0$, and we can ignore the corresponding terms.

Otherwise, $\tau\in S(\sigma)$ has between $2$ and $m$ wheels,
as $\sigma\in S(\sigma)$ has the most wheels, namely $m$ wheels, and the other elements of $S(\sigma)$ have strictly fewer wheels.
It follows from the definitions of $\text{spin}_{\sigma}$ and $\text{spin}_{\tau, \sigma}$ that homology class represented by $\text{spin}_{\sigma}\circ\text{spin}_{\tau, \sigma}\big(i_{id}(Z_{\sigma})\big)$ is of the form $F(W_{1}, \dots, W_{l})$, where the disks in $W_{j}$ are the same as the disks in the $j^{\text{th}}$-wheel of $\tau$.
By assumption the wheels in $F(W_{1}, \dots, W_{l})$ are in increasing order of largest label, so by Proposition \ref{filter on m wheels is a proper filter on m wheels} it follows that we can write $\text{spin}_{\sigma}\circ\text{spin}_{\tau, \sigma}\big(i_{id}(Z_{\tau})\big)$ as a sum of filters on $l$ wheels in the basis of Theorem \ref{AMthmB}.
Additionally, if $\tau=\sigma$, then it follows from the definition of $\text{spin}_{\sigma}$ that homology class represented by $\text{spin}_{\sigma}\circ\text{spin}_{\tau, \sigma}\big(i_{id}(Z_{\tau})\big)$ is $F\big(W(i_{1, 1}, \dots, i_{1, n_{1}}),\dots, W(i_{m, 1}, \dots, i_{m, n_{m}})\big)$.

Since $AF\big(W(i_{1, 1}, \dots, i_{1, n_{1}}),\dots, W(i_{m, 1}, \dots, i_{m, n_{m}})\big)$ is null-homologous for large enough $w$, it follows that the terms in the right hand side of the sum are null-homologous. 
Products of wheels in the basis from Theorem \ref{AMthmB} are never null-homologous while filters always are for large enough $w$. Therefore,
\begin{multline*}
AF\big(W(i_{1, 1}, \dots, i_{1, n_{1}}),\dots, W(i_{m, 1}, \dots, i_{m, n_{m}})\big)\\
=a_{\sigma}F\big(W(i_{1, 1}, \dots, i_{1, n_{1}}),\dots, W(i_{m, 1}, \dots, i_{m, n_{m}})\big)+\sum_{2\le l<m} a_{F(W_{1},\dots, W_{l})}F(W_{1},\dots, W_{l})
\end{multline*}
where all the filters are in the basis of Theorem \ref{AMthmB}. 


Next, we check that $a_{\sigma}=1$. 
Let $p$ be the map
\[
p: C_{*}\Big(\text{cell}\big(n-\#\sigma,\mathcal{W}(\sigma), w\big);\Q\Big)\to C_{*}\Big(P\big(n-\#\sigma,\mathcal{W}(\sigma), w\big);\Q\Big)
\]
that sends a cell in $\text{cell}\big(n-\#\sigma,\mathcal{W}(\sigma), w\big)$ to the cell in $P\big(n-\#\tau,\mathcal{W}(\sigma), w\big)$ that arises from forgetting the ordering of the blocks, while multiplying the sign of the cell by the weighted sign of the weighted permutation that rearranges the blocks, see Figure \ref{spintoprojection}.

\begin{figure}[h]
\centering
\captionsetup{width=.8\linewidth}
\includegraphics[width = 12cm]{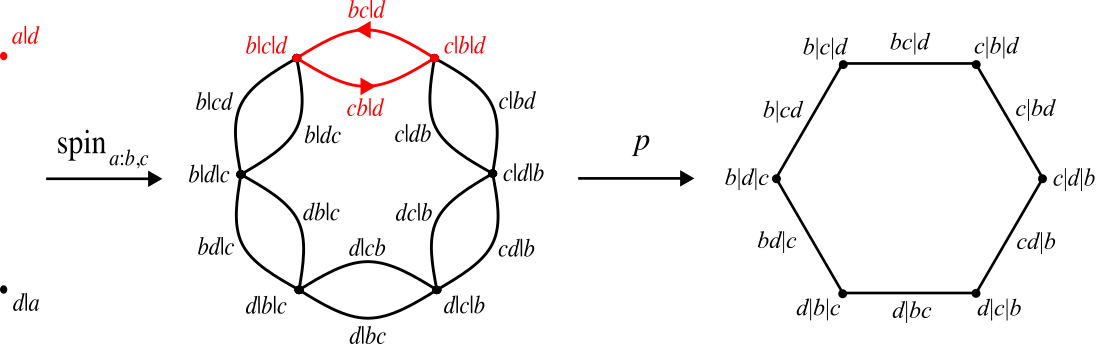}
\caption{The image of the cycle $a|d$ in $C_{0}\Big(\text{P}\big(\{a,d\}, \{1, 2\}, 2\big);\Q\Big)$ under the composition of the maps $\text{spin}_{a:b,c}$ and $p$.
The cells $bc|d$ and $cb|d$ point in different directions due to the nature of the spin-map, and, as a result, their image $bc|d$ in $C_{1}\Big(\text{P}\big(\{b,c,d\}, \{1, 1,1\}, 2\big);\Q\Big)$ has coefficient $0$, i.e., is the trivial cycle.
}
\label{spintoprojection}
\end{figure}

Any element in $H_{m-2}\Big(\text{cell}\big(n-\#\sigma, \mathcal{W}(\sigma), w\big);\Q\Big)$ arising from an application of a $\text{spin}$-map goes to $0$ under $p_{*}$ as the $\text{spin}$-map is arises from the doubling maps that produce a pair cells that project to the same cell with the opposite signs. 
Applying $p$ to the standard cycle representing the homology class $AF\big(W(i_{1, 1}, \dots, i_{1, n_{1}}),\dots, W(i_{m, 1}, \dots, i_{m, n_{m}})\big)$, we see that
\[
p\Big(\text{spin}_{\sigma}\big(i_{id}(Z_{\tau})\big)\Big)=0,
\]
for $\tau\neq \sigma$, but 
\[
p\big(i_{id}(Z_{\sigma})\big)=Z_{\sigma}.
\]
By the definition of $q$ it follows that
\[
p\big(q(Z)\big)=Z_{\sigma},
\]
so $a_{\sigma}=1$, see Figure \ref{qinclusionpprojection}.

\begin{figure}[h]
\centering
\captionsetup{width=.8\linewidth}
\includegraphics[width = 16cm]{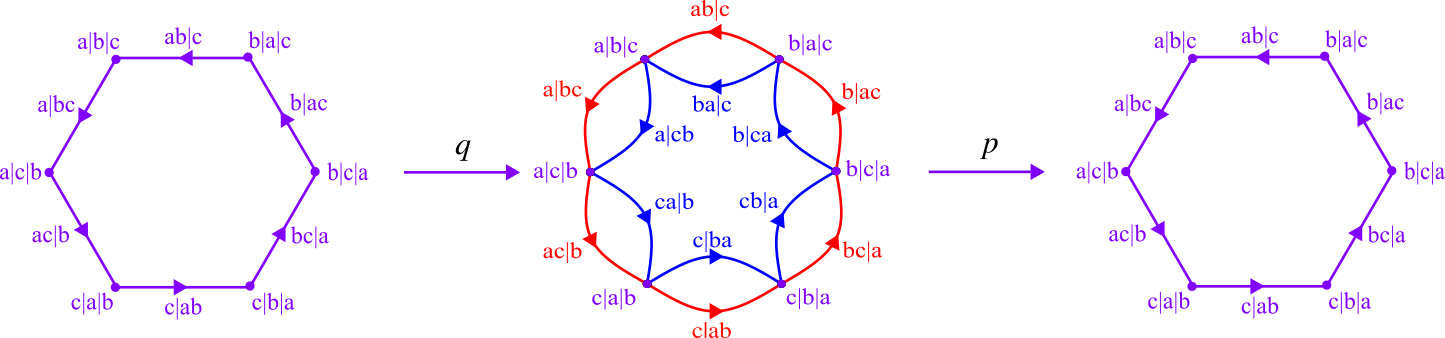}
\caption{The image of the cycle $\partial(abc)$ in $C_{1}\Big(\text{P}\big(\{a,d\}, \{1, 2\}, 2\big);\Q\Big)$ under the composition of the maps $q$ and $p$.
The image under $q$ is the average of red cells and the blue cells.
Since the $1$-cells in $C_{1}\Big(\text{cell}\big(\{a,b,c\}, \{1, 1,1\}, 2\big);\Q\Big)$ with the same boundary have coefficients pointing them in the same direction the resulting cycle is non-trivial after applying $p$.
}
\label{qinclusionpprojection}
\end{figure}
\end{proof}

\begin{exam}
In $H_{1}\big(\text{conf}(3,2);\Q\big)$ we have that
\[
AF\big(W(1),W(2),W(3)\big)=F\big(W(1),W(2),W(3)\big)-F\big(W(1), W(3,2)\big)-F\big(W(2), W(3,1)\big)-F\big(W(1,2),W(3)\big).
\]
See Figures \ref{qinclusion} and \ref{gluepermutadratogetcell}.
\end{exam}

Next, we prove that proper averaged-filters are well-behaved with respect to the action of the symmetric group.
This fact is critical in proving our representation stability results, as it will ensure that a certain number is well-defined, unlike its integral analogue.

\begin{prop}\label{snactonavfilter}
Let  $AF\big(W(i_{1, 1}, \dots, i_{1, n_{1}}),\dots, W(i_{m, 1}, \dots, i_{m, n_{m}})\big)$ be a proper averaged-filter, and set $n=n_{1}+\cdots+n_{m}$. 
Then, for any $\sigma \in S_{n}$, $\sigma\Big(AF\big(W(i_{1, 1}, \dots, i_{1, n_{1}}),\dots, W(i_{m, 1}, \dots, i_{m, n_{m}})\big)\Big)$ is a linear combination of proper averaged-filters of the form $AF\big(W(j_{1, 1}, \dots, j_{1, n_{1}}),\dots, W(j_{m, 1}, \dots, j_{m, n_{m}})\big)$, where, for all $k$, the labels $\{\sigma(i_{k, 1}),\dots, \sigma(i_{k, n_{k}})\}$ and $\{j_{k, 1},\dots, j_{k, n_{k}}\}$ are the same, and  $j_{k, 1}>j_{k, l}$ for all $l\neq1$.
\end{prop}

\begin{proof}
An averaged-filter 
\[
AF\big(W(i_{1, 1}, \dots, i_{1, n_{1}}),\dots, W(i_{m, 1}, \dots, i_{m, n_{m}})\big)
\]
is the image of the boundary of the cell $W(i_{1, 1}, \dots, i_{1, n_{1}})\,\cdots\,W(i_{m, 1}, \dots, i_{m, n_{m}})$ in the weighted permutohedron $P\Big(\big\{W(i_{1, 1}, \dots, i_{1, n_{1}}),\dots, W(i_{m, 1}, \dots, i_{m, n_{m}})\big\},\{n_{1},\dots, n_{m}\}\Big)$, under $q$ and the maps $\text{spin}_{W(i_{1, 1}, \dots, i_{1, n_{1}})}$ $\dots$, $\text{spin}_{W(i_{m, 1}, \dots, i_{m, n_{m}})}$. 
It follows that
\[
\sigma\Big(AF\big(W(i_{1, 1}, \dots, i_{1, n_{1}}),\dots, W(i_{m, 1}, \dots, i_{m, n_{m}})\big)\Big)
\] 
is the image of the boundary of the cell $W\big(\sigma(i_{1, 1}), \dots, \sigma(i_{1, n_{1}})\big)\,\cdots\,W\big(\sigma(i_{m, 1}), \dots, \sigma(i_{m, n_{m}})\big)$ in the weighted permutohedron $P\Big(\big\{W\big(\sigma(i_{1, 1}), \dots, \sigma(i_{1, n_{1}})\big),\dots, W\big(\sigma(i_{m, 1}), \dots, \sigma(i_{m, n_{m}}\big))\big\},\{n_{1},\dots, n_{m}\}\Big)$, under $q$ and the spin-maps $\text{spin}_{W\big(\sigma(i_{1, 1}), \dots, \sigma(i_{1, n_{1}})\big)}$, $\dots$, $\text{spin}_{W\big(\sigma(i_{m, 1}), \dots, \sigma(i_{m, n_{m}})\big)}$, i.e.,
\[
AF\Big(W\big(\sigma(i_{1, 1}), \dots, \sigma(i_{1, n_{1}})\big),\dots, W\big(\sigma(i_{m, 1}), \dots, \sigma(i_{m, n_{m}})\big)\Big).
\]

By Proposition \ref{spin relation}, we can write
\[
\text{spin}_{W\big(\sigma(i_{1, 1}), \dots, \sigma(i_{1, n_{1}})\big)}=\sum a_{j_{k, 1}, \dots, j_{k, n_{k}}} \text{spin}_{W(j_{k, 1}, \dots, j_{k, n_{k}})}
\]
where the wheels in the sum are the proper wheels on the labels $\sigma(i_{k, 1}), \dots, \sigma(i_{k, n_{k}})$ such that the largest label comes first. 

It follows that
\begin{multline*}
\sigma\Big(AF\big(W(i_{1, 1}, \dots, i_{1, n_{1}}),\dots, W(i_{m, 1}, \dots, i_{m, n_{m}})\big)\Big)\\
=AF\big(\sum a_{j_{1, 1}, \dots, l_{1,n_{1}}} W(j_{1, 1}, \dots, l_{1,n_{1}}), \cdots, \sum a_{j_{m, 1}, \dots, l_{m, n_{m}}} W(j_{m, 1}, \dots, l_{m, n_{m}})\big)\\
=\sum a_{j_{1, 1}, \dots, l_{1,n_{1}}} \cdots a_{j_{m, 1}, \dots, l_{m, n_{m}}} AF\big(W(j_{1, 1}, \dots, j_{1, n_{1}}),\dots, W(j_{m, 1}, \dots, j_{m, n_{m}})\big),
\end{multline*}
as desired.
\end{proof}

By Propositions \ref{sign of averaged-filter} and \ref{snactonavfilter} every averaged-filter is a sum of proper averaged-filters such that the largest label in each wheel comes first and the wheels are ordered in increasing order by largest labeled disk.

In order to replace filters with averaged-filter we give an ordering to the elements of the basis for $H_{n-2}\big(\text{cell}(n, w);\Q\big)$ from Theorem \ref{AMthmB}.

\begin{defn}
We say that a filter is more \emph{complex} than another if it is composed of more wheels, and that the concatenation product of two wheels is less complex than any filter. The \emph{complexity} of a filter is the number of wheels in it.
\end{defn}

Complexity gives a partial ordering to the basis for $H_{n-2}\big(\text{cell}(n, w);\Q\big)$ from Theorem \ref{AMthmB} as a filter one $n$ disks has homological degree $n-2$.
We arbitrarily extend this ordering to a total ordering on the entire basis for $H_{n-2}\big(\text{cell}(n, w);\Q\big)$.
This allows us to replace filters with averaged-filters in the simplest case.
Later, this will allow us to replace all filters on at least $3$ wheels with averaged-filters and rewrite any element of homology without using filters.

\begin{prop}\label{replacefilter}
For $w<n\le \frac{3w}{2}$ the homology group $H_{n-2}\big(\text{conf}(n,w);\Q\big)$ has a basis consisting of concatenations of proper wheels and non-trivial proper averaged-filters. 
We say one wheel ranks above another if it has more disks, or has the same number of disks and its largest label is greater. A cycle is in the basis if and only if:
\begin{enumerate}
\item The largest labeled disk in each proper wheel (regardless of whether or not it is in a (averaged-) filter) comes first.
\item The largest labeled disk in each wheel in an averaged-filter comes first and the wheels inside each averaged-filter on at least $3$ wheels are in ascending order by largest label (regardless of the number of disks). 
\item In an averaged-filter on $2$ wheels the wheels are ordered by increasing rank.
\item Adjacent wheels not inside an averaged-filter are ordered from higher to lower rank.
\end{enumerate}
\end{prop}

\begin{proof}
By Theorem \ref{AMthmB} this proposition holds with ``averaged-filter" replaced with ``filter." Thus, it suffices to show that we can replace the filters on more than $3$ wheels with averaged-filters.

By Proposition \ref{newfilterfromold} we can write a proper averaged-filter on $m$ wheels in the proposed basis as the sum of the proper filter on the same set of wheels in the basis of Theorem \ref{AMthmB} and filters on fewer than $m$ wheels in the basis of Theorem \ref{AMthmB}, i.e.,
\begin{multline*}
AF\big(W(i_{1,1}, \dots, i_{1, n_{1}}), \dots, W(i_{m, 1}, \dots, i_{m, n_{m}})\big)\\
=F\big(W(i_{1,1}, \dots, i_{1, n_{1}}), \dots, W(i_{m,1}, \dots, i_{m, n_{m}})\big)+\sum_{l<m} a_{F(W_{1},\dots, W_{l})}F(W_{1},\dots, W_{l}).
\end{multline*}

Note that there is a clear bijection between the basis elements of $H_{n-2}\big(\text{conf}(n,w);\Q\big)$ from Theorem \ref{AMthmB} and the basis proposed in the proposition that sends concatenation products of wheels and proper filters on 2 wheels to themselves and the proper filter on more than 2 wheels to the proper averaged-filter on the same set of wheels. 
Extend the complexity ordering on the basis elements of $H_{n-2}\big(\text{conf}(n,w);\Q\big)$ from Theorem \ref{AMthmB} to a total ordering, and give this ordering to the set proposed in the proposition via the bijection described above. 
Given this pair of orderings it follows that the matrix that takes the proposed basis to the basis of Theorem \ref{AMthmB} is upper triangular with $1$s on the diagonal. 
Therefore, this matrix is invertible. Since $H_{n-2}\big(\text{conf}(n,w);\Q\big)$ is finite dimensional, it follows that the proposed basis is in fact a basis for $H_{n-2}\big(\text{conf}(n,w);\Q\big)$.
\end{proof}

Next, we prove the existence of a family of relations that relate concatenation products of wheels and averaged-filters.
These relations play an important role in the homology of the ordered configuration space of unit-diameter disks on infinite strips, and provide limits for how many orders of representation stability $\text{conf}(\bullet,w)$ exhibits.

\begin{lem}\label{relationlem}
Given $m+1\ge 3$ wheels $W(i_{1, 1},\dots, i_{1, n_{1}}),\dots, W(i_{m+1,1},\dots, i_{m+1,n_{l+1}})$ such that $\sum_{i=1}^{m+1}n_{i}=n$ and such that any sum of $m-1$ of the $n_{i}$ is at most $w$, then there is a relation of the form
\begin{multline*}
\sum_{k=1}^{m+1} \pm W(i_{k,1}, \dots, i_{k,n_{k}})|AF\big(W(i_{1,1}, \dots, i_{1,n_{1}}),\dots, \widehat{W(i_{k,1}, \dots, i_{k,n_{k}})}, \dots, W(i_{m+1,1}, \dots, i_{m+1,n_{m+1}})\big)\\
=\sum_{k=1}^{m+1} \pm AF\big(W(i_{1,1}, \dots, i_{1,n_{1}}),\dots, \widehat{W(i_{k,1}, \dots, i_{k,n_{k}})}, \dots, W(i_{m+1,1}, \dots, i_{m+1,n_{m+1}})\big)|W(i_{k,1}, \dots, i_{k,n_{k}}),
\end{multline*}
in $H_{n-3}\big(\text{conf}(n, w);\Q\big)$, where the signs are determined by the size of the wheels.
\end{lem}

\begin{proof}
We prove this by considering cycles in $C_{m-2}\Big(P\big(\{W_{1}, \dots, W_{m+1}\}, \{n_{1}, \dots, n_{m+1}\}, w\big);\Q\Big)$ as the images of these cycles in $H_{n-3}\big(\text{conf}(n, w);\Q\big)$ are the classes in question,
see Figure \ref{permutohedrarelation} for visualization of an example of this relation at the level of weighted permutohedra.
The top dimensional cell $W_{1}\, \cdots\, W_{m+1}$ is the only $m$-cell in $P\big(\{W_{1}, \dots, W_{m+1}\}, \{n_{1}, \dots, n_{m+1}\}\big)$.
Since $\partial$ is a boundary map it follows that $\partial^{2}(W_{1}\, \cdots\, W_{m+1})=0$ in $C_{m-2}\Big(P\big(\{W_{1}, \dots, W_{m+1}\}, \{n_{1}, \dots, n_{m+1}\}\big);\Q\Big)$.
By our assumption on the sizes of the wheels with respect to $w$, every $(m-2)$-cell in $P\big(\{W_{1}, \dots, W_{m+1}\}, \{n_{1}, \dots, n_{m+1}\}\big)$ is in the restriction $P\big(\{W_{1}, \dots, W_{m+1}\}, \{n_{1}, \dots, n_{m+1}\}, w\big)$, so this holds in $C_{m-2}\Big(P\big(\{W_{1}, \dots, W_{m+1}\}, \{n_{1}, \dots, n_{m+1}\}, w\big);\Q\Big)$.

By the nature of permutohedra, every codimension-1 cell in $P\big(\{W_{1}, \dots, W_{m+1}\}, \{n_{1}, \dots, n_{m+1}\}\big)$ appears with multiplicity $\pm 1$ in boundary of the top dimensional cell. 
It follows that 
\[
\partial^{2}(W_{1}\, \cdots\, W_{m+1})=\sum_{j=1}^{m}\sum_{\substack{i_{1}<\cdots<i_{j}\\i_{j+1}<\cdots< i_{m+1}}}\pm\partial(W_{i_{1}}\,\cdots\, W_{i_{j}}|W_{i_{j+1}}\,\cdots\,W_{i_{m+1}})=0
\]
in $C_{m-2}\Big(P\big(\{W_{1}, \dots, W_{m+1}\}, \{n_{1}, \dots, n_{m+1}\},w\big);\Q\Big)$.

Not all $(m-1)$-cells in $P\big(\{W_{1}, \dots, W_{m+1}\}, \{n_{1}, \dots, n_{m+1}\}\big)$ are in $P\big(\{W_{1}, \dots, W_{m+1}\}, \{n_{1}, \dots, n_{m+1}\}, w\big)$;
namely, the cells of the form $W_{k}|W_{1}\,\cdots\,\widehat{W_{k}}\,\cdots\, W_{m+1}$ and $W_{1}\,\cdots\,\widehat{W_{k}}\,\cdots\, W_{m+1}|W_{k}$ might not be in the restriction $P\big(\{W_{1}, \dots, W_{m+1}\}, \{n_{1}, \dots, n_{m+1}\}, w\big)$, though every other $(m-1)$-cell is.
Moreover, for all $k$, 
\[
\partial (W_{k}|W_{1}\,\cdots\,\widehat{W_{k}}\,\cdots\, W_{m+1})=W_{k}|\partial(W_{1}\,\cdots\,\widehat{W_{k}}\,\cdots\, W_{m+1}),
\]
and 
\[
\partial (W_{1}\,\cdots\,\widehat{W_{k}}\,\cdots\, W_{m+1}|W_{k})=\partial(W_{1}\,\cdots\,\widehat{W_{k}}\,\cdots\, W_{m+1})|W_{k}.
\]
Therefore,
\begin{multline*}
\partial^{2}(W_{1}\, \cdots\, W_{m+1})=\sum_{k=1}^{m+1}\pm W_{k}|\partial(W_{1}\,\cdots\,\widehat{W_{k}}\,\cdots\, W_{m+1})+\sum_{k=1}^{m+1}\pm\partial (W_{1}\,\cdots\,\widehat{W_{k}}\,\cdots\, W_{m+1})|W_{k}\\
+\sum_{j=2}^{m+1}\sum_{\substack{i_{1}<\cdots<i_{j}\\i_{j+1}<\cdots< i_{m+1}}}\pm\partial (W_{i_{1}}\,\cdots\, W_{i_{j}}|W_{i_{j+1}}\,\cdots\,W_{i_{m+1}}).
\end{multline*}

Note that $\partial(W_{1}\,\cdots\,\widehat{W_{k}}\,\cdots\, W_{m+1})\in C_{m-2}\Big(P\big(\{W_{1},\dots,\widehat{W_{k}},\dots, W_{m+1}\}, \{n_{1},\dots, \widehat{n_{k}}, \dots, n_{m+1}\}, w\big);\Q\Big)$ is the cycle that generates the averaged-filter 
\[
AF\big(W(i_{1,1}, \dots, i_{1,n_{1}}),\dots, \widehat{W(i_{k,1}, \dots, i_{k,n_{k}})}, \dots, W(i_{m+1,1}, \dots, i_{m+1,n_{m+1}})\big).
\]
It follows that the image of the cycle $W_{k}|\partial(W_{1}\,\cdots\,\widehat{W_{k}}\,\cdots\, W_{m+1})$ in $H_{n-3}\big(\text{conf}(n, w);\Q\big)$ under the composition $\text{spin}_{W_{k}}\circ \text{spin}_{W_{1}}\circ\cdots\circ\text{spin}_{W_{m+1}}\circ q$ is 
\[
W(i_{k,1}, \dots, i_{k,n_{k}})|AF\big(W(i_{1,1}, \dots, i_{1,n_{1}}),\dots, \widehat{W(i_{k,1}, \dots, i_{k,n_{k}})}, \dots, W(i_{m+1,1}, \dots, i_{m+1,n_{m+1}})\big).
\]
Since, $W_{i_{1}}\,\cdots\, W_{i_{j}}|W_{i_{j+1}}\,\cdots\,W_{i_{m+1}}$ is in $P\big(\{W_{1}, \dots, W_{m+1}\}, \{n_{1}, \dots, n_{m+1}\}, w\big)$, it follows that the images of $\partial(W_{i_{1}}\,\cdots\, W_{i_{j}}|W_{i_{j+1}}\,\cdots\,W_{i_{m+1}})$ are null-homologous in $H_{n-3}\big(\text{conf}(n, w);\Q\big)$.
Therefore,
\begin{multline*}
\text{spin}_{W_{k}}\circ \text{spin}_{W_{1}}\circ\cdots\circ\text{spin}_{W_{m+1}}\circ q\big(\partial^{2}(W_{1}\, \cdots\, W_{m+1})\big)\\
=\sum_{k=1}^{m+1} \pm W(i_{k,1}, \dots, i_{k,n_{k}})|AF\big(W(i_{1,1}, \dots, i_{1,n_{1}}),\dots, \widehat{W(i_{k,1}, \dots, i_{k,n_{k}})}, \dots, W(i_{m+1,1}, \dots, i_{m+1,n_{m+1}})\big)\\
+\sum_{k=1}^{m+1} \pm AF\big(W(i_{1,1}, \dots, i_{1,n_{1}}),\dots, \widehat{W(i_{k,1}, \dots, i_{k,n_{k}})}, \dots, W(i_{m+1,1}, \dots, i_{m+1,n_{m+1}})\big)|W(i_{k,1}, \dots, i_{k,n_{k}})\\
=0.
\end{multline*}
Rearranging gives the relation in the statement of the Lemma. 
\end{proof}

\begin{figure}[h]
\centering
\captionsetup{width=.8\linewidth}
\includegraphics[width = 8cm]{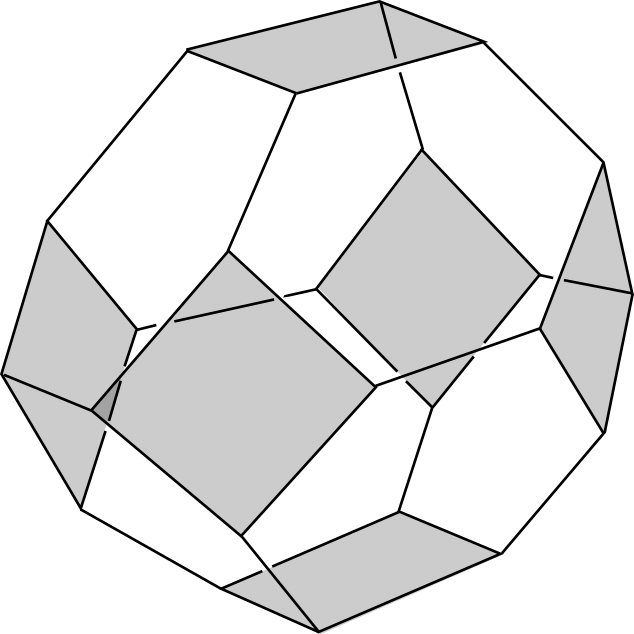}
\caption{The weighted permutohedron $P\big(\{W_{1},W_{2},W_{3},W_{4}\}, \{1,1,1,1\}, 2\big)$. 
The $6$ square faces are included, whereas the $8$ hexagonal faces are not, and their boundaries yield homology classes of the form $W_{i}|AF(W_{j}, W_{k}, W_{l})$ or $AF(W_{j}, W_{k}, W_{l})|W_{i}$.
It follows from this image that any one of these $8$ classes in first homology can be written as a sum of the other $7$.
}
\label{permutohedrarelation}
\end{figure}

Some terms in the relation of Lemma \ref{relationlem} might be trivial. 
If $n_{1}+\cdots+\widehat{n_{k}}+\cdots+n_{m+1}\le w$, then the averaged-filter
\[
AF\big(W(i_{1,1}, \dots, i_{1,n_{1}}),\dots, \widehat{W(i_{k,1}, \dots, i_{k,n_{k}})}, \dots, W(i_{m+1,1}, \dots, i_{m+1,n_{m+1}})\big)
\]
is trivial in homology as the weighted permutohedron $P\big(\{W_{1}, \dots, \widehat{W_{k}}, \dots, W_{m+1}\}, \{n_{1}, \dots, \widehat{n_{k}},\dots, n_{m+1}\},w\big)$ is homotopy equivalent to a point, not $S^{m-2}$. 
As a result, the $(m-1)$-cells $W_{k}|W_{1}\,\cdots\,\widehat{W_{k}}\,\cdots\, W_{m+1}$ and $W_{1}\,\cdots\,\widehat{W_{k}}\,\cdots\, W_{m+1}|W_{k}$ are in $P\big(\{W_{1}, \dots, \widehat{W_{k}}, \dots, W_{m+1}\}, \{n_{1}, \dots, \widehat{n_{k}},\dots, n_{m+1}\},w\big)$ and the resulting concatenation products are trivial in homology. 
While we could omit such terms, we choose not to in order to present a unified family of relations.

In this relation the coefficient of 
\[
W(i_{k,1}, \dots, i_{k,n_{k}})|AF\big(W(i_{1,1}, \dots, i_{1,n_{1}}),\dots, \widehat{W(i_{k,1}, \dots, i_{k,n_{k}})}, \dots, W(i_{m+1,1}, \dots, i_{m+1,n_{m+1}})\big)
\]
is $(-1)^{k-1}(-1)^{(n_{k}-1)(n_{1}+\cdots+n_{k-1}-k+1)}$, and the coefficient of 
\[
AF\big(W(i_{1,1}, \dots, i_{1,n_{1}}),\dots, \widehat{W(i_{k,1}, \dots, i_{k,n_{k}})}, \dots, W(i_{m+1,1}, \dots, i_{m+1,n_{m+1}})\big)|W(i_{k,1}, \dots, i_{k,n_{k}})
\]
is $(-1)^{k-1}(-1)^{(n_{k}-1)(n_{k+1}+\cdots+n_{m+1}-m+k-1)}$.
We will not need these facts.

Additionally, Lemma \ref{relationlem} holds if we replace averaged-filters with filters.

Next, we find a basis for homology containing all but one of the terms in a relation described in Lemma \ref{relationlem}.
Moreover, our basis will replace filters with averaged-filters.
We will use this basis to find another basis in which we remove the need for averaged-filters on $2$ wheels.

\begin{lem}\label{replaceallbigfilters1}
The homology group $H_{k}\big(\text{conf}(n,w);\Q\big)$ has a basis consisting of concatenations of proper wheels and proper non-trivial averaged-filters such that in each wheel inside or outside an averaged-filter the highest labeled disk comes first.
We say one wheel ranks above another if it has more disks, or has the same number of disks and its largest label is greater. 
A cycle is in the basis if and only if:
\begin{enumerate}
\item The wheels in an averaged-filter are in order of increasing rank.
\item Adjacent wheels not inside an averaged-filter are in order of decreasing rank.
\item Every wheel immediately to the left of an averaged-filter ranks above the least wheel in the averaged-filter.
\end{enumerate}
\end{lem}

\begin{proof}
There is a bijection between the basis in Theorem \ref{AMthmB} and the set described in the lemma arising from that map that takes proper filters to proper averaged-filters on the same set of wheels (note an averaged-filter on 2 wheels is the same as a filter on 2 wheels). 
Thus, if we can write the basis elements of Theorem \ref{AMthmB} in terms of the elements described in the lemma, it follows that the set described in the lemma is a basis.

An element in the basis of Theorem \ref{AMthmB} is a concatenation product of proper wheels and proper filters, i.e., it is of the form
\[
Z_{1}|\cdots|Z_{r}
\]
for some $r$, where $Z_{j}$ is either a proper wheel $W_{j}$ or a proper filter $F_{j}$, and consecutive $Z_{j}$ satisfy the requirements of points $3$ and $4$ of Theorem \ref{AMthmB}. 
By Propositions \ref{newfilterfromold} and \ref{sign of averaged-filter} we can write the proper filters on $m$ wheels in a basis element of Theorem \ref{AMthmB} as a sum of proper averaged-filters on at most $m$ wheels such these averaged-filters satisfy point $1$, after potentially applying the relation of Proposition \ref{sign of averaged-filter} to rearrange the wheels in an averaged-filter.
Moreover, the largest labeled disk in each wheel comes first. 
Therefore,
\[
Z_{1}|\cdots|Z_{r} = \sum a_{Z'_{1}|\cdots|Z'_{r}} Z'_{1}|\cdots|Z'_{r}
\]
where, if $Z_{j}$ is the proper wheel $W_{j}$, then $Z'_{j}$ is the proper wheel $W_{j}$, and if $Z_{j}$ is a proper filter on $m\ge 3$ wheels, then $Z'_{j}$ is a proper averaged-filter on at most $l\le m$ wheels $AF_{j}(W_{j,1}, W_{j,2}, \dots, W_{j,l})$.
Note, if $Z_{j}=F(W_{j,1}, W_{j,2})$, then $Z'_{j}=AF(W_{j,1}, W_{j,2})=F(W_{j,1}, W_{j,2})$.
Unfortunately, it does not follow that consecutive $Z'_{j}$ satisfy the requirements of points $2$ and $3$ of the lemma.

It suffices to show that every concatenation product of the form
\[
Z'_{1}|\cdots|Z'_{k}
\]
where the $Z'_{j}$ are proper wheels where the largest label comes first or proper averaged-filters on at least $2$ wheels satisfying point $1$ such that the largest label in each wheel comes first, can be written as a sum of elements in the set described in the lemma.

The conditions in the lemma allow anything (a wheel or an averaged-filter) to come after an averaged-filter. 
Therefore, it suffices to show that if we have a concatenation product of proper wheels such that in each proper wheel the largest label comes first and such that together the wheels satisfy property $3$, and we concatenate to the right of this any proper wheel such that the largest label comes first or any proper averaged-filter satisfying point $1$ such that in each wheel the largest label is first the resulting concatenation product is equal to a sum of elements in the set described in the lemma such that the set of proper wheels (both those not in and in an averaged-filter) is the same.

We proceed by induction on the number of wheels in the string of wheels. 
This is true if we start with the empty string of wheels, as anything can follow an averaged-filter.
Next, we assume this is true if there are $j-1$ wheels in the concatenation product. 
Since the $Z'_{1},\dots, Z'_{j}$ are all wheels, we write them as $Z'_{l}=W_{l}$ where for all $l$, where $W_{l}$ is a proper wheel such that the largest label comes first, i.e., $W_{l}=W(i_{l,1},\dots, i_{l, n_{l}})$ such that $i_{l, 1}>i_{l, s}$ for all $s\neq 1$. 
Additionally, since the $W_{l}$ satisfy point $2$, we have that $W_{1}>\cdots>W_{j}$ where the ordering is given by rank.

If $Z'_{j+1}$ is a proper wheel such that the largest label comes first we write
\[
Z'_{j+1}=W_{j+1}=W(i_{j+1, 1},\dots, i_{j+1, n_{j+1}}),
\]
where $i_{j+1, 1}>i_{j+1, s}$ for all $s\neq 1$.
If $W_{j+1}$ is outranked by $W_{j}$, then
\[
W_{1}|\cdots|W_{j}|W_{j+1}
\]
is in the set described in the lemma.

If $W_{j+1}$ outranks $W_{j}$, then
\[
W_{j}|W_{j+1}=F(W_{j}, W_{j+1})\pm W_{j+1}|W_{j}=AF(W_{j}, W_{j+1})\pm W_{j+1}|W_{j}
\]
where the sign depends on the number of disks in the wheels $W_{j}$ and $W_{j+1}$. Therefore,
\[
W_{1}|\cdots|W_{j-1}|W_{j}|W_{j+1}=W_{1}|\cdots|W_{j-1}|AF(W_{j}, W_{j+1})\pm W_{1}|\cdots|W_{j-1}|W_{j+1}|W_{j}.
\]
By the assumption, we can write
\[
W_{1}|\cdots|W_{j-1}|AF(W_{j}, W_{j+1})
\]
as a sum of elements in the set described in the lemma as $W_{j-1}$ outranks $W_{j}$.
Similarly, the induction hypothesis proves that we can write
\[
W_{1}|\cdots|W_{j-1}|W_{j+1}=\sum a_{Z''_{1},\dots, Z''_{t}} Z''_{1}|\dots|Z''_{t}
\]
where the $Z''_{1}|\dots|Z''_{t}$ are on the same set of proper wheels as $W_{1}|\cdots|W_{j-1}|W_{j+1}$ and are in the set described in the lemma. 
Note that all the wheels $W_{1},\dots, W_{j+1}$ outrank $W_{j}$ by assumption. 
Therefore, any such element $Z''_{1}|\dots|Z''_{t}$ is such that if $Z''_{t}$ is a proper wheel it ranks above $W_{j}$. 
It follows that $Z''_{1}|\dots|Z''_{t}|W_{j}$ is in the set described in the lemma, and, as a result, if $Z'_{j+1}$ is a proper wheel, the induction hypothesis holds.

We also need to check the case in which $Z'_{j+1}$ is a proper averaged-filter on $m\ge 2$ wheels satisfying property $1$, and such that in every wheel the largest label comes first. 
We write $Z'_{j+1}=AF(W_{j+1},\dots, W_{j+m})$. 

If $W_{j}$ outranks $W_{j+1}$, then
\[
W_{1}|\cdots|W_{j}|AF(W_{j+1},\dots, W_{j+m})
\]
is in the set described in the lemma.

If $W_{j}$ is outranked by $W_{j+1}$, then by Lemma \ref{relationlem}
\[
W_{j}|AF(W_{j+1},\dots, W_{j+m})=\pm AF(W_{j+1},\dots, W_{j+m})|W_{j}
\]
\[
+\sum_{i=1}^{m} \pm W_{j+i}|AF(W_{j}, \dots, \widehat{W_{j+i}},\dots, W_{j+m})
+\sum_{i=1}^{m} \pm AF(W_{j}, \dots, \widehat{W_{j+i}},\dots, W_{j+m})|W_{j+i},
\]
where the signs depend on the order of the wheels and the number of disks they are on.
Additionally, we may assume that all the averaged-filters in this relation are non-trivial as we can ignore any trivial terms.

By the induction hypothesis
\[
W_{1}|\cdots|W_{j-1}|AF(W_{j+1},\dots, W_{j+m})
\]
is equal to a sum of elements in the set described in the lemma such that these elements have the same set of proper wheels as $W_{1}|\cdots|W_{j-1}|AF(W_{j+1},\dots, W_{j+m})$. 
Let $Z''_{1}|\cdots|Z''_{t}$ be a summand. 
By assumption $W_{j}$ is outranked by all these wheels, so
\[
Z''_{1}|\cdots|Z''_{t}|W_{j}
\]
is in the set described in the lemma, and have the same set of proper wheels as
\[
W_{1}|\cdots|W_{j-1}|AF(W_{j+1},\dots, W_{j+m})|W_{j}.
\]
Therefore,
\[
W_{1}|\cdots|W_{j-1}|AF(W_{j+1},\dots, W_{j+m})|W_{j}.
\]
is equal to a sum of elements that are in the set described in the lemma, and these elements are on the same set of proper wheels.

For $1\le i\le m$, the element
\[
W_{1}|\cdots|W_{j-1}|AF(W_{j}, \dots, \widehat{W_{j+i}},\dots, W_{j+m})
\]
is in the set described in the lemma, as $W_{j-1}$ outranks $W_{j}$, so 
\[
W_{1}|\cdots|W_{j-1}|AF(W_{j}, \dots, \widehat{W_{j+i}},\dots, W_{j+m})|W_{j+i}
\]
is also in the set described in the lemma.
Clearly it is on the same set of wheels as the concatenation product $W_{1}|\cdots|W_{j}|AF(W_{j}, \dots, W_{j+m})$.

By the induction hypothesis, for $1\le i\le m$,
\[
W_{1}|\cdots|W_{j-1}|W_{j+1}
\]
is equal to a sum of elements in the set described in the lemma such that these elements are on the same set of wheels. 
Let $Z''_{1}|\cdots|Z''_{t}$ be such a summand. 
Then, every wheel in $Z''_{1}|\cdots|Z''_{t}$ outranks $W_{j}$, so 
\[
Z''_{1}|\cdots|Z''_{t}|AF(W_{j}, \dots, \widehat{W_{j+i}},\dots, W_{j+m})
\]
is on the same set of wheels as $W_{1}|\cdots|W_{j-1}|W_{j+1}|AF(W_{j}, \dots, \widehat{W_{j+i}},\dots, W_{j+m})$, and is in the set described in the lemma.
It follows that
\[
W_{1}|\cdots|W_{j}|AF(W_{j+1},\dots, W_{j+m})
\]
is equal to a sum of elements in the set described in the lemma on the same set of wheels.

We have shown that the induction hypothesis holds in the case of $j$ consecutive wheels, so the lemma holds.
\end{proof}

We do not want to have to consider averaged-filters on $2$ wheels as they can be written as a sum of concatenation products of wheels and their existence makes finding a finite presentation for the twisted algebra $H_{*}\big(\text{conf}(\bullet, w);\Q\big)$ more challenging.
As such, we find a new basis for homology devoid of averaged-filters on $2$ wheels.

\begin{thm}\label{AMWthmB''}
The homology group $H_{k}\big(\text{conf}(n,w);\Q\big)$ has a basis consisting of concatenations of proper wheels and non-trivial proper averaged-filters on $m\ge 3$ wheels such that in each wheel, inside or outside of an averaged-filter, the largest label comes first.
We say one wheel ranks above another if it has more disks or if they have the same number of disks and its largest label is greater. 
A cycle is in the basis if and only if:
\begin{enumerate}
\item The wheels inside each proper averaged-filter are in order of increasing rank. 
\item If $W_{1}$ and $W_{2}$ are adjacent proper wheels on $n_{1}$ and $n_{2}$ disks, respectively, where $W_{1}$ is to the left of $W_{2}$, then $W_{1}$ outranks $W_{2}$ or $n_{1}+n_{2} > w$.
\item Every wheel to the immediate left of a proper averaged-filter ranks above the least wheel in the averaged-filter.
\end{enumerate}
\end{thm}

\begin{proof}
We show that every element in the basis of Lemma \ref{replaceallbigfilters1} can be written as a sum of elements in the set described in the theorem. 
Then, we give an injective map from the set described in the theorem to the basis of Lemma \ref{replaceallbigfilters1}.
This proves that the set described in the theorem is a basis.

An element of the basis of Lemma \ref{replaceallbigfilters1} is of the form
\[
Z_{1}|Z_{2}|\cdots|Z_{r}
\]
for some $r$, where $Z_{j}$ is either a proper wheel whose largest label comes first or a proper averaged-filter such that in each wheel the largest label comes first and the wheels are ordered by increasing rank. 
Moreover, if $Z_{j+1}$ is a proper averaged-filter, then $Z_{j}$ is not a proper wheel of smaller rank than the wheels of $Z_{j+1}$, and if $Z_{j+1}$ is a proper wheel, it is of smaller rank than $Z_{j}$. 

Let $Z_{j} = AF(W_{j,1}, W_{j,2})=F(W_{j,1}, W_{j,2})=W_{j, 1}|W_{j, 2} \pm W_{j, 2}|W_{j,1}$ be the first (from the left) proper averaged-filter on $2$ wheels (here the sign depends on the number of disks in each wheel). 
We will replace $Z_{j}$ with $W_{j, 1}|W_{j, 2} \pm W_{j, 2}|W_{j,1}$, i.e., 
\[
Z_{1}|\cdots|Z_{j-1}|Z_{j}|Z_{j+1}|\cdots|Z_{k} =Z_{1}|\cdots|Z_{j-1}|W_{j, 1}|W_{j, 2}|Z_{j+1}|\cdots |Z_{k}\; \pm \;Z_{1}|\cdots|Z_{j-1}|W_{j, 2}|W_{j,1}|Z_{j+1}|\cdots |Z_{k}.
\]

Since $W_{j, 1}$ has smaller rank than $W_{j, 2}$ and $Z_{j-1}|Z_{j}$ is allowed, it follows that either $Z_{j-1}$ is a proper averaged-filter on at least $3$ wheels or a wheel that outranks $W_{j, 1}$. 
Either way, 
\[
Z_{1}|\cdots|Z_{j-1}|W_{j, 1}|W_{j, 2}
\]
is in the set described in the theorem. 

Since $W_{j, 2}$ outranks $W_{j, 1}$, it is wheel on at least $\big\lfloor\frac{w}{2}\big\rfloor+1$ disks. 
If $Z_{j+1}$ is wheel and outranks $W_{j,2}$, then the total number of disks in $W_{j,2}$ and $Z_{j+1}$ is greater than $W$, and 
\[
Z_{1}|\cdots|Z_{j-1}|W_{j, 1}|W_{j, 2}|Z_{j+1}
\]
is in the set described in the theorem. 

Similarly, if $Z_{j+1}$ is a wheel that is outranked by $W_{j, 2}$, then
\[
Z_{1}|\cdots|Z_{j-1}|W_{j, 1}|W_{j, 2}|Z_{j+1}
\]
is in the set described in the theorem. 

If $Z_{j+1}$ is a proper averaged-filter on at least $3$ wheels, then $W_{j, 2}$ must outrank the least wheel in $Z_{j+1}$, as the least ranked wheel in a proper averaged-filter has at most $\big\lfloor\frac{w}{2}\big\rfloor$ disks. It follows that 
\[
Z_{1}|\cdots|Z_{j-1}|W_{j, 1}|W_{j, 2}|Z_{j+1}
\]
is in the set described in the theorem.

Regardless of the nature of $Z_{j+1}$, the concatenation product
\[
Z_{1}|\cdots|Z_{j-1}|W_{j, 1}|W_{j, 2}|Z_{j+1}|\dots|Z_{l-1},
\]
where $Z_{l}$ is the next proper averaged-filter on 2 wheels, satisfies the conditions of the theorem.

For 
\[
Z_{1}|\cdots|Z_{j-1}|W_{j, 2}|W_{j,1}
\]
note that if $Z_{j-1}$ is proper averaged-filter consisting of at least $3$ wheels, then 
\[
Z_{1}|\cdots|Z_{j-1}|W_{j, 2}|W_{j,1}
\]
is in the set described in the theorem, and if $Z_{j-1}$ is a proper wheel, then, since $Z_{j-1}|Z_{j}$ is allowed, it follows that $Z_{j-1}$ outranks $W_{j,1}$ so $Z_{j-1}$ and $W_{j,2}$ have more than $w$ wheels in total. 
Therefore, 
\[
Z_{1}|\cdots|Z_{j-1}|W_{j, 2}|W_{j,1}
\]
is in the set described in the theorem.

If $Z_{j+1}$ is a wheel of rank smaller than that of $W_{j,1}$, then
\[
Z_{1}|\cdots|W_{j, 2}|W_{j,1}|Z_{j+1}
\]
is in the set described in the theorem. 

If $Z_{j+1}$ is a wheel and $Z_{j+1}$ and $W_{j,1}$ have more than $w$ disks combined, then
\[
Z_{1}|\cdots|Z_{j-1}|W_{j, 2}|W_{j,1}|Z_{j+1}
\]
is in the set described in the theorem. 

If $Z_{j+1}$ is a wheel that outranks $W_{j, 1}$ and they have at most $w$ disks combined, then, by Proposition \ref{commutewheels},
\[
W_{j,1}|Z_{j+1},
\]
is homologous up to sign to 
\[
Z_{j+1}|W_{j, 1},
\]
so 
\[
Z_{1}|\cdots|Z_{j-1}|W_{j, 2}|Z_{j+1}|W_{j,1}
\]
is in the set described in the theorem, as $W_{j,2}$ and $Z_{j+1}$ have more than $w$ disks combined. We can repeat this process as long as the subsequent $Z_{j+i}$ are wheels.

If $Z_{j+1}$ is a proper averaged-filter on at least 3 wheels, then, if $W_{j, 1}$ outranks the least wheel in $Z_{j+1}$, it follows that 
\[
Z_{1}|\cdots|Z_{j-1}|W_{j, 2}|W_{j,1}|Z_{j+1}
\]
is in the set described in the theorem.  
If $W_{j,1}$ does not outrank the least wheel in $Z_{j+1}$, then we can use the relation from Lemma \ref{relationlem} to replace $W_{j, 1}|Z_{j+1}$ with a sum of concatenation products of proper wheels and proper averaged-filters on at least $3$ wheels (and the wheels are the same wheels as in $W_{j, 1}$ and $Z_{j, 1}$), i.e., 
\[
W_{j, 1}|Z_{j+1} = \sum W'_{j, 1}|Z'_{j+1}+\sum Z'_{j+1}|W'_{j, 1} + Z_{j+1}|W_{j+1}.
\]
Since $W_{j,1}$ is outranked by all the wheels in $Z_{j+1}$, it follows that all the wheels of the form  $W'_{j, 1}$ outrank $W_{j, 1}$ and $W_{j,2}|W'_{j,1}$ satisfy the conditions of the theorem.
Additionally, it follows that $W_{j,2}$ outranks the least wheel in each averaged-filter $Z'_{j+1}$, as $W_{j, 1}$ is the least wheel in all the $Z'_{j+1}$.
Thus, all the products
\[
Z_{1}|\cdots|Z_{j-1}|W_{j, 2}|W'_{j,1}|Z_{j+1}
\]
\[
Z_{1}|\cdots|Z_{j-1}|W_{j, 2}|Z'_{j+1}|W'_{j,1}
\]
and 
\[
Z_{1}|\cdots|Z_{j-1}|W_{j, 2}|Z'_{j+1}|W_{j,1}
\]
satisfy the conditions of the theorem.

We can repeat the above process to see the summands used to replace
\[
Z_{1}|\cdots|Z_{j-1}|Z_{j}|Z_{j+1}|\cdots|Z_{l-1},
\]
where $Z_{l}$ is the next proper filter on two wheels, satisfy the conditions in the theorem.
We can repeat this process until we have no more filters on 2 wheels in the summands.
Doing so, we see that we are able to write 
\[
Z_{1}|\cdots|Z_{k}
\]
as a sum of terms in the set described by the theorem.
Therefore, the theorem describes a generating set for homology.

Next, we show that the set described in the theorem is at most the size of the basis described in Lemma \ref{replaceallbigfilters1}.
An element in the set described in the theorem is of the form
\[
Z'_{1}|\cdots|Z'_{s}
\]
where the $Z'_{i}$ are either proper wheels such that the largest label comes first or proper averaged-filters on at least $3$ wheels such that in each wheel the largest label comes first. 
Working from left to right we pair off consecutive wheels $Z'_{j}$ and $Z'_{j+1}$ such that $Z'_{j+1}$ outranks $Z'_{j}$ and $Z'_{j}$ is not already in a pair. 

The process described above that shows that the set in the theorem generates homology proves that 
\[
Z'_{1}|\cdots|Z'_{j-1}|Z'_{j}|Z'_{j+1}|\cdots|Z'_{k}
\]
appears in our decomposition of 
\[
Z'_{1}|\cdots|AF(Z'_{j}, Z'_{j+1})|\cdots|Z'_{k},
\]
as if $Z'_{j-1}$ is a wheel not paired off, then $Z'_{j-1}$ outranks $Z'_{j}$ by assumption, so it can come immediately to the left of $AF(Z'_{j}, Z'_{j+1})$, and if it is a wheel that is paired off or an averaged-filter it, or the averaged-filter on $2$ wheels that replaces it, can come to the left of $AF(Z'_{j}, Z'_{j+1})$, as anything can follow a filter on 2 wheels or an averaged-filter.
Anything can come to the right of $AF(Z'_{j}, Z'_{j+1})$. 
This process gives an injective map from the set in Theorem to the basis of Lemma \ref{replaceallbigfilters1}.
Thus, the cardinality of the set of the theorem is at most the cardinality of the basis of Lemma \ref{replaceallbigfilters1}. 
Since every element in the basis of Lemma \ref{replaceallbigfilters1} can be written in terms of the set in the theorem it follows that the set described in the theorem is a basis.
\end{proof}

Thus, we have proven that there is a basis for $H_{k}\big(\text{cell}(n,w);\Q\big)$ that consists only of concatenation products of non-trivial proper averaged-filters on at least $3$ wheels and proper wheels. 
We use this theorem to find a finite presentation for $H_{*}\big(\text{cell}(\bullet, w);\Q\big)$ as a twisted noncommutative algebra.

\section{Representation Stability}\label{rep stab sec}
In this section we recall the definition of a twisted (commutative) algebra, and use twisted algebras to state and prove our main results. 
Alpert--Manin proved that the homology of the ordered configuration space of open unit-diameter disks in the infinite strip of width $w$ has the structure of a finitely generated twisted algebra, and gave a set of generators \cite[Theorem A]{alpert2021configuration1}. 
We prove that $H_{*}\big(\text{conf}(\bullet, w);\Q\big)$ is a finitely presented twisted algebra by giving a finite presentation.
We use this presentation to prove Theorem \ref{first order stability}, which states that for all $k\ge0$, the $k^{\text{th}}$-homology groups $H_{k}\big(\text{conf}(\bullet, w);\Q\big)$ have the structure of a finitely generated $\text{Sym}\bigg(\Big(H_{k}\big(\text{conf}(1,w);\Q\big)\Big)^{b+1}\bigg)$-module for some $b\ge 0$, i.e., a twisted algebra version of Theorem \ref{first order stab intro}.
In doing so, we show that the ordered configuration space of unit-diameter disks in the infinite strip of width $w$ exhibits a notion of first-order representation stability, something that Alpert--Manin suggest might not be true if one restricts themselves to integer coefficients \cite[Proposition 8.6]{alpert2021configuration1}.
We also prove Theorem \ref{higherorderstability}, a twisted algebra version of Theorem \ref{higher order stability intro}, showing that the ordered configuration space of open unit-diameter disks in the infinite strip of width $w$ exhibits notions of $2^{\text{nd}}$ through $\big\lfloor\frac{w+1}{3}\big\rfloor^{\text{th}}$-order representation stability.

\subsection{Twisted Algebras}
We recall the definition of a twisted algebra and prove that $H_{*}\big(\text{conf}(\bullet, w);\Q\big)$ is a finitely presented twisted algebra by giving a finite presentation.
We will use this presentation to prove our representation stability results.

\begin{defn}
Let $R$ be a commutative ring. A \emph{twisted algebra over $R$} is an associative unital graded $R$-algebra $A$ supported in non-negative degrees and equipped with an $S_{n}$-action on the degree $n$ piece $A_{n}$ such that the multiplication map
\[
A_{n}\otimes A_{m}\to A_{n+m}
\]
is $S_{n}\times S_{m}$ equivariant.
\end{defn}

\begin{defn}
A twisted algebra $A$ over a ring $R$ is \emph{finitely generated} if there are finitely many elements $a_{1}, \dots, a_{k}\in A$, where $a_{j}\in A_{d_{j}}$, such that every element in $A$ can be written as a finite $R$-linear combination of symmetric group actions on products of the $a_{j}$. The set $\{a_{1}, \dots, a_{k}\}$ is a \emph{generating set} for $A$.
We say that $A$ is \emph{finitely generated in degree $d$} if there exists a finite generating set $\{a_{1}, \dots, a_{k}\}$ such that $a_{j}\in A_{d_{j}}$ where $d_{j}\le d$ for all $j$.
\end{defn}

\begin{defn}
A twisted algebra $A$ over a ring $R$ is \emph{finitely presented} if $A$ has a finite generating set $\{a_{1},\dots, a_{k}\}$ and there exists a finite set of relations $\mathcal{R}=\{r_{1},\dots, r_{t}\}$ in pure degree in terms of the generators $a_{1},\dots, a_{k}$ such that every relation in $A$ can be written as a finite $R$-linear combination of symmetric group actions on products of the $r_{j}$.
We say that $A$ is \emph{finitely presented in degree $p$} if there exists a finite generating set $\{a_{1}, \dots, a_{k}\}$ and a finite set of relations $\{r_{1},\dots, r_{t}\}$ such that $a_{j}\in A_{d_{j}}$ where $d_{j}\le p$ for all $j$ and $r_{i}\in A_{p_{i}}$ where $p_{i}\le p$ for all $i$.
\end{defn}

In their paper, Alpert--Manin proved that the homology of the ordered configuration space of open unit-diameter disks in the infinite strip of width $w$ is a finitely generated twisted algebra \cite[Theorem A]{alpert2021configuration1} generated in degree $\le\frac{3}{2}w$. 
We prove that $H_{*}\big(\text{conf}(\bullet,w);\Q\big)$ is a finitely presented twisted algebra, and give bounds for the presentation degree.

\begin{thm}\label{finitepresentation} 
For fixed $w$, the twisted algebra $H_{*}\big(\text{conf}(\bullet,w);\Q\big)$ is finitely presented.
Moreover, there is a finite presentation for $H_{*}\big(\text{conf}(\bullet,w);\Q\big)$ such that the generators are in $H_{\le\frac{3}{2}w-2}\big(\text{conf}(\le\frac{3}{2}w,w);\Q\big)$ and the relations are in $H_{\le2w-3}\big(\text{conf}(\le2w,w);\Q\big)$, i.e., $H_{*}\big(\text{conf}(\bullet,w);\Q\big)$ is generated in degree $\le\frac{3}{2}w$ and presented in degree $\le2w$.
\end{thm}

\begin{proof}
By Theorem \ref{AMWthmB''}, for every $n, k, w$, there is a basis for $H_{k}\big(\text{conf}(n,w);\Q\big)$  consisting of concatenation products of proper wheels on at most $w$ disks and non-trivial proper averaged-filters on at least $3$ wheels such that in every wheel the largest label comes first and such that the wheels in the averaged-filters are in order of increasing rank. 
It follows that $H_{*}\big(\text{conf}(\bullet,w);\Q\big)$ is generated by the proper wheels whose largest label comes first and the non-trivial proper averaged-filters such that in every wheel the largest label comes first and such that the wheels in the average-filters are in order of increasing rank. 
By the nature of twisted algebras it suffices to restrict ourselves to one proper wheel on $d$ disks for $1\le d\le w$ and one proper averaged-filter for each partition of $n$, where $w< n\le \frac{3}{2}w$, and such that any subset of $1$ fewer element than the number of elements in the partition has sum at most $w$.
This is a finite generating set for $H_{*}\big(\text{conf}(\bullet,w);\Q\big)$.

We claim that these generators along with the relations of Propositions \ref{properwheels}, \ref{commutewheels}, \ref{sign of averaged-filter}, and \ref{replace filter with sum of proper proper} and Lemma \ref{relationlem} yield a finite presentation for $H_{*}\big(\text{conf}(\bullet,w);\Q\big)$ as a twisted algebra.
Since these relations use at most $2w$ disks and there are only finitely many ways to group at most $2w$ disks into wheels and averaged-filters, it follows that this set of relations, and therefore this presentation, is finite. 

To see that these relations yield a presentation, it suffices to show that any element of homology that arises from a symmetric group action on a product of generators can be written as a sum of the basis elements from Theorem \ref{AMWthmB''} using these relations.
Let $Z_{1}|\cdots|Z_{r}\in H_{k}\big(\text{conf}(n,w);\Q\big)$ be a product of the generators, i.e., the proper wheels such that the largest label comes first and averaged-filters on at least $3$ wheels whose wheels are ordered in increasing rank and whose wheels are such that the largest label comes first.
Given any $\sigma\in S_{n}$, consider
\[
\sigma(Z_{1}|\cdots|Z_{r}).
\]
By the twisted algebra structure of $H_{k}\big(\text{conf}(n,w);\Q\big)$ it follows that
\[
\sigma(Z_{1}|\cdots|Z_{r})=\sigma(Z_{1})|\cdots|\sigma(Z_{r}).
\]
Moreover, if $Z_{j}$ is a proper wheel on $n_{j}$ disks, then $\sigma(Z_{j})$ is a proper wheel on $n_{j}$ disks, and if $Z_{j}$ is a proper averaged-filter on $m$ wheels of size $n_{j,1}, \dots, n_{j, m}$,  then $\sigma(Z_{j})$ is proper-averaged-filter on $m$ wheels of size $n_{j,1},\dots, n_{j,m}$.
If $\sigma(Z_{j})$ is a proper wheel, we can use the relations of Proposition \ref{properwheels} to write 
\[
\sigma(Z_{j})=\sum a_{Z'_{j}}Z'_{j}
\]
for some constants $a_{Z'_{j}}$ such that the $Z'_{j}$ are proper wheels of the same size and are on the same set of disks as $\sigma(Z_{j})$ such that the largest label comes first.
Similarly, if $\sigma(Z_{j})$ is a proper averaged-filter, we can use the relations of Propositions \ref{sign of averaged-filter} and \ref{replace filter with sum of proper proper} to write 
\[
\sigma(Z_{j})=\sum a_{Z'_{j}}Z'_{j}
\]
for some constants $a_{Z'_{j}}$ such that the $Z'_{j}$ are proper averaged-filters whose wheels are of the same size and are on the same set of disks as the wheels of $\sigma(Z_{j})$, and whose wheels are ordered in increasing rank and the largest labeled disk in each wheel comes first.
As a result we have
\[
\sigma(Z_{1}|\cdots|Z_{r})=\sum a_{Z'_{1}}Z'_{1}|\cdots|\sum a_{Z'_{r}}Z'_{r}=\sum a_{Z'_{1}}\cdots \sum a_{Z'_{r}}Z'_{1}|\cdots|Z'_{r}=\sum a_{Z'_{1}}\cdots a_{Z'_{r}}Z'_{1}|\cdots|Z'_{r}.
\]
where the $Z'_{j}$ are proper wheels such that the largest label comes first or proper averaged-filters such the wheels are ordered in increasing rank and such that in each wheel the largest label comes first.
We use the relations of Proposition \ref{commutewheels} and Lemma \ref{relationlem} to write the summands $Z'_{1}|\cdots|Z'_{r}$ as a sum of elements in the basis of Theorem \ref{AMWthmB''}.

In an element of the basis of Theorem \ref{AMWthmB''}, a proper wheel or a proper averaged-filter can come to the right of a proper wheel on at least $\big\lfloor\frac{w}{2}\big\rfloor+1$ disks or a proper averaged-filter. 
Thus, much like Lemma \ref{replaceallbigfilters1}, it suffices to show that if we have a concatenation product of any number of proper wheels, each on at most $\big\lfloor\frac{w}{2}\big\rfloor$ disks, in the basis of Theorem \ref{AMWthmB''}, and we concatenate on the right with any proper wheel where the largest label comes first or any non-trivial proper averaged-filter where the wheels are ordered in increasing rank, and in each wheel the largest label comes first, then the resulting element can be written as a sum elements in the basis of Theorem \ref{AMWthmB''}, such that each element arises from the same set of wheels.

We do this by induction on the number of wheels following an averaged-filter or a wheel on at least $\big\lfloor\frac{w}{2}\big\rfloor+1$ disks. 
This is trivially true in the case there are no wheels in the concatenation product after the averaged-filter or large wheel.\
Assume this true if we have any concatenation product of at most $j-1$ wheels of the desired form. 
Let 
\[
W_{1}|\cdots|W_{j}
\]
be a concatenation product of $j$ proper wheels, each consisting of at most $\big\lfloor\frac{w}{2}\big\rfloor$ disks such that the largest label in each wheel comes first each and such that $W_{i}$ outranks $W_{i+1}$. 
We consider 
\[
W_{1}|\cdots|W_{j}|Z_{j+1}
\]
where $Z_{j+1}$ is either a proper wheel where the largest label comes first or a non-trivial proper averaged-filter such that the wheels are ordered in increasing rank and such that in each wheel the largest label comes first.

If $Z_{j+1}=W_{j+1}$ is a proper wheel outranked by $W_{j}$ or the number of disks in $W_{j}$ and $W_{j+1}$ is more than $w$, then
\[
W_{1}|\cdots|W_{j}|Z_{j+1}
\]
is in the basis of Theorem \ref{AMWthmB''}.

If $Z_{j+1}=W_{j+1}$ is a proper wheel outranking $W_{j}$ such that the number of disks in $W_{j}$ and $W_{j+1}$ is at most $w$, then, by Proposition \ref{commutewheels},
\[
W_{j}|W_{j+1}=\pm W_{j+1}|W_{j}
\]
where the sign depends on the number of disks in $W_{j}$ and the number of disks in $W_{j+1}$. 
By the induction hypothesis we can write
\[
W_{1}|\cdots|W_{j-1}|W_{j+1}
\]
as a sum of elements in the basis of Theorem \ref{AMWthmB''} that consist of the same set of wheels. 
Let $Z'_{1}|\cdots|Z'_{t}$ be a non-trivial basis element in the sum. 
By assumption every wheel in $Z'_{1}|\cdots|Z'_{r}$ outranks $W_{j}$, so 
\[
Z'_{1}|\cdots|Z'_{t}|W_{j}
\]
is in the basis of Theorem \ref{AMWthmB''}.
Therefore, we can write
\[
W_{1}|\cdots|W_{j}|Z_{j+1}
\]
as a sum of elements in the basis of Theorem \ref{AMWthmB''} arising from the same set of wheels.

If $Z_{j+1}$ is a proper averaged-filter, $AF(W_{j+1}, \dots, W_{j+m})$ such that the least wheel of the $W_{j+1}, \dots, W_{j+m}$, i.e., $W_{j+1}$, is outranked by $W_{j}$, then
\[
W_{1}|\cdots|W_{j}|Z_{j+1}=W_{1}|\cdots|W_{j}|AF(W_{j+1}, \dots, W_{j+m})
\]
is in the basis of Theorem \ref{AMWthmB''}.

If $Z_{j+1}$ is a proper averaged-filter, $AF(W_{j+1}, \dots, W_{j+m})$ such that the least wheel of the $W_{j+1}, \dots, W_{j+m}$, that is $W_{j+1}$, outranks $W_{j}$, then, by Lemma \ref{relationlem},
\begin{multline*}
W_{j}|AF(W_{j+1}, \dots, W_{j+m})=\pm AF(W_{j+1}, \dots, W_{j+m})|W_{j}\\
\pm \sum_{i=1}^{m}\pm W_{j+i}|AF(W_{j},\dots, \widehat{W_{j+i}},\dots, W_{j+m})\pm \sum_{i=1}^{m}\pm AF(W_{j},\dots, \widehat{W_{j+i}},\dots, W_{j+m})|W_{j+i}
\end{multline*}
where the signs depend on number of disks in each wheel, and we can ignore terms where the averaged-filter is trivial.

By assumption $W_{j-1}$ outranks $W_{j}$, so each 
\[
W_{1}|\cdots|W_{j-1}|AF(W_{j},\dots, \widehat{W_{j+i}},\dots, W_{j+m})|W_{j+i}
\]
is in the basis of Theorem \ref{AMWthmB''}.

By the induction hypothesis we can write
\[
W_{1}|\cdots|W_{j-1}|AF(W_{j+1}, \dots, W_{j+m})
\]
as a sum of elements in the basis of Theorem \ref{AMWthmB''} that arise from the same set of wheels. 
Let $Z'_{1}|\cdots|Z'_{t}$ be a summand. 
By assumption $W_{j}$ is outranked by every wheel $W_{1},\dots, W_{j-1}, W_{j+1}, \dots, W_{j+m}$, so
\[
Z'_{1}|\cdots|Z'_{r}|W_{j}
\]
is in the basis of Theorem \ref{AMWthmB''} and it arises from the same set of wheels.
Therefore,
\[
W_{1}|\cdots|W_{j-1}|AF(W_{j+1}, \dots, W_{j+m})|W_{j}
\]
can be written as a sum elements in the basis of Theorem \ref{AMWthmB''} arising from the same set of wheels.

Finally, by the induction hypothesis, for each $i$ we can write
\[
W_{1}|\cdots|W_{j-1}|W_{j+i}
\]
as a sum of elements in the basis of Theorem \ref{AMWthmB''} arising from the same set of wheels. 
Let $Z'_{1}|\cdots|Z'_{t}$ be a summand. 
By assumption $W_{j}$ is outranked by every wheel $W_{1}, \dots, W_{j-1}, W_{j+i}$ so
\[
Z'_{1}|\cdots|Z'_{t}|AF(W_{j},\dots, \widehat{W_{j+i}},\dots, W_{j+m})
\]
is in the basis of Theorem \ref{AMWthmB''} and whose set of wheels is $\{W_{1}, \dots, W_{j+m}\}$.

Therefore, we can write
\[
W_{1}|\cdots|W_{j}|AF(W_{j+1},\dots, W_{j+m})
\]
as sum of elements of Theorem \ref{AMWthmB''} arising from the same set of wheels. 
We have proved the induction step, and therefore the theorem, as any relation not arising from the relations of Propositions \ref{properwheels}, \ref{commutewheels}, \ref{sign of averaged-filter} and \ref{replace filter with sum of proper proper} and Lemma \ref{relationlem}, would imply the existence of a basis strictly smaller than the basis of Theorem \ref{AMWthmB''}, a contradiction.
\end{proof}

Letting $w\to \infty$ eliminates the need for averaged-filters and the relations involving them, and doing this recovers the planetary systems/tall trees presentations of $H_{*}\big(F_{\bullet}(\R^{2});\Q\big)$ as a twisted algebra; for more, see \cite{sinha2006homology, knudsen2018configuration}.

We will use the presentation for $H_{*}\big(\text{conf}(\bullet, w);\Q\big)$ given in the proof of Theorem \ref{finitepresentation} to prove our representation stability results. 
First, we recall a few more definitions concerning twisted algebras.
For a more thorough introduction to twisted commutative algebras and their connection to representation stability see \cite{sam2012introduction,  sam2017grobner, nagpal2019noetherianity}.

\begin{defn}
A twisted algebra $A$ is \emph{commutative} if the following twisted version of the commutativity axiom holds: For $x\in A_{m}$ and $y\in A_{m}$ we have $yx=\tau(xy)$, where $\tau\in S_{n+m}$ switches the first $n$ and last $m$ elements of $[n+m]$.
We say that a twisted algebra $A$ is \emph{skew-commutative} if for $x\in A_{m}$ and $y\in A_{m}$ we have $yx=(-1)^{nm}\tau(xy)$.
\end{defn}

Note that $H_{*}\big(\text{conf}(\bullet, w);\Q\big)$ is not a twisted commutative algebra (tca).

\begin{exam}
A commutative ring $R$ is a tca concentrated in degree $0$, where the symmetric group acts trivially.
\end{exam}

\begin{exam}
For $m\le \big\lfloor\frac{w}{2}\big\rfloor$, 
\[
\mathcal{L}_{m}:=\begin{cases}\text{Sym }H_{m-1}\big(\text{conf}(m, w);\Q\big)\cong \text{Sym }H_{m-1}\big(F_{m}(\R^{2});\Q\big) &\mbox{for $m$ odd}\\\bigwedge H_{m-1}\big(\text{conf}(m, w);\Q\big)\cong \bigwedge H_{m-1}\big(F_{m}(\R^{2});\Q\big)&\mbox{for $m$ even}\end{cases}
\]
is a twisted (skew-)commutative algebra over $\Q$. 
If $m=1$, this tca is isomorphic to $\text{Sym}\big(\text{Sym}^{1}(\Q)\big)$, and if $m=2$, it is isomorphic to $\bigwedge\big(\text{Sym}^{2}(\Q)\big)$, the twisted algebras of first- and second-order representation stability for configuration spaces of points in a manifold.
See \cite{sam2012introduction} and \cite{nagpal2019noetherianity} for more information on these twisted commutative algebras.
\end{exam}

\begin{defn}
A \emph{module over a tca $A$} is a graded $A$-module (in the usual sense) equipped with an action of $S_{n}$ on $M_{n}$ for all $n$ for which the multiplication map
\[
A_{n}\otimes M_{m}\to M_{n+m}
\]
is $S_{n}\times S_{m}$ equivariant.
\end{defn}

\begin{exam}
Being a module over the tca $R$ is equivalent to being an FB-module over $R$, where FB is the category of finite sets and bijections. 
In an abuse of notation, we write FB-mod for the category of modules over $R$ when $R$ is viewed as a tca. 
Every module over a t(s)ca is an FB-module.
For more on FB, see, for example, \cite{church2015fi}.
\end{exam}

\begin{exam}
If $X$ is a connected noncompact finite type manifold of dimension $d\ge 2$, then for all $i\ge 0$, $H_{i}\big(F_{\bullet}(X)\big)$ is a module over $\mathcal{L}_{1}$. 
Here multiplication of an element $Y\in H_{i}\big(F_{m}(X)\big)$ by the fundamental class $Z$ of $H_{i}\big(F_{1}(\R^{d})\big)$ is $Z\otimes Y\in H_{i}\big(F_{m+1}(X)\big)$.
See \cite[Section 1.1]{miller2019higher}.
\end{exam}

\begin{defn}
Let \emph{$H^{A}_{0}(-)_{\bullet}$} be the functor from $A$-mod to FB-mod such that
\[
H^{A}_{0}(M)_{n}:=\text{coker}\Big(\bigoplus_{n=p\sqcup q, p\neq0}A_{p}\otimes_{\text{FB}} M_{q}\to M_{n}\Big).
\]
\end{defn}

\begin{defn}
We say that an $A$-module $M$ is \emph{finitely generated} if $\bigoplus^{\infty}_{k=0}H^{A}_{0}(M)_{k}$ is finitely generated as an $R$-module. 
\end{defn}

\begin{defn}
Let $M$ be an module over the tca $A$. We say that $\deg M\le d$ if $M_{k}=0$ for all $k>d$. We say that $M$ is \emph{generated in degree $\le d$} if $\deg H^{A}_{0}(M) \le d$.
\end{defn}

We will see that the twisted algebra $H_{*}\big(\text{conf}(\bullet, w);\Q\big)$ is a finitely generated module over a number of twisted commutative algebras.

\begin{defn}
Let $A$ be a tca. We define
\[
M^{A}:\text{FB-mod}\to A\text{-mod}
\]
via the formula
\[
M^{A}(U):=A\otimes_{\text{FB}}U.
\]
The $A$-module structure on $M^{A}(U)$ is induced by the multiplication $A\otimes A\to A$. An $A$-module of the form $M^{A}(U)$ is said to be \emph{free}. 
\end{defn}

If $M$ is a finitely generated module over a tca $A$, we say that it is \emph{representation stable}. 

\begin{thm}\label{CEF,MW}
(Church--Ellenberg--Farb \cite[Theorem 6.4.3]{church2015fi} in the oriented case and Miller--Wilson \cite[Theorem 1.1]{miller2019higher} in the general case) If $X$ is a connected noncompact finite type manifold of dimension at least $2$, then $H_{k}\big(F_{\bullet}(X);R\big)$ is a finitely generated free $\text{Sym}\Big( H_{0}\big(F_{1}(\R^{d})\big)\Big)$-module generated in degree $\le 2k$ for all $k$.
\end{thm}

We will need slightly larger twisted commutative algebras than $\mathcal{L}_{m}$ to prove first- (and higher) order representation stability results for the ordered configuration space of open unit-diameter disks in the infinite strip of width $w$.

\begin{defn}
For $m\le\big\lfloor\frac{w}{2}\big\rfloor$, let
\[
\mathcal{L}^{w}_{m}(b):=\begin{cases}\text{Sym }\bigg(\Big(H_{m-1}\big(\text{conf}(m, w);\Q\big)\Big)^{\oplus b}\bigg)&\mbox{for $m$ odd}\\\bigwedge \bigg(\Big(H_{m-1}\big(\text{conf}(m, w);\Q\big)\Big)^{\oplus b}\bigg)&\mbox{for $m$ even}\end{cases}.
\]
\end{defn}

When $b=1$, these $\mathcal{L}^{w}_{m}(b)$ are isomorphic to the $\mathcal{L}_{m}$, the twisted commutative algebras that Miller--Wilson suggested could be used to formalize higher order representation stability for the homology of the ordered configuration space of points in a connected noncompact manifold of dimension $d\ge 2$, as $F_{m}(\R^{2})\simeq \text{conf}(m, w)$ and wheels on $m$ disks (anti)-commute.
See \cite[Definition 3.29]{miller2019higher}.

Finally, we state a proposition that we will need later when proving our representation stability results.

\begin{prop}\label{FIdisFId+1}
Let $M$ an $\mathcal{L}^{w}_{n}(b)$-module. Then $M$ is also an $\mathcal{L}^{w}_{n}(b+1)$-module.
\end{prop}

\begin{proof}
Identify multiplication by an element in the $(b+1)^{\text{th}}$ copy of $H_{n-1}\big(\text{conf}(n, w);\Q\big)$ by multiplication by an isomorphic element in the $b^{\text{th}}$-copy of $H_{n-1}\big(\text{conf}(n, w);\Q\big)$, i.e., if $X_{b+1}$ is in the $(b+1)^{\text{th}}$ copy of $H_{n-1}\big(\text{conf}(n, w);\Q\big)$ and $X_{b}\cong X_{b+1}$ is in the $b^{\text{th}}$-copy of $H_{n-1}\big(\text{conf}(n, w);\Q\big)$, then for all $m\in M$ set $X_{b+1}\otimes m=X_{b}\otimes m$. 
\end{proof}

\subsection{First-Order Representation Stability for the Homology of $\text{conf}(\bullet, w)$}

We prove that for all $k\ge0$ and all $w\ge 2$, the homology groups $H_{k}\big(\text{conf}(\bullet, w);\Q\big)$ have the structure of a finitely generated $\mathcal{L}^{w}_{1}(b+1)$-module, where $b$ depends only on $k$ and $w$, i.e., that the ordered configuration space of open unit-diameter disks on the infinite strip of width $w$ is first-order representation stable.
First, we recall a similar result of Alpert for $H_{k}\big(\text{conf}(\bullet, 2);\Z\big)$ and a proposition of Alpert--Manin that suggests Alpert's technique does not generalize to larger $w$.

\begin{defn}
Let $Z\in H_{k}\big(\text{conf}(n, w);\Z\big)$ be a concatenation product of wheels and filters, then an \emph{integral barrier} in $Z$ is any wheel on $w$ disks or any filter.
\end{defn}

When $w=2$ every element in $H_{k}\big(\text{conf}(n, w);\Z\big)$ has exactly $k$ integral barriers.
Alpert proved that for all $k\ge 0$, the homology groups $H_{k}\big(\text{conf}(\bullet, 2);\Z\big)$ have the structure of a finitely generated $\mathcal{L}^{w}_{1}(k+1)$-module over $\Z$ generated in degree $\le 3k$ \cite[Theorem 6.1]{alpert2020generalized}. 
Multiplication of a class $Z\in H_{k}\big(\text{conf}(n, 2);\Z\big)$ by the generator for the $i^{\text{th}}$ copy of $H_{0}\big(\text{conf}(1,2);\Z\big)$ in $\mathcal{L}^{w}_{1}(k+1)$, corresponds to inserting a wheel on $1$ disk immediately to the left of the $i^{\text{th}}$ barrier of $Z$ if $i\le k$, and to inserting a wheel on $1$ disk at the far right of $Z$ if $i=k+1$. 
See Figure \ref{conf2FId}.

\begin{figure}[h]
\centering
\captionsetup{width=.8\linewidth}
\includegraphics[width = 10cm]{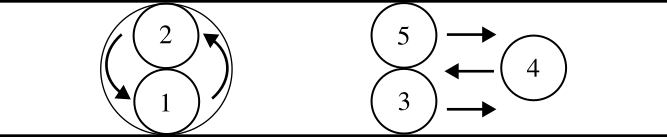}
\caption{A class in the basis for $H_{2}\big(\text{conf}(5,2);\Z\big)$ from Theorem \ref{AMthmB} consisting of a concatenation product of a proper wheel and a proper filter.
The $\mathcal{L}^{2}_{1}(3)$-module structure of $H_{2}\big(\text{conf}(\bullet,2);\Z\big)$ arises from the fact that there are 3 regions in which one could insert a wheel on $1$ disk, namely to the left of the wheel, in-between the wheel and the filter, and to the right of the filter.
}
\label{conf2FId}
\end{figure}

Alternatively, one could restate Alpert's theorem by saying that $H_{k}\big(\text{conf}(\bullet, 2);\Z\big)$ is a finitely generated FI$_{k+1}$-module over $\Z$ generated in degree $\le 3k$. 
For more on FI$_{d}$-modules and their role in representation stability, see \cite{ramos2017generalized}.

Unfortunately, Alpert's theorem does not generalize to $H_{k}\big(\text{conf}(\bullet, w);\Z\big)$ for $w\ge 3$ as the operation of counting barriers is not well defined for $w\ge 3$, as evinced by the following proposition of Alpert--Manin.

\begin{prop}\label{bad barriers}
(Alpert--Manin \cite[Proposition 8.6]{alpert2021configuration1}) There is a non-trivial class in $H_{4}(\text{conf}(8, 3);\Z)$ that can be written as a sum of products of 4 proper wheels each of which consists of 2 disks, or as a sum of products of 2 proper wheels each consisting of 2 disks and a proper filter consisting of 4 disks, or as a sum of products of 2 proper filters each consisting of 4 disks.
\end{prop}

In short, Proposition \ref{bad barriers} proves that there is a class in $H_{4}(\text{conf}(8, 3);\Z)$ that can be alternatively written as a sum of classes with no integral barriers, as a sum of classes with $1$ integral barrier, or as a sum of classes with $2$ integral barriers.
In general, for all $w\ge 3$ and $n\gg0$, there are relations between classes with different numbers of integral barriers.
It follows that the number of integral barriers in a homology class is not well-defined, and as such one cannot extend Alpert's proof of first-order representation stability to $\text{conf}(\bullet, w)$ for $w\ge 3$.
While integral versions of most the relations in the proof of Theorem \ref{finitepresentation} hold for $H_{*}\big(\text{conf}(\bullet, w);\Z\big)$, Proposition \ref{bad barriers} shows they are not sufficient for a presentation as the integral version of the relation of Proposition \ref{sign of averaged-filter} does not hold, and that relations that do no preserve the number of barriers are needed for a presentation.
We will show that with rational coefficients there is a better notion of a barrier, and that the number of these rational barriers are preserved by the relations in proof of Theorem \ref{finitepresentation}.

\begin{defn}
%
Let $Z\in H_{k}\big(\text{conf}(n, w);\Q\big)$ be a concatenation product of wheels and non-trivial averaged-filters, then a \emph{rational $d$-barrier} in $Z$ is any wheel on $w+1-d$ disks or any averaged-filter (on at least 3 wheels).
\end{defn}

Though we have defined all averaged-filters to be rational $d$-barriers for all $d$, there are some rational $d$-barriers that commute with wheels on $d$ disks, i.e., some averaged-filters do not block the motion of certain small wheels.
For example, if $w=5$ and $d=1$, consider the averaged-filter on $3$ wheels each consisting of $2$ disks.
One can check that in this case the the relation of Lemma \ref{relationlem} degenerates into the commutation relation of the wheel on $1$ disk and this averaged-filter.
While we could be more precise, we include such filters as rational barriers in order to describe a more unified theory.
One can check that this does not change the relevant twisted algebra for representation stability, though it might make our stable range slightly worse.

We write $B(n, k, w, d, m)$ for the subspace of elements in $H_{k}\big(\text{conf}(n, w);\Q\big)$ that can be written as a linear combination of elements with $m$ rational $d$-barriers.

We will show that for $d\le \big\lfloor\frac{w+1}{3}\big\rfloor$ the subspaces $B(n, k, w, d, m)$ and $B(n, k, w, d, m')$ give a direct sum decomposition of $H_{k}\big(\text{conf}(n,w);\Q\big)$.
This will allow us to prove that $\text{conf}(n,w)$ satisfies notions of  $1^{\text{st}}$- through $\big\lfloor\frac{w+1}{3}\big\rfloor^{\text{th}}$-order representation stability.
We begin by proving that the relations in the presentation for $H_{*}\big(\text{conf}(\bullet, w);\Q\big)$ given in the proof of Theorem \ref{finitepresentation} preserve the number of rational $d$-barriers for $d\le\big\lfloor\frac{w+1}{3}\big\rfloor$.
As we will see the upper bound $\big\lfloor\frac{w+1}{3}\big\rfloor$ is a direct consequence of sizes of the wheels in the relation in Lemma \ref{relationlem}.

\begin{prop}\label{relations preserve number of barriers}
For all $w\ge 2$ and $d\le \big\lfloor\frac{w+1}{3}\big\rfloor$, the relations given in the proof of Theorem \ref{finitepresentation} preserve the number of rational $d$-barriers.
\end{prop}

\begin{proof}
The only relation we need to check is that of Lemma \ref{relationlem} as the other relations preserve the size of the wheels and number of wheels and averaged-filters.

The relation of Lemma \ref{relationlem} allows us to rewrite the product of a wheel and averaged-filter as the sum of products of a wheel in the averaged-filter and an averaged-filter on the remaining wheels:
\begin{multline*}
\sum_{k=1}^{m+1} \pm W(i_{k,1}, \dots, i_{k,n_{k}})|AF\big(W(i_{1,1}, \dots, i_{1,n_{1}}),\dots, \widehat{W(i_{k,1}, \dots, i_{k,n_{k}})}, \dots, W(i_{m+1,1}, \dots, i_{m+1,n_{m+1}})\big)\\
=\sum_{k=1}^{m+1} \pm AF\big(W(i_{1,1}, \dots, i_{1,n_{1}}),\dots, \widehat{W(i_{k,1}, \dots, i_{k,n_{k}})}, \dots, W(i_{m+1,1}, \dots, i_{m+1,n_{m+1}})\big)|W(i_{k,1}, \dots, i_{k,n_{k}}).
\end{multline*}

Without loss of generality we may assume that $n_{1}\ge n_{2}\ge \cdots\ge n_{m}\ge n_{m+1}$.
We wish to maximize $n_{1}$ such $AF\big(W(i_{2,1}, \dots, i_{2,n_{2}}), \dots, W(i_{m+1,1}, \dots, i_{m+1,n_{m+1}})\big)$ is non-trivial.
Additionally, we want for some $k$ that $n_{k}<n_{1}$, as then
\[
W(i_{1,1}, \dots, i_{1,n_{1}})|AF\big(W(i_{2,1}, \dots, i_{2,n_{2}}), \dots, W(i_{m+1,1}, \dots, i_{m+1,n_{m+1}})\big)
\]
and
\[
W(i_{k,1}, \dots, i_{k,n_{k}})|AF\big(W(i_{1,1}, \dots, i_{1,n_{1}}), \dots,\widehat{W(i_{k,1}, \dots, i_{k,n_{k}})}, \dots, W(i_{m+1,1}, \dots, i_{m+1,n_{m+1}})\big)
\]
would have different numbers of rational $w-n_{k}+1$ barriers.
Note, by our assumptions all the averaged-filters are non-trivial.

By the assumptions of Lemma \ref{relationlem} we have that
\[
n_{1}+\cdots+n_{m-1}\le w,
\]
and
\[
n_{2}+\cdots+n_{m+1}>w.
\]
Together these imply
\[
w-n_{m}-n_{m+1}<w-n_{1},
\]
or
\[
n_{1}<n_{m}+n_{m+1}.
\]
In order to maximize $n_{1}$, we need $n_{m}+n_{m+1}$ as large as possible.
Without loss of generality we may assume that $n_{k}$ is the first $n_{i}$ such that $n_{i}<n_{1}$.
To maximize $n_{m}+n_{m+1}$ we must have $n_{m+1}=n_{m}=\cdots=n_{k}$.
It follows that
\[
n_{2}+\cdots+n_{k-1}+n_{k}+\cdots+n_{m+1}=(k-2)n_{1}+(m+2-k)n_{m+1}>w.
\]
Since $n_{m+1}<n_{1}$, it follows that $n_{m+1}$ is maximized when $k=2$ and $m=3$, i.e.,
\[
n_{2}+n_{3}+n_{4}=3n_{4}>w.
\]
Since
\[
n_{1}+n_{2}=n_{1}+n_{4}\le w,
\]
and 
\[
n_{1}<n_{3}+n_{4}=2n_{4},
\]
we see that $n_{1}$ is maximized when 
\[
3n_{4}=w+1
\]
and 
\[
n_{1}=2n_{4}-1,
\]
implying 
\[
n_{1}=\frac{2w+2}{3}-1=\frac{2w-1}{3}.
\]
In this case $W(i_{1,1}, \dots, i_{1,n_{1}})$ is a rational $w-\frac{2w-1}{3}+1=\frac{w+4}{3}$ barrier, but the other wheels are not.
Since this is the extremal case we see that the relation of Lemma \ref{relationlem} preserves the number of rational $d$-barriers for $d\le\big\lfloor\frac{w+1}{3}\big\rfloor$, but not necessarily if $d=\big\lfloor\frac{w+4}{3}\big\rfloor$.
\end{proof}

Note that this bound is not sharp.
If $w=4$, directly checking the relation of Lemma \ref{relationlem}, we see that the number of rational $2>\big\lfloor\frac{4+1}{3}\big\rfloor$-barriers is preserved. 

\begin{prop}\label{decompose into d barrier subspaces}
For all $w\ge 2$, $0\le d\le \big\lfloor\frac{w+1}{3}\big\rfloor$, and $k, n\ge 0$
\[
H_{k}\big(\text{conf}(n, w);\Q\big)=\bigoplus_{m} B(n, k, w, d, m).
\]
Moreover, $B(n, k, w, d, m)$ is $S_{n}$-invariant.
\end{prop}

\begin{proof}
Let $Z\in B(n, k,w, d, m)$, i.e., $Z$ is a non-trivial sum of elements with $m$ rational $d$-barriers for $d\le \big\lfloor\frac{w+1}{3}\big\rfloor$.
If we could write $Z$ as sum of elements in $\bigoplus_{i\neq m}B(n, k,w, d, i)$ this would imply that there is a relation that does not preserve the number of rational $d$-barriers.
By Proposition \ref{relations preserve number of barriers} all the relations in the proof of Theorem \ref{finitepresentation} preserve the number of rational $d$-barriers.
This implies that this new relation could not be written in terms of the relations of the proof of Theorem \ref{finitepresentation}, contradicting the fact that they suffice for a presentation.
Thus 
\[
H_{k}\big(\text{conf}(n, w);\Q\big)=\bigoplus_{m} B(n, k, w, d, m).
\]
Since the $S_{n}$-action is accounted for in the relations in the proof of Theorem \ref{finitepresentation}, we see that $B(n, k, w, d, m)$ is $S_{n}$-invariant.
\end{proof}

It follows from Proposition \ref{decompose into d barrier subspaces} that counting the number of rational $d$-barriers is a well-defined operation for $d\le \big\lfloor\frac{w+1}{3}\big\rfloor$.
Additionally, this proposition proves that we can take the basis elements for $H_{k}\big(\text{conf}(n,w);\Q\big)$ from Theorem \ref{AMWthmB''} with $m$ rational $d$-barriers to be a basis for $B(n, k, w, d, m)$, as each basis element has a fixed number of barriers. 
We use this to prove that the ordered configuration space of unit diameter disks in the infinite strip satisfies a notion of first-order representation stability.

\begin{thm}\label{first order stability}
For all $k\ge 0$ and $w\ge 2$, the homology groups $H_{k}\big(\text{conf}(\bullet, w);\Q\big)$ have the structure of an $\mathcal{L}^{w}_{1}(b+1)$-module, where $b=\big\lfloor \frac{k}{w-1}\big\rfloor$, finitely generated in degree $\le2k$ for $w\ge 3$, and degree $\le3k$ for $w=2$.
\end{thm}

\begin{proof}
We begin by proving that $H_{k}\big(\text{conf}(\bullet, w);\Q\big)$ is an $\mathcal{L}^{w}_{1}(b+1)$-module, where $b=\big\lfloor \frac{k}{w-1}\big\rfloor$.
Then we show that this $\mathcal{L}^{w}_{1}(b+1)$-module is finitely generated through a quick counting argument.

Let $X_{i}\in \mathcal{L}^{w}_{1}(b+1)$ denote the fundamental class of the $i^{\text{th}}$ copy of $H_{0}\big(\text{conf}(1,w);\Q\big)$ in $\mathcal{L}^{w}_{1}(b+1)$, and let $Z$ be an element in the basis for $H_{k}\big(\text{conf}(n, w);\Q\big)$ from Theorem \ref{AMWthmB''} such that $Z$ has $m$ rational $1$-barriers.
For $i\le m$, set $X_{i}\otimes Z$ to be element of $H_{k}\big(\text{conf}(n+1, w);\Q\big)$ arising from inserting a wheel on $1$ disk labeled $n+1$ immediately to the left of the $i^{\text{th}}$ rational $1$-barrier of $Z$, and for $i>m$, set $X_{i}\otimes Z$ to be the element of $H_{k}\big(\text{conf}(n+1, w);\Q\big)$ arising from $Z$ by placing a wheel on $1$ disk labeled $n+1$ to the right of $Z$. 
By Proposition \ref{decompose into d barrier subspaces} this is well defined. 

Since $w\ge 2$, Proposition \ref{commutewheels} proves that two wheels each consisting of $1$ disk commute.
Therefore, the multiplication map
\[
\big(\mathcal{L}^{w}_{1}(b+1)\big)_{m}\otimes H_{k}(\text{conf}(n, w);\Q)\to H_{k}(\text{conf}(n+m, w);\Q)
\] 
is $S_{m}\times S_{n}$ invariant, proving that $H_{k}\big(\text{conf}(\bullet, w);\Q\big)$ has the structure of an $\mathcal{L}^{w}_{1}(b+1)$-module. 

To see that as an $\mathcal{L}^{w}_{1}(b+1)$-module $H_{k}\big(\text{conf}(\bullet, w);\Q\big)$ is finitely generated in degree $\le 2k$ for $w\ge 3$ and degree $\le 3k$ for $w=2$, note that there can be at most $b= \big\lfloor\frac{k}{w-1}\big\rfloor$ rational $1$-barriers in a basis element of $H_{k}(\text{conf}(\bullet, w);\Q)$. 
This follows from the fact that the $1$-barriers using the fewest disks are the wheels consisting of $w$ disks, and these wheels have homological degree $w-1$. 
It follows that an element of $H_{k}\big(\text{conf}(\bullet, w);\Q\big)$ can have at most $\big\lfloor\frac{k}{w-1}\big\rfloor$ such $1$-barriers. 

In every element of the basis of Theorem \ref{AMWthmB''}, the only place wheels on $1$ disk not in an averaged-filter are allowed is immediately to the left of a $1$-barrier or at the far right. 
Therefore, an element $Z$ in $H_{k}\big(\text{conf}(n,w);\Q\big)$ that can be written as product of proper wheels and proper averaged-filters on at least $3$ wheels, such that one of the wheels consists of a single disk arises from a non-trivial multiplication of $\big(\mathcal{L}^{w}_{1}(b+1)\big)_{1}$ on an element of $H_{k}\big(\text{conf}(n-1,w);\Q\big)$.

The generators of $H_{k}\big(\text{conf}(\bullet, w);\Q\big)$ that have the worst ratio of number of disks to homological degree are the wheels on $2$ disks for $w\ge 3$, and averaged-filters on $3$ wheels, each consisting of $1$ disk for $w=2$.
Each of these generators have homological dimension $1$.
It follows that if $w\ge 3$ and $n>2k$, then every element of $H_{k}\big(\text{conf}(n, w);\Q\big)$ can be written as a sum of elements each of which has at least $1$ wheel on $1$ disk outside of an averaged-filter.
The same holds for $w=2$ and $n>3k$.
Therefore, 
\[
H^{\mathcal{L}^{w}_{1}(b+1)}_{0}\Big(H_{k}\big(\text{conf}(\bullet, w);\Q\big)\Big)_{n}=0
\]
for all $n>2k$ when $w\ge 3$ and $n>3k$ when $w=2$, bounding the generation degree of $H_{k}\big(\text{conf}(\bullet, w);\Q\big)$ as an $\mathcal{L}^{w}_{1}(b+1)$-module.
\end{proof}

It follows that Theorem \ref{first order stability} proves that the ordered configuration space of unit-diameter disks in the infinite strip of width $w$ is first-order representation stable.
Additionally, the proof shows that $b=\big\lfloor\frac{k}{w-1}\big\rfloor$ is a direct consequence of the fact that the most efficient rational $1$-barriers are the wheels on $w$ disks, and these wheels have homological degree $w-1$.

Theorem \ref{first order stability} is equivalent to the fact that for $k\ge 0$ and $w\ge 2$, the homology groups $H_{k}\big(\text{conf}(\bullet, w);\Q\big)$ have the structure of an FI$_{b+1}$-module, where $b=\big\lfloor \frac{k}{w-1}\big\rfloor$, finitely generated in degree $2k$ for $w\ge 3$, and degree $3k$ for $w=2$, i.e., Theorem \ref{first order stab intro}.
Thus, we have generalized Theorem 6.1 of Alpert in the case of rational homology and all widths $w\ge 2$ \cite[Theorem 6.1]{alpert2021configuration1}. 
By Ramos's results on FI$_{d}$-modules, it follows that one can use Theorem \ref{first order stability} to get upper bounds for Betti numbers of $H_{k}\big(\text{conf}(n, w);\Q\big)$ for $n>2k$ if $w\ge 3$ and $n>3k$ if $w=2$ if one knows the Betti numbers for $H_{k}\big(\text{conf}(m, w);\Q\big)$ for all  $m\le2k$ if $w\ge 3$ and $m\le3k$ if $w=2$.
Additionally, one gets upper bounds for the multiplicities of irreducible $S_{n}$-representations in a decomposition of $H_{k}\big(\text{conf}(n, w);\Q\big)$ for $n>2k$ if $w\ge 3$ and $n>3k$ if $w=2$ if one knows how to decompose $H_{k}\big(\text{conf}(m, w);\Q\big)$ into irreducible $S_{m}$-representations for all $m\le2k$ if $w\ge 3$ and $m\le3k$ if $w=2$.

In the next section, we will show that the ordered configuration space of unit-diameter disks in the infinite strip of width $w$ exhibits notions of higher order representation stability.

\subsection{Higher Order Stability}
In this section, we prove that the ordered configuration space of open unit-diameter disks in the infinite strip of width $w$ exhibits reasonable notions of higher order representation stability. 
In the previous section, we proved that $H_{k}\big(\text{conf}(\bullet, w);\Q\big)$ is a finitely generated module over the twisted commutative algebra $\mathcal{L}^{w}_{1}(b+1)$ where $b=\big\lfloor\frac{k}{w-1}\big\rfloor$. 
Note as $k$ varies, $\mathcal{L}^{w}_{1}(b+1)$ varies. 
This is different from the case of first-order representation stability for configuration spaces of points in connected noncompact manifolds, where the twisted algebra of first-order stability is always isomorphic to $\mathcal{L}_{1}$, regardless of the value of $k$. 
In their statement of second-order representation stability for the ordered configuration space of points in a connected noncompact manifold, Miller--Wilson used the zeroth $\mathcal{L}_{1}$-homology of the homology of configuration space. 
Unfortunately, we cannot do the same for finite $b$.
However, by taking a limit we are able to remedy this without losing any information.

Recall that Proposition \ref{FIdisFId+1} proves that any module $M$ over $\mathcal{L}^{w}_{1}(b)$ can be made a module over $\mathcal{L}^{w}_{1}(b+1)$ by setting $X_{b+1}\otimes m=X_{b}\otimes m$ for all $m\in M$, where $X_{j}$ is the fundamental class of the $j^{\text{th}}$ copy of $H_{0}\big(\text{conf}(1,w);\Q\big)$ in $\mathcal{L}^{w}_{1}(b)$.
By setting $X_{i}\otimes m=X_{b+1}\otimes m$ for all $i>b+1$ and all $m\in M$ and where $b=\big\lfloor\frac{k}{w-1}\big\rfloor$, it follows that for all $k$ and $w$ we have that $H_{k}\big(\text{conf}(\bullet, w);\Q\big)$ is an $\mathcal{L}^{w}_{1}(\infty)$-module, and
\[
H^{\mathcal{L}^{w}_{1}(\infty)}_{0}\Big(H_{k}\big(\text{conf}(\bullet, w);\Q\big)\Big)=H^{\mathcal{L}^{w}_{1}(b+1)}_{0}\Big(H_{k}\big(\text{conf}(\bullet, w);\Q\big)\Big).
\]
As a corollary to Theorem \ref{first order stability}, for all $k\ge 0$ and $w\ge 2$, the homology groups $H_{k}\big(\text{conf}(\bullet, w);\Q\big)$ have the structure of an $\mathcal{L}^{w}_{1}(\infty)$-module, finitely generated in degree $2k$ for $w\ge 3$ and degree $3k$ for $w=2$.

It follows that the twisted algebra $H_{*}\big(\text{conf}(\bullet, w);\Q\big)$ is an $\mathcal{L}^{w}_{1}(\infty)$-module, though it is not finitely generated as we need generators in each degree $k$. 
Still, we can make sense of the the zeroth $\mathcal{L}^{w}_{1}(\infty)$-homology of $H_{*}(\text{conf}(\bullet, w);\Q)$.
Doing so will allow us prove higher order representation stability results.

\begin{prop}\label{zeroth Winf(1) homology presentation}
For all $w\ge 2$, 
\[
H^{\mathcal{L}^{w}_{1}(\infty)}_{0}\Big(H_{*}\big(\text{conf}(\bullet, w);\Q\big)\Big)
\]
is a finitely presented twisted algebra.
Moreover, there is a presentation for $H^{\mathcal{L}^{w}_{1}(\infty)}_{0}\Big(H_{*}\big(\text{conf}(\bullet, w);\Q\big)\Big)$ with the same generators and relations as $H_{*}\big(\text{conf}(\bullet, w);\Q\big)$ has in the proof of Theorem \ref{finitepresentation}, along the extra relation that wheels on $1$ disk not in an averaged-filter are trivial. 
\end{prop}

\begin{proof}
First, note that $H^{\mathcal{L}^{w}_{1}(\infty)}_{0}\Big(H_{*}\big(\text{conf}(\bullet, w);\Q\big)\Big)$ is a twisted algebra as it is an associative graded unital $\Q$-algebra with an $S_{n}$-action on $H^{\mathcal{L}^{w}_{1}(\infty)}_{0}\Big(H_{*}\big(\text{conf}(n, w);\Q\big)\Big)$ for all $n$, and multiplication of an element in $H^{\mathcal{L}^{w}_{1}(\infty)}_{0}\Big(H_{k}\big(\text{conf}(n, w);\Q\big)\Big)$ by an element in $H^{\mathcal{L}^{w}_{1}(\infty)}_{0}\Big(H_{j}\big(\text{conf}(m, w);\Q\big)\Big)$ is $S_{n}\times S_{m}$ invariant.

Let $X_{i}\in \big(\mathcal{L}^{w}_{1}(\infty)\big)_{1}$ be the fundamental class of the $i^{\text{th}}$ copy of $H_{0}\big(\text{conf}(1, w);\Q\big)$ in $\mathcal{L}^{w}_{1}(\infty)$. 
Multiplication by $X_{i}$ corresponds inserting a wheel on $1$ disk immediately to the left of the $i^{\text{th}}$ $1$-barrier for small $i$ and at the far right for large $i$. 
By applications of the (anti)-commutation relations of Proposition \ref{commutewheels} it follows that multiplication by a wheel on $1$ disk in $H_{*}\big(\text{conf}(\bullet, w);\Q\big)$ corresponds to multiplication by $X_{i}$ for some $i$.
Therefore, $H^{\mathcal{L}^{w}_{1}(\infty)}_{0}\Big(H_{*}\big(\text{conf}(\bullet, w);\Q\big)\Big)$ is a twisted algebra with the same generators and relations as $H_{*}\big(\text{conf}(\bullet, w);\Q\big)$ has in Theorem \ref{finitepresentation} along with the extra relation that wheels on $1$ disk are trivial.
\end{proof}

We prove two results that will help us prove that the ordered configuration space of unit diameter disks in the infinite strip of width $w$ exhibits notions of $2^{\text{nd}}-$ through $\big\lfloor\frac{w+1}{3}\big\rfloor^{\text{th}}$-order representation stability.

\begin{prop}\label{L infinity module}
For $0\le d\le \big\lfloor\frac{w+1}{3}\big\rfloor$, the homology groups $H_{*}\big(\text{conf}(\bullet, w);\Q\big)$ have the structure of an $\mathcal{L}^{w}_{d}(\infty)$-module.
\end{prop}

\begin{proof}
For fixed $d$, we can use Proposition \ref{decompose into d barrier subspaces} to decompose $H_{k}\big(\text{conf}(n, w);\Q\big)$:
\[
H_{k}\big(\text{conf}(n, w);\Q\big)=\bigoplus_{m} B(n, k, w, d, m).
\]

By Proposition \ref{properwheels} the $S_{d}$-span of the proper wheel $W(d, d-1, \dots, 1)$ generates $H_{d-1}\big(\text{conf}(d, w);\Q\big)$. 
Let $X_{i}\in \mathcal{L}^{w}_{d}(\infty)$ correspond to the wheel $W(d, d-1, \dots, 1)$ in the $i^{\text{th}}$-copy of $H_{d-1}\big(\text{conf}(d, w);\Q\big)$ in $\mathcal{L}^{w}_{d}(\infty)$, and let multiplication of $Z\in B(n, k, w, d, m)\subseteq H_{k}\big(\text{conf}(n, w);\Q\big)$ by $X_{i}$ correspond to placing a wheel $W(n+d, n+d-1, \dots, n+1)$ to the immediate left of the $i^{\text{th}}$ rational $d$-barrier of $z$ if $i\le m$, and to the far right if $i>m$. 
Since wheels on $d\le \big\lfloor\frac{w+1}{3}\big\rfloor$ disks (anti)-commute, thsis yields an $S_{m}\times S_{n}$-equivariant map
\[
\mathcal{L}^{w}_{m}(\infty)_{m}\otimes H_{k}\big(\text{conf}(n, w);\Q\big)\to H_{k+m-1}\big(\text{conf}(n+m, w);\Q\big),
\]
and this hold for all $n, k\ge 0$. 
Since the $X_{i}$ generate $\mathcal{L}^{w}_{d}(\infty)$ it follows that this defines an $\mathcal{L}^{w}_{d}(\infty)$-module structure on $H_{*}\big(\text{conf}(\bullet, w);\Q\big)$.
\end{proof}

Next, we shows that this $\mathcal{L}^{w}_{d}(\infty)$-module structure extends to a series of quotients of $H_{*}\big(\text{conf}(\bullet, w);\Q\big)$.
We will use this fact to give higher order representation stability structures.

\begin{lem}\label{higherorderpresentation}
For $0\le d\le \big\lfloor\frac{w+1}{3}\big\rfloor$, the quotient 
\[
H^{\mathcal{L}^{w}_{d}(\infty)}_{0}\Bigg(\cdots \bigg(H^{\mathcal{L}^{w}_{1}(\infty)}_{0}\Big(H_{*}\big(\text{conf}(\bullet, w);\Q\big)\Big)\bigg)\cdots\Bigg)
\]
is a finitely presented twisted algebra.
Moreover, it has a finite presentation consisting of the generators and the relations given for 
\[
H_{*}\big(\text{conf}(\bullet, w);\Q\big)
\]
in the proof of Theorem \ref{finitepresentation} along with the relations that any proper wheel on $d$ disks or fewer that is not in an averaged-filter is trivial. 
\end{lem}

\begin{proof}
The case $d=0$ is Theorem \ref{finitepresentation}, and the case $d=1$ is Proposition \ref{zeroth Winf(1) homology presentation}.
Assume that we have proved the lemma through the case for $d\le j$, for some $j\le \big\lfloor\frac{w+1}{3}\big\rfloor-1$. 
We will show that the Lemma holds in the case $d=j+1$.

Note that 
\[
H^{\mathcal{L}^{w}_{j+1}(\infty)}_{0}\Bigg(\cdots \bigg(H^{\mathcal{L}^{w}_{1}(\infty)}_{0}\Big(H_{*}\big(\text{conf}(\bullet, w);\Q\big)\Big)\bigg)\cdots\Bigg)
\] 
satisfies the definition of a twisted algebra for the same reasons 
\[
H^{\mathcal{L}^{w}_{1}(\infty)}_{0}\Big(H_{*}\big(\text{conf}(\bullet, w);\Q\big)\Big)
\]
does.

Since we have assumed that the Lemma holds in the case $d=j$, we have that
\[
H^{\mathcal{L}^{w}_{j}(\infty)}_{0}\Bigg(\cdots \bigg(H^{\mathcal{L}^{w}_{1}(\infty)}_{0}\Big(H_{*}\big(\text{conf}(\bullet, w);\Q\big)\Big)\bigg)\cdots\Bigg)
\]
is a finitely presented twisted algebra, and it has a presentation consisting of the same generators and relations as the presentation for $H_{*}\big(\text{conf}(\bullet, w);\Q\big)$ found in the proof of Theorem \ref{finitepresentation} along with the relations that all proper wheels on at most $j$ disks not inside an averaged-filter are trivial. 
Note, these extra relations preserve the number of rational $j+1$-barriers as only wheels on more than $n-j>\big\lfloor\frac{w+1}{3}\big\rfloor$ disks can act as rational $j+1$-barriers. 
Thus, the image of $B(n, k, w, j+1, m)\subset H_{k}\big(\text{conf}(n, w);\Q\big)$ under the quotient is well-defined for all $n$, $k$, and $m$, and multiplication by $X_{i}\in \mathcal{L}^{w}_{j+1}(\infty)$, where $x$ is a proper wheel on $j+1$ disks in the $i^{\text{th}}$-copy of $H_{j}\big(\text{conf}(j+1, w);\Q\big)$ is well-defined. 

Every element of $H^{\mathcal{L}^{w}_{j}(\infty)}_{0}\Bigg(\cdots \bigg(H^{\mathcal{L}^{w}_{1}(\infty)}_{0}\Big(H_{*}\big(\text{conf}(\bullet, w);\Q\big)\Big)\bigg)\cdots\Bigg)$ can be written as sum of products of generators such that any wheel on $j+1$ disks is  part of a sequence of consecutive wheels on $j+1$ disks immediately to the right of a rational $j+1$-barrier or at the far right of the product. 
This follows since the relations given in the presentation in the proof of Theorem \ref{finitepresentation} allow us to (anti)-commute all smaller wheels as $j+1\le \big\lfloor\frac{w+1}{3}\big\rfloor$ and all smaller wheels are trivial by the induction hypothesis. 
Therefore, multiplication by $X_{i}\in\mathcal{L}^{w}_{j+1}(\infty)$ corresponds to multiplication in $H^{\mathcal{L}^{w}_{j}(\infty)}_{0}\Bigg(\cdots \bigg(H^{\mathcal{L}^{w}_{1}(\infty)}_{0}\Big(H_{*}\big(\text{conf}(\bullet, w);\Q\big)\Big)\bigg)\cdots\Bigg)$ by a wheel on $j+1$ disks, and the converse holds as well.

Thus, the quotient
\[
H^{\mathcal{L}^{w}_{j+1}(\infty)}_{0}\Bigg(\cdots \bigg(H^{\mathcal{L}^{w}_{1}(\infty)}_{0}\Big(H_{*}\big(\text{conf}(\bullet, w);\Q\big)\Big)\bigg)\cdots\Bigg)
\]
has a finite presentation consisting of the same generators and relations we gave for $H_{*}\big(\text{conf}(\bullet, w);\Q\big)$ in the proof of Theorem \ref{finitepresentation}, along with the extra relations that all wheels on no more than $j+1$ disks are trivial.
\end{proof}

We use Lemma \ref{higherorderpresentation} to prove that the ordered configuration space of unit diameter disks in the infinite strip of width $w$ exhibits notions of $2^{\text{nd}}$- through $\big\lfloor\frac{w+1}{3}\big\rfloor^{\text{th}}$-order representation stability.
First, we prove an intermediate proposition.

\begin{prop}\label{finitely generated infinity module}
Let
\[
\mathcal{W}^{d}_{i}(\bullet):=H^{\mathcal{L}^{w}_{d-1}(\infty)}_{0}\Bigg(\cdots \bigg(H^{\mathcal{L}^{w}_{1}(\infty)}_{0}\Big(H_{\frac{(d-1)|\bullet|+i}{d}}\big(\text{conf}(\bullet, w);\Q\big)\Big)\bigg)\cdots\Bigg).
\]
For $1\le d\le\big\lfloor\frac{w+1}{3}\big\rfloor$, the sequence $\mathcal{W}^{d}_{i}(\bullet)$ has the structure of a finitely generated $\mathcal{L}^{w}_{d}(\infty)$-module generated in degree $(d+1)i$ for $2d+1\le w$, and degree $(d+1)i+d$ for $2d=w$.
\end{prop}

We start with $H_{\frac{(d-1)|\bullet|+i}{d}}\big(\text{conf}(\bullet, w);\Q\big)$ as we wish to increase homological degree by $d-1$ while adding $d$ new disks as adding a wheel on $d$ disks does exactly that.
The shifting by $i$ ensures that every homology group appears in some $\mathcal{W}^{d}_{i}(\bullet)$.

\begin{proof}
By Lemma \ref{higherorderpresentation}, 
\[
H^{\mathcal{L}^{w}_{d-1}(\infty)}_{0}\Bigg(\cdots \bigg(H^{\mathcal{L}^{w}_{1}(\infty)}_{0}\Big(H_{*}\big(\text{conf}(\bullet, w);\Q\big)\Big)\bigg)\cdots\Bigg)
\]
is an $\mathcal{L}^{w}_{d}(\infty)$-module. 
Since multiplication by a generator of $\mathcal{L}^{w}_{d}(\infty)$ increases homological degree by $d-1$ and the number of disks by $d$, it follows that
\[
H^{\mathcal{L}^{w}_{d-1}(\infty)}_{0}\Bigg(\cdots \bigg(H^{\mathcal{L}^{w}_{1}(\infty)}_{0}\Big(H_{*}\big(\text{conf}(\bullet, w);\Q\big)\Big)\bigg)\cdots\Bigg)
\]
decomposes as a direct sum of $\mathcal{L}^{w}_{d}(\infty)$-modules of the form $\mathcal{W}^{d}_{i}(\bullet)$, i.e., 
\begin{multline*}
H^{\mathcal{L}^{w}_{d-1}(\infty)}_{0}\Bigg(\cdots \bigg(H^{\mathcal{L}^{w}_{1}(\infty)}_{0}\Big(H_{*}\big(\text{conf}(\bullet, w);\Q\big)\Big)\bigg)\cdots\Bigg)\\
=\bigoplus_{i=0}^{\infty}H^{\mathcal{L}^{w}_{d-1}(\infty)}_{0}\Bigg(\cdots \bigg(H^{\mathcal{L}^{w}_{1}(\infty)}_{0}\Big(H_{\frac{(d-1)|\bullet|+i}{d}}\big(\text{conf}(\bullet, w);\Q\big)\Big)\bigg)\cdots\Bigg).
\end{multline*}

We count the maximal number of rational $d$-barriers in $\mathcal{W}^{d}_{i}(\bullet)$ as this will give us a bound on the generation degree. 
The maximal number of rational $d$-barriers in $\mathcal{W}^{d}_{i}(\bullet)$ is equal to the largest non-trivial term of the sequence $H^{\mathcal{L}^{w}_{d}(\infty)}_{0}\big(\mathcal{W}^{d}_{i}(\bullet)\big)$.
By Lemma \ref{higherorderpresentation}, every element in $H^{\mathcal{L}^{w}_{d}(\infty)}_{0}\big(\mathcal{W}^{d}_{i}(\bullet)\big)$ can be written as a sum of products of wheels on at least $d+1$ disks, and averaged-filters.

When $2d+1\le w$, the least efficient of these generators in terms of homological degree to the number of disks are the wheels on $d+1$ disks. 
When $|\bullet|=(d+1)i$, the quotient $H^{\mathcal{L}^{w}_{d-1}(\infty)}_{0}\big(\mathcal{W}^{d}_{i}(\bullet)\big)\neq0$, as the product of $i$ wheels on $d+1$ disks is non-trivial.
For $|\bullet|>(d+1)i$, any element in $H_{\frac{(d)|\bullet|+i}{d}}\big(\text{conf}(\bullet, w);\Q\big)$ can be written as a sum elements, each of which includes a wheel on fewer than $d+1$ disks not inside an averaged-filter. 
Such products are trivial in $H^{\mathcal{L}^{w}_{d}(\infty)}_{0}\big(\mathcal{W}^{d}_{i}(\bullet)\big)$ by Lemma \ref{higherorderpresentation}. 
Therefore, $\mathcal{W}^{d}_{i}(\bullet)$ is generated in degree $\le (d+1)i$ as an $\mathcal{L}^{w}_{d}(\infty)$-module. 
Since each term of $\mathcal{W}^{d}_{i}(\bullet)$ is finite dimensional, it follows that $\mathcal{W}^{d}_{i}(\bullet)$ is finitely generated as an $\mathcal{L}^{w}_{d}(\infty)$-module.

If $2d=w$, the least efficient of the generators in terms of homological degree to the number of disks are the proper averaged-filters on $w+1$ disks. 
When $|\bullet|=(d+1)i+d$, the quotient $H^{\mathcal{L}^{w}_{d}(\infty)}_{0}\big(\mathcal{W}^{d}_{i}(\bullet)\big)\neq0$, as the product of $i$ averaged-filters each on $w+1$ disks is non-trivial. 
For $|\bullet|>(d+1)i+d$, it follows that any product of elements in $H_{\frac{(d-1)|\bullet|+i}{d}}\big(\text{conf}(\bullet, w);\Q\big)$ can be written as a sum of elements each of which includes a wheel on fewer than $d+1$ disks. 
Such products are trivial in $H^{\mathcal{L}^{w}_{d}(\infty)}_{0}\big(\mathcal{W}^{d}_{i}(\bullet)\big)$ by Lemma \ref{higherorderpresentation}. 
It follows that $\mathcal{W}^{d}_{i}(\bullet)$ is generated in degree $\le (d+1)i+d$ as an $\mathcal{L}^{w}_{d}(\infty)$-module. 
Since each term of $\mathcal{W}^{d}_{i}(\bullet)$ is finite dimensional, $\mathcal{W}^{d}_{i}(\bullet)$ is finitely generated as an $\mathcal{L}^{w}_{d}(\infty)$-module.
\end{proof}

We use Proposition \ref{finitely generated infinity module} to prove the following theorem which shows that the ordered configuration space of open unit diameter disks in the infinite strip of width $w$ exhibits reasonable notions of $1^{\text{st}}-$ through $\big\lfloor\frac{w+1}{3}\big\rfloor^{\text{th}}$-order representation stability.

\begin{thm}\label{higherorderstability}
Let
\[
\mathcal{W}^{d}_{i}(\bullet):=H^{\mathcal{L}_{d-1}(\infty)}_{0}\Bigg(\cdots \bigg(H^{\mathcal{L}_{1}(\infty)}_{0}\Big(H_{\frac{(d-1)|\bullet|+i}{d}}\big(\text{conf}(\bullet, w);\Q\big)\Big)\bigg)\cdots\Bigg).
\]
For $1\le d\le\big\lfloor\frac{w+1}{3}\big\rfloor$,
\[
\mathcal{W}^{d}_{i}(\bullet)
\]
is a finitely generated $\mathcal{L}_{d}(b+1)$-module, where $b=\big\lfloor\frac{di}{w-d}\big\rfloor$, generated in degree $(d+1)i$ for $ 2d+1\le w$, and degree $(d+1)i+d$ for $2d=w$.
\end{thm}

\begin{proof}
By Proposition \ref{finitely generated infinity module}, for $1\le d\le \big\lfloor\frac{w+1}{3}\big\rfloor$,
\[
\mathcal{W}^{d}_{i}(\bullet):=H^{\mathcal{L}^{w}_{d-1}(\infty)}_{0}\Bigg(\cdots \bigg(H^{\mathcal{L}^{w}_{1}(\infty)}_{0}\Big(H_{\frac{(d-1)|\bullet|+i}{d}}\big(\text{conf}(\bullet, w);\Q\big)\Big)\bigg)\cdots\Bigg)
\]
is a finitely generated $\mathcal{L}^{w}_{d}(\infty)$-module generated in degree $(d+1)i$ for $w\ge 2d+1$, and degree $(d+1)i+d$ for $w=2d$.

We count the maximum number of barriers in an element in $\mathcal{W}^{d}_{i}(\bullet)$ as this will allow us to replace  $\mathcal{L}^{w}_{d}(\infty)$ with the finitely generated twisted algebra $\mathcal{L}^{w}_{d}(b+1)$ without loss of information.

For $2d+1\le w$, the module $\mathcal{W}^{d}_{i}(\bullet)$ is generated in degree at most $(d+1)i$. 
This occurs in homological degree 
\[
\frac{(d-1)(d+1)i+i}{d}=di.
\]
By Lemma \ref{higherorderpresentation}, every element in $H^{\mathcal{L}^{w}_{d}(\infty)}_{0}\big(\mathcal{W}^{d}_{i}(\bullet)\big)$ can be written as a product of wheels on at least $d+1$ disks, and averaged-filters. 
The smallest of these generators that are also rational $d$-barriers are the wheels on $w+1-d$ disks, which have homological degree $w-d$. 
We can have at most $\big\lfloor\frac{di}{w-d}\big\rfloor$ rational $d$-barriers in an element of $\mathcal{W}^{d}_{i}(\bullet)$.
Since multiplication is well defined by Proposition \ref{decompose into d barrier subspaces}, we only need the first $b+1:=\big\lfloor\frac{di}{w-d}\big\rfloor+1$ of the copies of $H_{d-1}\big(\text{conf}(d, w);\Q\big)$ in $\mathcal{L}^{w}_{d}(\infty)$, as multiplication by the other terms is equal to multiplication by the $(b+1)^{\text{th}}$ copy of $H_{d-1}\big(\text{conf}(d, w);\Q\big)$. 
Therefore, we can restrict $\mathcal{L}^{w}_{d}(\infty)$ to $\mathcal{L}^{w}_{d}(b+1)$, and not lose information on the structure of $\mathcal{W}^{d}_{i}(\bullet)$ as an $\mathcal{L}^{w}_{d}(\infty)$-module. It follows from Proposition \ref{finitely generated infinity module}, that for $w\ge 2d+1$, $\mathcal{W}^{d}_{i}(\bullet)$ is a finitely generated $\mathcal{L}^{w}_{d}(b+1)$-module generated in degree $(d+1)i$.

For $2d=w$, the module $\mathcal{W}^{d}_{i}(\bullet)$ is generated in degree at most $(d+1)i+d$. This occurs in homological degree 
\[
\frac{(d-1)\big((d+1)i+d\big)+i}{d}=\frac{d^{2}i-i+d^{2}-d+i}{d}=di+d-1.
\]
By Lemma \ref{higherorderpresentation}, every element in $H^{\mathcal{L}^{w}_{d}(\infty)}_{0}\big(\mathcal{W}^{d}_{i}(\bullet)\big)$ can be written as a product of wheels on at least $d+1$ disks, and averaged-filters. 
The smallest of these generators that are also rational $d$-barriers in terms of the homological degree are the wheels on $w+1-d$ disks, which have homological degree $w-d$.
Therefore, since $2d=w$, we can have at most 
\[
\bigg\lfloor\frac{di+d-1}{w-d}\bigg\rfloor=\bigg\lfloor\frac{di+d-1}{d}\bigg\rfloor=\bigg \lfloor \frac{di}{d}+\frac{d-i}{d}\bigg\rfloor=\bigg \lfloor \frac{di}{w-d}+\frac{d-i}{d}\bigg\rfloor=\bigg\lfloor\frac{di}{w-d}\bigg\rfloor
\]
of these rational $d$-barriers. 
Since multiplication is well defined by Proposition \ref{decompose into d barrier subspaces}, we only need the first $b+1:=\big\lfloor\frac{di}{w-d}\big\rfloor+1$ of the copies of $H_{d-1}\big(\text{conf}(d, w);\Q\big)$ in $\mathcal{L}^{w}_{d}(\infty)$, as multiplication by the other generators is equal to multiplication by the generator of the $(b+1)^{\text{th}}$ copy of $H_{d-1}\big(\text{conf}(d, w);\Q\big)$. 
Therefore, we can restrict $\mathcal{L}^{w}_{d}(\infty)$ to $\mathcal{L}^{w}_{d}(b+1)$ and not lose information on the structure of $\mathcal{W}^{d}_{i}(\bullet)$ as an $\mathcal{L}^{w}_{d}(\infty)$-module. It follows from Proposition \ref{finitely generated infinity module}, that for $w\ge 2d+1$, the $\mathcal{L}^{w}_{d}(b+1)$-module $\mathcal{W}^{d}_{i}(\bullet)$ is finitely generated in degree $(d+1)i$.
\end{proof}

Our proof of Theorem \ref{higherorderstability} shows that for $i$ and $|\bullet|$ finite, there exist integer valued functions $b_{m}(i, |\bullet|)$ such that one can replace the $H^{\mathcal{L}^{w}_{m}(\infty)}_{0}$ in the definition of $\mathcal{W}^{d}_{i}(\bullet)$ with $H^{\mathcal{L}^{w}_{m}\big(b_{m}(i, |\bullet|)+1\big)}_{0}$, as all but the first $b_{m}(i, |\bullet|)$ of the copies of $H_{(m-1)}\big(\text{conf}(m, w);\Q\big)$ in $\mathcal{L}^{w}_{m}(\infty)$ will prove to be redundant. 
Thus, we have shown that the ordered configuration space of unit diameter disks in the infinite strip of width $w$ exhibits notions of $1^{\text{st}}$ through $\big\lfloor\frac{w+1}{3}\big\rfloor^{\text{th}}$-order representation stability.

In the spirit of \cite{church2014fi, church2015fi, ramos2017generalized, miller2019higher}, Theorem \ref{higherorderstability} suggests that if we have a decomposition of $H_{d-1}\big(\text{conf}(d, w);\Q\big)$ into irreducible representations and we know how to decompose $\mathcal{W}^{d}_{i}(n)$ into irreducible $S_{m}$-representations for all $m\le (d+1)i$ (or $m\le(d+1)i+d$ for $w=2$), then we have upper bounds on the multiplicities of the irreducible $S_{n}$-representations in a decomposition of $\mathcal{W}^{d}_{i}(n)$ for $n>(d+1)i$ (or $n>(d+1)i+d$ for $w=2$).

\bibliographystyle{amsalpha}
\bibliography{HigherOrderStabilityDisksOnAStrip}
\end{document}